\documentclass[11pt,  reqno]{amsart}

\usepackage{amsmath,amssymb,amscd,amsthm,amsxtra, esint}

\usepackage{forest}

\usepackage[implicit=true]{hyperref}

\allowdisplaybreaks[2]

\definecolor{gr}{rgb}   {0.,   0.69,   0.23 }
\definecolor{bl}{rgb}   {0.,   0.5,   1. }
\definecolor{mg}{rgb}   {0.85,  0.,    0.85}
\definecolor{yl}{rgb}   {0.8,  0.7,   0.}
\definecolor{or}{rgb}  {0.7,0.2,0.2}

\newtheorem{theorem}{Theorem} [section]

\newtheorem{lemma}[theorem]{Lemma}
\newtheorem{proposition}[theorem]{Proposition}
\newtheorem{remark}[theorem]{Remark}

\newtheorem{definition}[theorem]{Definition}
\newtheorem{corollary}[theorem]{Corollary}

\DeclareMathOperator*{\intt}{\int}

\DeclareMathOperator*{\supp}{supp}

\newcommand{\I}{\hspace{0.5mm}\textup{I}} 
\newcommand{\II}{\textup{I \hspace{-2.8mm} I} }

\newcommand{\noi}{\noindent}
\newcommand{\Z}{\mathbb{Z}}
\newcommand{\R}{\mathbb{R}}

\newcommand{\T}{\mathbb{T}}

\newcommand{\N}{\mathcal{N}}

\newcommand{\NB}{\mathbb{N}}

\let\Re=\undefined\DeclareMathOperator*{\Re}{Re}

\let\P= \undefined
\newcommand{\P}{\mathbf{P}}

\renewcommand{\u}{{\textup {\bf  u}}}
\renewcommand{\v}{{\textup {\bf v}}}
\newcommand{\n}{{\textup {\bf n}}}

\newcommand{\F}{\mathcal{F}}
\newcommand{\NN}{\mathcal{N}}
\newcommand{\RR}{\mathcal{R}}
\newcommand{\TT}{\mathcal{T}}

\newcommand{\Nf}{\mathfrak{N}}

\newcommand{\al}{\alpha}
\newcommand{\be}{\beta}
\newcommand{\dl}{\delta}

\newcommand{\eps}{\varepsilon}

\newcommand{\g}{\gamma}
\newcommand{\G}{\Gamma}
\newcommand{\ld}{\lambda}
\newcommand{\Ld}{\Lambda}
\newcommand{\s}{\sigma}
\newcommand{\ft}{\widehat}
\newcommand{\wt}{\widetilde}
\newcommand{\cj}{\overline}
\newcommand{\dx}{\partial_x}
\newcommand{\dt}{\partial_t}
\newcommand{\dd}{\partial}

\newcommand{\jb}[1]
{\langle #1 \rangle}

\renewcommand{\l}{\ell}

\newcommand{\les}{\lesssim}
\newcommand{\ges}{\gtrsim}

\newcommand{\ind}{\mathbf 1}

\newcommand{\bn}{{\bf n}}

\renewcommand{\o}{\omega}
\renewcommand{\O}{\Omega}

\numberwithin{equation}{section}
\numberwithin{theorem}{section}

\newtheorem*{ackno}{Acknowledgements}

\usepackage{tikz}

\usetikzlibrary{trees}

\usetikzlibrary{shapes.misc}
\usetikzlibrary{shapes.symbols}
\usetikzlibrary{shapes.geometric}
\tikzset{
	dot/.style={circle,fill=black,draw=black,inner sep=1pt,minimum size=0.5mm},
	>=stealth,
	}
\tikzset{
	ddot/.style={circle,fill=white,draw=black,inner sep=2pt,minimum size=0.8mm},
	>=stealth,
	}

\tikzset{decision/.style={ 
        draw,
        diamond,
        aspect=1.5
    }}

\tikzset{dia2/.style
={diamond,fill=white,draw=black,inner sep=0pt,minimum size=1mm},
	>=stealth,
	}

\tikzset{dia/.style
={star,fill=black,draw=black,inner sep=0pt,minimum size=1mm},
	>=stealth,
	}

\makeatletter
\def\DeclareSymbol#1#2#3{\expandafter\gdef\csname MH@symb@#1\endcsname{\tikz[baseline=#2,scale=0.15]{#3}}}
\def\<#1>{\csname MH@symb@#1\endcsname}
\makeatother

\DeclareSymbol{X}{-2.4}{\node[dot] {};}
\DeclareSymbol{1}{0}{\draw[white] (-.4,0) -- (.4,0); \draw (0,0)  -- (0,1.2) node[dot] {};}
\DeclareSymbol{2}{0}{\draw (-0.5,1.2) node[dot] {} -- (0,0) -- (0.5,1.2) node[dot] {};}

\DeclareSymbol{T0}{-2.7}
 { \draw (0,0) node[ddot]{};}

\DeclareSymbol{T1}{0}
 {  
 \draw (0,0) node[dot]{} -- (0,4) node[ddot] {}; 
 \draw (-3,0) node[dot] {} -- (0,4)node[ddot] {} -- (3,0) node[dot] {};
\draw[line width=1pt] (0, 4)node[ddot, label=above:$r_1$]{} 
.. controls(3,6) and (5,6) ..
(8, 4)node[ddot, label=above:$r_2$]{}
(4, 12) node[label ] {$\underline{j = 1}$};
 }

\DeclareSymbol{T2}{0}
 {\draw (0,0) node[dot]{} -- (0,4) node[ddot] {}; 
 \draw (-3,0) node[dot] {} -- (0,4)node[ddot] {} -- (3,0) node[dot] {};
 \draw (0,-4) node[dot]{} -- (0,0) node[dot] {}; 
 \draw (-3,-4) node[dot] {} -- (0,0)node[dot] {} -- (3,-4) node[dot] {};
\draw[line width=1pt] (0, 4)node[ddot, label=above:$r_1$]{} 
.. controls(3,6) and (5,6) ..
(8, 4)node[ddot, label=above:$r_2$]{}
(4, 12) node[label ] {$\underline{j = 2}$};
 }

\DeclareSymbol{T31}{0}
 {\draw (0,0) node[dot]{} -- (0,4) node[ddot] {}; 
 \draw (-3,0) node[dot] {} -- (0,4)node[ddot] {} -- (3,0) node[dot] {};
 \draw (0,-4) node[dot]{} -- (0,0) node[dot] {}; 
 \draw (-3,-4) node[dot] {} -- (0,0)node[dot] {} -- (3,-4) node[dot] {};
 \draw (-3,-8) node[dot]{} -- (-3,-4) node[dot] {}; 
 \draw (-6,-8) node[dot] {} -- (-3,-4)node[dot] {} -- (0,-8) node[dot] {};
\draw[line width=1pt] (0, 4)node[ddot, label=above:$r_1$]{} 
.. controls(4,6) and (6,6) ..
(10, 4)node[ddot, label=above:$r_2$]{}
(5, 12) node[label ] {$\underline{j = 3}$};
 }

\DeclareSymbol{T3}{0}
 {\draw (0,0) node[dot]{} -- (0,4) node[ddot] {}; 
 \draw (-3,0) node[dot] {} -- (0,4)node[ddot] {} -- (3,0) node[dot] {};
 \draw (0,-4) node[dot]{} -- (0,0) node[dot] {}; 
 \draw (-3,-4) node[dot] {} -- (0,0)node[dot] {} -- (3,-4) node[dot] {};
\draw (10,0) node[dot]{} -- (10,4) node[ddot] {}; 
 \draw (7,0) node[dot] {} -- (10,4)node[ddot] {} -- (13,0) node[dot] {};
\draw[line width=1pt] (0, 4)node[ddot, label=above:$r_1$]{} 
.. controls(4,6) and (6,6) ..
(10, 4)node[ddot, label=above:$r_2$]{}
(5, 12) node[label ] {$\underline{j = 3}$};
 }

\tikzstyle{dot1} = [ draw=  gray!00, 
 rectangle, rounded corners, fill=gray!00,  inner sep=1pt, inner ysep=3pt]

\tikzstyle{dot2} = [ draw=  black, 
ellipse, fill=gray!00,  inner sep=1pt, inner ysep=3pt]

\tikzstyle{dot3} = [ draw=  gray!00, 
ellipse, fill=gray!00,  inner sep=1pt, inner ysep=3pt]

\DeclareSymbol{T4}{0}
 {\draw (0,0) node[dot1]{$\phantom{n_2^{(1)} = n^{(2)}}$} -- (0,15) node[dot1] {$\phantom{n^{(1)}}$}; 
 \draw (-13,0) node[dot1] {$n_1^{(1)}$} -- (0,15)node[ddot] {$\phantom{n^{(1)}}$} 
 -- (13,0) node[dot1] {$n_3^{(1)}$};
 \draw (0,-15) node[dot1]{$n_2^{(2)}$} -- (0,0) node[dot1] {$\phantom{n_2^{(1)} = n^{(2)}}$}; 
 \draw (-13,-15) node[dot1] {$n_1^{(2)}$} -- (0,0)node[dot1] {$n_2^{(1)} = n^{(2)}$} -- (13,-15) node[dot1] {$n_3^{(2)}$};
\draw (40,0) node[dot1]{{$n_2^{(3)}$}} -- (40,15) node[dot3]  {$\phantom{n^{(1)} = n^{(3)}}$} ; 
 \draw (27,0) node[dot1] {$n_1^{(3)}$} -- (40,15)node[dot3]  {$\phantom{n^{(1)} = n^{(3)}}$}  -- (53,0) node[dot1] {$n_3^{(3)}$};
\draw[line width=1pt] (0, 15)node[ddot, label=above:$r_1$]{$n^{(1)}$} 
.. controls(10,23) and (27,23) ..
(40, 15)node[dot2, label=above:$r_2$] {$n^{(1)} = n^{(3)}$} ;
 }

\begin{document}

\baselineskip = 14pt

\title
[GWP of  the periodic cubic fourth order NLS] 
{Global well-posedness of the periodic cubic fourth order NLS in negative Sobolev spaces}

\author[T.~Oh and  Y.~Wang]
{Tadahiro Oh and Yuzhao Wang}

\address{
Tadahiro Oh, School of Mathematics\\
The University of Edinburgh\\
and The Maxwell Institute for the Mathematical Sciences\\
James Clerk Maxwell Building\\
The King's Buildings\\
Peter Guthrie Tait Road\\
Edinburgh\\ 
EH9 3FD\\
 United Kingdom}

\email{hiro.oh@ed.ac.uk}

\address{
Yuzhao Wang, School of Mathematics\\
The University of Edinburgh\\
and The Maxwell Institute for the Mathematical Sciences\\
James Clerk Maxwell Building\\
The King's Buildings\\
Peter Guthrie Tait Road\\
Edinburgh\\ 
EH9 3FD\\
 United Kingdom}

 \curraddr{School of Mathematics\\
Watson Building\\
University of Birmingham\\
Edgbaston\\
Birmingham\\
B15 2TT\\ United Kingdom}

\email{y.wang.14@bham.ac.uk}

\subjclass[2010]{35Q55}

\keywords{fourth order nonlinear Schr\"odinger equation;  
biharmonic nonlinear Schr\"odinger equation;  
short-time Fourier restriction norm method;
normal form reduction; enhanced uniqueness}

\begin{abstract}
	
We consider the Cauchy problem for the cubic fourth order nonlinear Schr\"odinger equation (4NLS) on the circle. 
In particular, we prove global well-posedness of  the renormalized 4NLS
in negative Sobolev spaces $H^s(\T)$, $s > -\frac{1}{3}$,
with enhanced uniqueness.
The proof consists of two separate arguments.
(i) We first prove global existence 
in $H^s(\T)$, $s > -\frac{9}{20}$, 
via the short-time Fourier restriction norm method.
By following the argument in Guo-Oh for the cubic NLS, 
this also leads to non-existence of solutions for the (non-renormalized) 4NLS 
in negative Sobolev spaces. 
(ii) We then  prove enhanced uniqueness in $H^s(\T)$, $s > -\frac{1}{3}$, by 
establishing an energy estimate for the difference of two solutions
with the  same  initial condition. 
For this purpose, we perform an infinite iteration 
of normal form reductions on  the $H^s$-energy functional,
allowing us to introduce
 an infinite sequence of correction terms to 
 the $H^s$-energy functional
 in the spirit of the $I$-method.
In fact, 
the main novelty of this paper
is this reduction of the $H^s$-energy functionals (for a single solution and
for the difference of two solutions with the same initial condition)
to sums of infinite series of multilinear terms
of increasing degrees.

\end{abstract}


\maketitle

\tableofcontents

\section{Introduction}

\noi
\subsection{The cubic  nonlinear Schr\"odinger equation with quartic dispersion}
In this paper, we consider the Cauchy problem for the cubic fourth order nonlinear  Schr\"odinger equation
(4NLS) on the circle $\T = \R/(2\pi \Z)$:
\begin{equation}
    \label{4NLS0}
    \begin{cases}
        i  \dt u =  \dx^4  u   \pm |u|^2 u \\
        u|_{t= 0} = u_0,
    \end{cases}
    \qquad (x, t)  \in \T \times \R,
\end{equation}

\noi
where $u$ is a complex-valued function. 
The equation \eqref{4NLS0}
is also called 
the biharmonic nonlinear Schr\"odinger equation (NLS)
and it was studied in \cite{IK, Turitsyn} in the context of stability of solitons in magnetic materials.
See also \cite{Karpman, KS, BKS,  FIP}
for a more general
class of fourth order NLS:
\begin{align}
i \dt u =  \ld \dx^2 u + \mu  \dx^4 u \pm  |u|^{2}u. 
\label{4NLS}
\end{align}

\noi
In the following, we focus our attention on the equation \eqref{4NLS0}.
See Remark \ref{REM:general} 
for a brief discussion on \eqref{4NLS}.

Our main goal is to study the well/ill-posedness issue
of \eqref{4NLS0} in the low regularity setting.
We first recall the scaling symmetry for \eqref{4NLS0};
 if $u(x, t)$ is a solution to \eqref{4NLS0}
on $\R$, then
$u_\ld(x,t) = \ld^{-2} u(\ld^{-1} x,\ld^{-4}t)$
is also a solution to \eqref{4NLS0} on $\R$ with the scaled initial data
$u_{0, \ld}(x) = \ld^{-2} u_0(\ld^{-1}x)$.
This scaling  symmetry
induces the so-called scaling critical Sobolev regularity $s_\text{crit} := -\frac 32$,
leaving the homogeneous $\dot{H}^{s_\text{crit}}$-norm invariant
under the scaling symmetry.
On the one hand, the scaling argument provides
heuristics indicating that  a PDE is 
 well-posed in $H^s$ for $s \geq s_\text{crit}$
 and is ill-posed in $H^s$ for $s < s_\text{crit}$.
This heuristics certainly applies to many equations,
including NLS and the nonlinear wave equations.
See \cite{CCT2b}.
On the other hand,  this heuristics is known to often fail  in negative Sobolev spaces.
This is indeed the case for \eqref{4NLS0} and its renormalized variant \eqref{4NLS1}.

In  \cite{OTz}, the first author and Tzvetkov proved that
 \eqref{4NLS0} is globally well-posed in $H^s(\T)$ for $ s\geq 0$. 
The proof is based on 
the Fourier restriction norm method (namely, utilizing  the $X^{s, b}$-space defined in \eqref{Xsb1})
with the $L^4$-Strichartz estimate:
\begin{equation}
\label{L4}
  \| u\|_{L_{x, t}^4} \lesssim \|u\|_{X^{0,\frac5{16}}} 
\end{equation}

\noi
along with  the conservation of the $L^2$-norm.
Following the approach in \cite{BGT, CCT1}, 
it was also shown in \cite{OTz} that  \eqref{4NLS0} is mildly ill-posed in $H^s(\T)$, $s < 0$, 
in the sense that the solution map$: u_0 \in H^s(\T) \mapsto u \in C([-T, T]; H^s(\T))$
is not locally uniformly continuous for $s < 0$.
Moreover, following the work \cite{GO}, it was pointed out in \cite{OTz} that 
 \eqref{4NLS0} is indeed ill-posed
in negative Sobolev spaces by establishing
a non-existence result. 
See Corollary  \ref{COR:NE} below 
for a precise statement.
We also mention the following norm inflation result due to Choffrut-Pocovnicu \cite{CP}.
Let $s < -\frac 76$. Then, 
given any $\eps > 0$, 
there exist a solution $u_\eps$ to \eqref{4NLS0} on $\T$
and $t_\eps  \in (0, \eps) $ such that 
\[ \| u_\eps(0)\|_{H^s(\T)} < \eps \qquad \text{ and } \qquad \| u_\eps(t_\eps)\|_{H^s(\T)} > \eps^{-1}.\] 

\noi
See also \cite{OW}.
It is worthwhile to note that the regularity $-\frac 76$ is higher than the scaling critical regularity
$s_\text{crit} = -\frac 32$
and that this norm inflation result for $s < -\frac 76$
also applies to the renormalized 4NLS \eqref{4NLS1} below.

In the next subsection, we introduce an alternative formulation
for \eqref{4NLS0} such that 
(i)~it is equivalent to \eqref{4NLS0} in $L^2(\T)$ but
(ii)~it behaves better than \eqref{4NLS0} in negative Sobolev spaces.
In the following, 
the defocusing/focusing nature of the equation \eqref{4NLS0} does not play any role.
Hence,  we assume that it is defocusing, i.e.~with the $+$ sign in \eqref{4NLS0}.

\subsection{Renormalized cubic fourth order NLS}

Given a  global solution $u \in C(\R; L^2(\T))$ to \eqref{4NLS0}, 
we define the following invertible gauge transformation $\mathcal{G}$ by 
\begin{equation*}
\mathcal{G}(u)(t) : = e^{ 2 i t \mu(u)} u(t)
\end{equation*}

\noi
with its inverse 
\begin{equation}
\mathcal{G}^{-1}(u)(t) : = e^{- 2 i t \mu(u)} u(t),
\label{gauge2}
\end{equation}

\noi
where $\mu(u) = \fint |u(x, t)|^2 dx := \frac{1}{2\pi} \int_\T |u(x, t)|^2 dx.$
Thanks to the $L^2$-conservation, 
$\mu(u)$ is defined, independently of $t \in \R$,
as long as $u_0 \in L^2(\T)$.
A direct computation shows that the gauged function,  which we still denote by $u$, 
satisfies the following renormalized 4NLS:
\begin{equation}
	\label{4NLS1} 
	\begin{cases}
		i \dt u =  \dx^4 u +( |u|^2 -2  \fint |u|^2 dx) u \\
		u|_{t= 0} = u_0,
	\end{cases}
\quad 	(x, t)  \in \T \times \R.
\end{equation}

\noi
This renormalization appears as an equivalent formulation 
of the Wick renormalization\footnote{By viewing $u$
as a complex-valued Gaussian random variables, 
the Wick renormalization of $|u|^2 u$
is nothing but a projection onto the Wiener homogenous chaoses of order three.}
in Euclidean quantum field theory \cite{BO96, OS, OTh}.
For this reason, we will refer to \eqref{4NLS1}
as the {\it Wick ordered cubic 4NLS}
in the following.

In view of the invertibility of 
$\mathcal{G}$ on $L^2(\T)$, 
we see that  the original cubic 4NLS \eqref{4NLS0}
and the Wick ordered cubic 4NLS \eqref{4NLS1} describe equivalent dynamics
on $ L^2(\T)$.
On the other hand, the gauge transformation $\mathcal{G}$ does not make sense outside $L^2(\T)$.
Hence, they describe genuinely different dynamics, if any, outside $L^2(\T)$.

It is easy to see that this specific choice of  gauge for \eqref{4NLS1}
 removes a certain singular component from the cubic nonlinearity.
Indeed, 
the nonlinearity on the right-hand side of \eqref{4NLS1} can be written as
\begin{align}  
\textstyle  \Nf(u)= \Nf(u, u, u)
 :\!&= \textstyle \big( |u|^2 - 2 \fint  |u|^2 dx\big)u
 =  \N (u, u, u) - \RR (u, u, u),
 \label{Xnonlin}
\end{align}

\noi
where the non-resonant part $\N$ and the resonant part $\RR$ are defined by
\begin{align}
\label{NN1}
& \N(u_1, u_2, u_3) (x, t)  = \sum_{n \in \Z} \sum_{\substack{n = n_1 - n_2 + n_3\\n \ne n_1, n_3}}
 \ft{u}_1(n_1, t)\cj{\ft{u}_2(n_2, t)}\ft{u}_3(n_3, t)e^{inx}, \\
\label{NN2}
& \RR(u_1, u_2, u_3) (x, t) = \sum_{n \in \Z} \ft{u}_1(n, t)\cj{\ft{u}_2(n, t)}\ft{u}_3(n, t)e^{inx}.
\end{align}

\noi
Namely, the gauge transformation basically eliminates
the contribution from $ n = n_1$ or $n = n_3$.
In the following, 
we {\it choose} to study 
the Wick ordered cubic 4NLS \eqref{4NLS1}.
As with any renormalization procedure or gauge choice, we stress that this is a matter of choice.
See Remark \ref{REM:choice}.

We now state our first result.

\begin{theorem}[Global existence]\label{THM:1}
Let $s \in \big(-\frac{9}{20}, 0\big)$.
Given $u_0\in H^s(\T)$, there exists
a global solution 
 $u \in C(\R;  H^s(\T))$ to
the Wick ordered cubic 4NLS \eqref{4NLS1}
with $u|_{t = 0} = u_0$. 
\end{theorem}

On the one hand, as in \cite{OTz}, 
one can easily prove local well-posedness of \eqref{4NLS1} in $L^2(\T)$
by  a Picard iteration.
On the other hand, 
it is easy to see that \eqref{4NLS1}
is mildly ill-posed in negative Sobolev spaces 
in the sense of the failure of local uniform continuity
of the solution map
\cite{CO, OTz}; see Remark 1.4 in \cite{CO}.
This in particular implies that one can not
use a Picard iteration to construct solutions to \eqref{4NLS1}
in negative Sobolev spaces.
We instead use a more robust  energy method to construct solutions.
More precisely, we 
use the short-time Fourier restriction norm method to prove Theorem \ref{THM:1}. 
Here, the short-time Fourier restriction norm method
simply means that we use 
dyadically defined  $X^{s, b}$-type spaces
with suitable localization in time,
depending on the dyadic size of spatial frequencies.
A precursor of this method
appeared in the work of Koch-Tzvetkov \cite{KTz}, 
where 
localization in time was combined with the Strichartz norms.
The short-time Fourier restriction norm method has been very effective 
in establishing a priori bounds on solutions
in low regularity spaces (yielding even uniqueness in some cases), 
in particular, where a solution map
is known to fail to be locally uniformly continuous.
See \cite{CCT3, IKT, KT1, Guo1, KT2,  KP}.

Given $T > 0$, let $F^s(T) \subset C([-T, T]; H^s(\T))$ denote the local-in-time version
of the $X^{s, b}$-space adapted to appropriately chosen short-time scales
and let $N^s(T)$ be its ``dual'' space.
See Section \ref{SEC:2} for their precise definitions.
In establishing the local existence part of   Theorem \ref{THM:1}, 
our main goal is to establish  the following three estimates:
\begin{align}
   \text{Linear estimate:}& 
&    \| u \|_{F^s(T)}  &  \les \|u\|_{E^s(T)} +  \| \Nf(u)\|_{N^s(T)},  \label{Alinear}
 & \\
   \text{Nonlinear estimate:}& 
&  \| \Nf(u) \|_{N^s(T)}  & \les \|u\|_{F^s(T)}^3,  \label{Anonlinear}
& \\
   \text{Energy estimate:}& 
&    \| u \|_{E^s(T)}^2  &  \les \|u_0\|_{H^s}^2 + \|u\|_{F^s(T)}^{4}, \label{Aenergy}  
&
\end{align}

\noi
where $E^s(T) \thickapprox L^\infty([-T, T]; H^s(\T))$.
These three estimates yield an a priori bound
on (smooth) solutions in $H^s(\T)$, which allows us to prove existence 
of local-in-time solutions 
(without uniqueness)
by a compactness argument.
As we see in the later sections, the short-time restriction adapted
to the spatial dyadic scales allows us to gain extra modulation (i.e.~smoothing)
 in the resonant case.
This in particular enables us to prove the trilinear estimate \eqref{Anonlinear} below $L^2(\T)$.

As for the global existence part, 
we employ the following  $H^s_M$-norm adapted to the parameter $M\geq 1$
defined by 
\begin{align*}
 \| f \|_{H^s_M} = \big\|(M^2 + n^2)^\frac{s}{2} \ft f(n)\big\|_{\l^2_n}.
\end{align*}

\noi
While the   $H^s_M$-norm is equivalent to the standard $H^s$-norm, 
we have the following decay property when $s < 0$:
\begin{align*}
\lim_{M\rightarrow \infty} \| f \|_{ H_M^s} = 0
\end{align*}

\noi
for all $f \in H^s(\T)$.
This  allows us to reduce the problem to a small data theory in some appropriate sense.
See Section~\ref{SEC:GWP}.

As a corollary to the local-in-time a priori estimate  established in the proof of Theorem \ref{THM:1}
for solutions to the Wick ordered cubic 4NLS \eqref{4NLS1}
(see Remark \ref{REM:LWP}), 
we obtain the following non-existence result
for the original cubic 4NLS claimed above.

\begin{corollary}
\label{COR:NE}

Let $s \in \big( -\frac{9}{20}, 0\big)$
and $u_0 \in H^s (\T)\setminus L^2(\T)$.
Then, for any $T>0$, 
there exists no distributional  solution $u \in C([-T, T]; H^s(\T))$ to 
the cubic  4NLS \eqref{4NLS0}
such that

\begin{itemize}
\item[\textup{(i)}] $u|_{t = 0} = u_0$, 

\smallskip

\item[\textup{(ii)}] There exist smooth global solutions $\{u_n\}_{n\in \mathbb{N}}$ 
to \eqref{4NLS0} such that 
$u_n \to u$ in $ C([-T, T]; \mathcal{D}'(\T))$ as $n \to \infty$. 
\end{itemize}

\end{corollary}

In \cite{GO}, the first author with Guo proved a similar non-existence result
for the standard cubic NLS:
\begin{equation*}
\textstyle i \dt u =  \dx^2 u + |u|^2 u 
\end{equation*}

\noi
in $H^s(\T)$, $s > -\frac 18$, 
by first establishing an a priori estimate for 
solutions to the following Wick ordered cubic NLS:
\begin{equation}
\label{NLS1} 
\textstyle i \dt u =  \dx^2 u + ( |u|^2 -2  \fint |u|^2 dx) u. 
\end{equation}

\noi
The main idea of the proof is to exploit the fast oscillation 
in the phase of the inverse gauge transformation \eqref{gauge2}
and apply Riemann-Lebesgue lemma.
See Section 9 in \cite{GO} for details. 
Note that our assumption in Corollary \ref{COR:NE}
is slightly weaker than that in Theorem 1.1 in \cite{GO},
namely, the convergence in (ii) is assumed only in $C([-T, T]; \mathcal{D}'(\T))$ 
but that the same proof applies
since the only ingredient needed from this assumption is the following convergence:
$\jb{u_n(\cdot, t), \phi(\cdot, t)}_{L^2_x}
\to  \jb{u(\cdot, t), \phi(\cdot, t)}_{L^2_x}$
for any test function    $\phi \in \mathcal D(\T\times [-T, T])$.

\begin{remark}\label{REM:choice}\rm
By introducing another  gauge transformation
$\mathcal{G}_\g(u)(t) : = e^{  i\g t  \mu(u)} u(t)$
with a parameter $\g\in \R$, 
we arrive at  a different  renormalized cubic 4NLS:
\begin{align}
i \dt u 
& = \textstyle  \dx^4 u + ( |u|^2 -\g  \fint |u|^2 dx) u\notag\\
& = \dx^4 u + ( |u|^2 -\g \cdot \infty) u.
\label{IntroWNLS}
\end{align}

\noi
As it was mentioned in \cite{GO} in the context of the cubic NLS, 
it is crucial 
 to {\it subtract off the  right amount of infinity}
 in this renormalization procedure.
It is easy to see that \eqref{4NLS0}, \eqref{4NLS1}, and \eqref{IntroWNLS}
are all equivalent in $L^2(\T)$.
In negative Sobolev spaces, however, 
they are very different.
In fact, the same non-existence result in negative Sobolev spaces
holds for \eqref{IntroWNLS} unless $\g = 2$, which 
 shows that ``$2\cdot \infty$'' is the right amount to subtract
in the renormalization procedure.	

\end{remark}

\begin{remark}\rm
By applying our analysis with a parameter $M \geq 1$, 
we can extend the local existence result
of the Wick ordered cubic NLS \eqref{NLS1} in \cite{GO}
to global existence (without uniqueness) in $H^s(\T)$, $s > -\frac 18$.
\end{remark}

\medskip

Next, we turn our attention 
to the uniqueness issue
of the solutions constructed in Theorem \ref{THM:1}.
The main source of 
difficulty 
lies in establishing 
 an  energy estimate 
for the difference of two solutions.
The energy estimate \eqref{Aenergy} for a single solution follows from 
an argument analogous to the $I$-method (the method of almost conservation laws)
\cite{CKSTT2, CKSTT1}, which is ultimately 
 based on the conservation of the $L^2$-norm for \eqref{4NLS1}.
The $L^2$-norm of the difference of two solutions, however, is not conserved
under \eqref{4NLS1}.
Moreover, an estimate of the form:
\begin{align*}
    \| u  - v \|_{E^s(T)}^2  &  \les \|u(0) - v(0) \|_{H^s}^2 + 
 \big(   \|u\|_{F^s(T)}^{3}+ \| v\|_{F^s(T)}^{3}\big) \|u- v\|_{F^s(T)}
\end{align*}

\noi
is {\it false} since it would imply smooth dependence on initial data
and such smooth dependence is known to fail in negative Sobolev spaces
\cite{CO, OTz}.
In the following, we establish an energy estimate
for the difference of two solutions with the {\it same} initial condition
and thus  prove uniqueness of solutions to the Wick ordered cubic 4NLS \eqref{4NLS1}.
Furthermore, our argument for proving uniqueness does not use any auxiliary function space
(in particular, we do not use the short-time Fourier restriction norm method)
and thus yields uniqueness in an {\it enhanced} sense.

\begin{theorem}[Global well-posedness with enhanced uniqueness]\label{THM:2}
Let $s \in (-\frac{1}{3}, 0)$.
Then, the Wick ordered cubic 4NLS \eqref{4NLS1}
is globally well-posed in $H^s(\T)$.
More precisely, the solution constructed in Theorem \ref{THM:1} is unique
and the solution map is continuous.
Here, the uniqueness  holds in an enhanced sense;
the solution constructed in Theorem \ref{THM:1} is unique
among all the solutions in $C(\R; H^s(\T))$ to \eqref{4NLS1}
with the same initial data
equipped with smooth approximating solutions.\footnote{Here, we implicitly assume
that these solutions belong to various  auxiliary functions spaces
so that the cubic nonlinearity makes sense in some appropriate manner.
The point is that we do not need to know which auxiliary function space each solution belongs to.
Moreover, we assume that they satisfy the local-in-time estimate: $\| u \|_{C_T H^s} \les \| u_0\|_{H^s}$
for some $T = T(\|u_0\|_{H^s}) > 0$.
See Remark \ref{REM:app}.}

\end{theorem}

Note that our enhanced uniqueness does {\it not} assert unconditional uniqueness\footnote{By slightly modifying the presentation in \cite{GKO}, one can easily prove unconditional uniqueness
of \eqref{4NLS0} and \eqref{4NLS1} in $C(\R; H^s(\T))$ for $s \geq \frac 16$.
Clearly, the threshold $s \geq \frac 16$ is sharp in view of the embedding:
$H^\frac{1}{6}(\T) \subset L^3(\T)$ (in making sense of the cubic nonlinearity).
Recall  also the non-uniqueness result by Christ \cite{CH1}
 of weak solutions in the extended sense in negative Sobolev spaces, 
 where the nonlinearity is interpreted only as a limit of smooth nonlinearities.}
in $C(\R; H^s(\T))$, 
since we do assume that solutions with smooth approximating solutions have some extra regularity
so that the cubic nonlinearity makes sense.
Instead, our uniqueness statement should be interpreted as follows;
given $u_0 \in H^s(\T)$, let $u$ be a solution to \eqref{4NLS1} with $u|_{t = 0} = u_0$ constructed
in Theorem \ref{THM:1} via this particular version of the short-time Fourier restriction norm method.
Suppose that $v$ is another solution to \eqref{4NLS1} with $v|_{t = 0} = u_0$
constructed by some other method, for example, 
by another version of the short-time Fourier restriction norm method
or by an adaptation of Takaoka-Tsutsumi's argument \cite{TT} to \eqref{4NLS1},
where the definition of the $X^{s, b}$-space incorporates the initial data.\footnote{In a recent paper \cite{Kwak},
Kwak applied the ideas from \cite{TT, NTT} and proved local well-posedness of \eqref{4NLS1} in $H^s(\T)$
for $s \geq - \frac 13$.  On the one hand, this result extends local well-posedness of \eqref{4NLS1}
to the endpoint regularity $s = -\frac{1}{3}$.  On the other hand, the uniqueness
in \cite{Kwak} holds only in (a variant of) the $X^{s, b}$-space.
When $s > -\frac 13$, 
 the enhanced uniqueness in Theorem \ref{THM:2}
allows us to conclude that 
 the local-in-time  solution constructed in \cite{Kwak}
 agrees with our solution constructed in Theorem \ref{THM:2} and hence is global.
When $s = -\frac 13$, global well-posedness of \eqref{4NLS1} is open.
We also mention an analogous work by Miyaji-Tsutsumi \cite{MT}, prior to \cite{Kwak}, 
on local well-posedness of the (renormalized) third order NLS in negative Sobolev spaces,
based on a variant of the $X^{s, b}$-space incorporating initial data.
}
In general, we do not have a way to compare these solutions belonging to 
different resolution spaces.
The enhanced uniqueness in Theorem \ref{THM:2}, however, asserts that $u$ and $v$ must agree.
 It is in this sense that our uniqueness statement in Theorem \ref{THM:2} is enhanced
 since it allows us to compare solutions constructed by different methods.
It seems that this notion of enhanced uniqueness is 
one of the strongest forms of uniqueness 
 ``in practice''.\footnote{Given a solution, it seems reasonable in practice to 
 assume that it comes with  at least one sequence of smooth approximating solutions.}

We stress that this enhanced uniqueness is by no means automatic
since we do {\it not} have a priori continuous dependence.\footnote{Our proof of continuous dependence follows as a consequence of the uniqueness statement and the a priori bound obtained
in the proof of Theorem \ref{THM:1}. See Section~\ref{SEC:uniq}.}
Let $u, v \in C(\R; H^s(\T))$ be two solutions
to \eqref{4NLS1} with the same initial data
with {\it some} smooth approximating solutions $\{ u_n \}_{n \in \NB}$
and $\{ v_m \}_{m \in \NB}$, respectively.
Then, given $T > 0$, we have
\[ \| u - v\|_{C_T H^s} \leq 
\| u - u_n\|_{C_T H^s}
+ \| u_n - v_m\|_{C_T H^s} + \| v_m - v\|_{C_T H^s}, \]

\noi
where $C_TH^s = C([-T, T]; H^s(\T))$.
The first and third terms on the right-hand side tend to 0 as $n, m \to \infty$.
We, however, do not have any way to compare $u_n$ and $v_m$
in general, since we do not even know how these solutions $u$ and $v$ are constructed.
Nonetheless,  our enhanced uniqueness in Theorem \ref{THM:2} allows us to conclude that  $u \equiv v$.

In establishing an  energy estimate for the difference of two solutions $u$ and $v$ to \eqref{4NLS1}
with the same initial condition, 
we perform an infinite iteration of normal form reductions
(= integration by parts in time).\footnote{In fact, 
this process basically corresponds to the Poincar\'e-Dulac normal form reductions.
See the introduction in \cite{GKO}.}
In \cite{GKO}, the first author with Guo and Kwon 
proved unconditional well-posedness 
of  the cubic NLS on $\T$ in low regularity
by performing  normal form reductions infinitely many times.
See also~\cite{CGKO}.\footnote{On the one hand, we implemented an infinite iteration of normal form reductions in \cite{CGKO}.
On the other hand, symmetrization at each step played a crucial role in \cite{CGKO}.
In this paper, we will not employ such  a symmetrization argument.
}
In our current setting, we do not work at the level of the equation \eqref{4NLS1} unlike \cite{GKO}.
We instead implement an infinite iteration scheme of normal form reductions
for  the evolution equations 
satisfied by energy quantities.
See \eqref{intromain3} and \eqref{intromain5} below.

More precisely, we first apply an infinite iteration of normal form reductions
to a solution  $u \in C([-T, T]; H^s(\T))$ to \eqref{4NLS1} 
and re-express $u$ as 
\begin{align}
 |\ft u(n, t)|^2 & -  |\ft u_0(n)|^2 
  = \mathfrak{S}_\infty(u)(n, t) \notag\\
&  : = \sum_{j=2}^\infty \NN_0^{(j)} (u)(n, t')\bigg|_0^t 
+  \int_0^t \bigg[\sum_{j = 2}^\infty   \RR^{(j)}( u)(n, t') 
+ \sum_{j = 1}^\infty\NN_1^{(j)} (u)(n, t')\bigg]
dt, 
\label{intromain1}
\end{align}

\noi
where 
$\NN_0^{(j)}$  (and  $\RR^{(j)}$ and $\NN_1^{(j)}$) are $2j$-linear forms
($(2j+2)$-linear forms, respectively).\footnote{More precisely, for fixed $t \in \R$,
$\{\NN_0^{(j)}(n, t)\}_{n \in \Z}$ is a sequence of  $2j$-linear forms.
Equivalently, by 
viewing $\NN_0^{(j)}(u)(n, t)$ as the Fourier coefficient of $\NN_0^{(j)}(u)(t)$, 
we can view $\NN_0^{(j)}(\cdot)(t)$ as a $2j$-linear operator.
With abuse of terminology,  however, we simple refer to $\NN_0^{(j)}$ as a $2j$-linear form in the following.
A similar comment applies to 
 $\RR^{(j)}$ and $\NN_1^{(j)}$.
}
Moreover, we show that these multilinear forms are bounded in $C([-T, T]; H^s(\T))$, 
$s > -\frac 13$, 
uniformly in $n \in \Z$.\footnote{In fact, we show that they are absolutely summable over $n \in \Z$.}
See Proposition \ref{PROP:main} below for a precise statement.

Now, take  
 two solutions $u$ and $v$ to \eqref{4NLS1}
constructed in Theorem \ref{THM:1} 
with the same initial condition $u|_{t = 0} = v|_{t = 0} =  u_0$,
satisfying $\| u \|_{C_TH^s}, \| v \|_{C_TH^s} \les \| u_0 \|_{H^s}$. 
Then, with \eqref{4NLS1}, we have 
\begin{align}
\frac{d}{dt} \| u (t)- v(t) \|_{H^s}^2
& =  -  2\Re i \sum_{n\in \Z} \jb{n}^{2s} 
\big[\ft {\N(u)}(n) -  \ft{\N(v)}(n)\big]
 \cj{(\ft u (n)  - \ft v(n))}  \notag \\
& \hphantom{XX}
+ 2\Re i \sum_{n\in \Z} \jb{n}^{2s} 
\big[\ft {\RR(u)}(n) - \ft {\RR(v)}(n)\big]
 \cj{(\ft u(n) - \ft v(n))}\notag\\
 & =: \I + \II.
 \label{intromain2}
 \end{align}
 
 \noi
Let us only consider the second term $\II$, corresponding to the resonant contribution.
Using \eqref{intromain1} with $u|_{t = 0} = v|_{t = 0} =  u_0$
and \eqref{NN2},
we obtain
\begin{align*}
|\II(t)|  
& \les  \bigg|\sum_{n\in \Z} \jb{n}^{2s}  \big(|\ft u(n, t) |^2 - |\ft v(n, t) |^2\big)\cj{(\ft u(n) - \ft v(n))} \ft v(n)\bigg| \notag\\
& \leq  \sup_{m \in \Z} \big| \mathfrak{S}_\infty(u)(m, t)  -  \mathfrak{S}_\infty(v)(m, t)\big|\cdot 
\sum_{n\in \Z} \jb{n}^{2s} 
  |\ft u(n) - \ft v(n)| |\ft v(n)| \notag\\
& \leq C(\|u_0\|_{H^s}) \| u - v\|_{C_T H^s}^2,
\end{align*}

\noi
where we used the multilinearity
of $\NN_0^{(j)}$,   $\RR^{(j)}$, and $\NN_1^{(j)}$
along with their $C([-T, T] H^s(\T))$-bounds to control the first factor 
and Cauchy-Schwarz inequality on the second factor.
As for the non-resonant contribution $\I$ in \eqref{intromain2}, 
we expand it  into a sum of infinite series
analogous to \eqref{intromain1}
and obtain an estimate of the form:
\begin{align*}
|\I(t)|  
& \leq C(\|u_0\|_{H^s}) \| u - v\|_{C_T H^s}^2.
\end{align*}

\noi
See Proposition \ref{PROP:main2} and Subsection \ref{SUBSEC:8.6}.
This yields the desired energy estimate for the difference of two solutions
with the same initial condition.

Therefore, the main task is to prove the identity \eqref{intromain1}
with good estimates.
We achieve this goal by performing integration by parts in an iterative manner, 
which introduces nonlinear terms of higher and higher degrees.
While these nonlinear terms thus introduced are of higher degrees, 
it turns out that they satisfy better estimates.
Namely, 
this infinite iteration of normal form reductions
allows us to exchange analytical difficulty
with combinatorial and notational complexity.
In order to keep track of all possible ways to perform integration by parts, 
we introduce the notion of {\it ordered bi-trees.}
We devote Section \ref{SEC:NF} for presenting the normal form reductions.

Lastly, we point out the connection to the $I$-method.\footnote{The connection between 
normal form reductions and modified energies in the $I$-method has already been pointed out in \cite{GO}.}
At each step of integration by parts, we introduce 
boundary terms. 
This corresponds to adding a correction term appearing in the $I$-method.
Namely,  in the context of the $I$-method, 
our approach  is nothing 
but to compute and estimate 
a modified energy of an {\it infinite order}.\footnote{The highest order of modified energies
used in the literature  is three in the application of the $I$-method to the KdV equation \cite{CKSTT1}, corresponding 
to two iterations of normal form reductions.}
For example, our argument yields
the following infinite expansion of the $H^s$-energy for a single solution:
\begin{align}
 \| u (t) \|_{H^s}^2  -  \| u (0) \|_{H^s}^2
&  = \sum_{j=2}^\infty \sum_{n\in \Z}\jb{n}^{2s}\NN_0^{(j)} (u)(n, t')\bigg|_0^t \notag\\
& \hphantom{XXX}+ 
 \int_0^t  \bigg[  \sum_{j = 2}^\infty \sum_{n\in \Z}\jb{n}^{2s}\RR^{(j)}( u)(n, t') 
 \notag\\
& \hphantom{XXXXXXXX}
+  \sum_{j = 1}^\infty \sum_{n\in \Z} \jb{n}^{2s} \NN_1^{(j)} (u)(n, t')
 \bigg]dt'.
\label{intromain3}
\end{align}

\noi
See \eqref{Y3} below.
Namely, defining a modified energy $E_\infty(u)$ of an infinite order by 
\begin{align*}
E_\infty(u) =  \| u  \|_{H^s}^2  
-  \sum_{j=2}^\infty\sum_{n\in \Z} \jb{n}^{2s}\NN_0^{(j)} (u)(n), 
\end{align*}

\noi
we obtain
\begin{align}
E_\infty(u)(t) & - E_\infty(u)(0)\notag\\
&  =  \int_0^t  \bigg[  \sum_{j = 2}^\infty  \sum_{n\in \Z}\jb{n}^{2s}\RR^{(j)}( u)(n, t') 
+  \sum_{j = 1}^\infty  \sum_{n\in \Z}  \jb{n}^{2s} \NN_1^{(j)} (u)(n, t')
 \bigg]dt'.
\label{intromain4}
\end{align}

\noi
While we do not need the modified energy $E_\infty(u)$ in this paper,
such an expansion by adding an infinite sequence of correction terms
seems to be new and of interest.

As for the difference of two solutions with the same initial data, 
while there are contributions from the resonant part
as well as the cross terms ($\I_{uv}$ and $\I_{vu}$ below) in the non-resonant part,
we also have a similar infinite expansion
(with two factors of $u - v$):
\begin{align}
 \| u (t)- v(t) \|_{H^s}^2
  & =   \int_0^t \Big\{( \I_{uu}(t') - \I_{uv}(t')) - (\I_{vu}(t')- \I_{vv}(t'))\Big\}
dt'\notag\\
& \hphantom{X}
+ \int_0^t    \sum_{n\in \Z} \jb{n}^{2s} 
\big( \mathfrak{S}_\infty(u)(n, t') - \mathfrak{S}_\infty(v)(n, t')\big)
\cj{(\ft u_n - \ft v_n)} \ft v_n(t') dt', 
\label{intromain5}
\end{align}

\noi
where the second term on the right-hand side involves a sum of infinite series
in view of~\eqref{intromain1}.
As for the integrands in the first integral,
see \eqref{Y1}, \eqref{Y3},  \eqref{Y4},  \eqref{Y9},  and \eqref{Y10},
where each integrand is written as a sum of infinite series.   
See Remark \ref{REM:uniq3}.

We conclude this introduction by stressing that 
 reducing the $H^s$-energy functionals
to the infinite series expansions \eqref{intromain3} and \eqref{intromain5}
(also see  \eqref{intromain4})
is 
 the main novelty of this paper.
In the proof of Theorem \ref{THM:2}, 
we use the infinite series expansion \eqref{intromain5} for the difference of two solutions
with the same initial data
to prove  uniqueness of solutions to \eqref{4NLS1}.
In \cite{OSTz}, 
(a variant of) the infinite series expansion \eqref{intromain3} for a single solution
plays an important role in establishing a crucial energy estimate
in studying the transport property of Gaussian measures on periodic functions
under the flow of 4NLS \eqref{4NLS0}.
  See Remark \ref{REM:OSTz}.
We hope that this idea of expanding energy functionals into infinite series
by normal form reductions 
can be applied to other equations
in various  settings.

\begin{remark}\label{REM:OSTz}\rm 
In a recent work \cite{OSTz}, the first author with Sosoe and Tzvetkov
established an optimal regularity result 
for quasi-invariance of the Gaussian measures on Sobolev spaces
under the original 4NLS \eqref{4NLS0}
by implementing a similar infinite iteration of normal form
reductions on the $H^s$-functional for solutions to  \eqref{4NLS0}
for  $s \in (\frac 12, 1)$.
While there are similarities between the normal form approach  in \cite{OSTz}
and in Section \ref{SEC:NF} of  this paper, 
more care is required in the present paper
since we need to gain derivatives at each step of normal form reductions
in order to estimate the multilinear forms
 $\NN_0^{(j)}(u)$,   $\RR^{(j)}(u)$, and $\NN_1^{(j)}(u)$
of arbitrarily large degrees
in terms of the {\it negative} Sobolev norm of $u$.

\end{remark}

\begin{remark} \rm
Recall that the mean-zero Gaussian white noise on $\T$
is formally given by 
\begin{align*}
 d \mu 
  = Z^{-1} e^{-\frac 12 \| \phi\|_{L^2}^2} d\phi .
 \end{align*}

\noi
In particular, a typical element under $\mu$ is given by\footnote{Throughout this paper, 
we drop the harmless factor of $2\pi$.} 
\begin{align}
\phi^\o(x) = \phi(x; \o) = \sum_{n \in \Z} g_n(\o)e^{inx}, 
\label{gauss1}
 \end{align}

\noi
where
$\{ g_n \}_{n \in \Z}$ is a sequence of independent standard complex-valued 
Gaussian random variables on a probability space $(\O, \mathcal{F}, P)$.
From \eqref{gauss1}, it is easy to see that 
 $\phi^\o$ in \eqref{gauss1} lies   in $H^s(\T)\setminus H^{-\frac 12}(\T)$, $s < -\frac 12$,  almost surely.
In particular, the regularity of the white noise is below the regularities stated
in Theorems \ref{THM:1} and \ref{THM:2}.

In view of the $L^2$-conservation for \eqref{4NLS1}
and the Hamiltonian structure of the equation, 
one may expect that the white noise is invariant under the dynamics of \eqref{4NLS1}.
In  \cite{OTzW}, the authors with Tzvetkov proved that this is indeed the case.
The main difficulty in \cite{OTzW} lies in constructing
 local-in-time dynamics with respect to the random initial data~\eqref{gauss1},
 which was overcome by a combination of new stochastic analysis and
 deterministic analysis
 different from the analysis presented in this paper.

 \end{remark}

\begin{remark}\label{REM:general}\rm 
We can also start our discussion with 
the more general cubic fourth order NLS \eqref{4NLS}
with $\mu \ne 0$
and consider its renormalized version.
In this case, 
the following phase function
\begin{align}
\phi_{\ld, \mu}(\bar n) 
& 	=  -\ld(n_1^2 - n_2^2 + n_3^2 - n^2)
	+ \mu(n_1^4 - n_2^4 + n_3^4 - n^4)\notag\\
& =  (n_1 - n_2)(n_1-n) 
\big\{- 2\ld + \mu \big(n_1^2 +n_2^2 +n_3^2 +n^2 + 2(n_1 +n_3)^2\big)\big\}.
 \label{phase1}
\end{align}

\noi
plays an important role in the analysis.
Compare this with \eqref{Phi}.

If the last factor in \eqref{phase1}
does not vanish for any $n_1, n_2, n_3, n \in \Z$, 
then the main results in this paper clearly hold with the same proofs.
Note that even if the last factor may be 0, 
i.e.~$ 2\ld \in \mu \NB$, the new resonance
occurs only for low frequencies,
where $\max (n_1^2, n_2^2, n_3^2, n^2) \les \frac{\ld}{\mu}$,
and hence the same argument basically holds.

\end{remark}

This paper is organized as follows.
In Section \ref{SEC:2}, 
we introduce notations and the function spaces along with their basic properties.
In Section \ref{SEC:Strichartz}, 
we present  multilinear Strichartz estimates,
which are then used to prove the crucial trilinear estimate \eqref{Anonlinear}
and the energy estimate \eqref{Aenergy} in Sections \ref{SEC:trilinear}
and \ref{SEC:energy}, respectively.
In Section \ref{SEC:GWP}, we present the proof of 
global existence (Theorem \ref{THM:1}).
In particular, given any $T > 0$, we choose $M = M(T) \gg 1$
such that the estimates \eqref{Alinear} - \eqref{Aenergy} adapted to the parameter $M$
allow us to construct solutions
on the time interval $[-T, T]$. 
In Section \ref{SEC:uniq}, 
by assuming the key propositions (Propositions \ref{PROP:main} and \ref{PROP:main2}), 
we prove uniqueness of solutions to \eqref{4NLS1} in $H^s(\T)$
for $s > -\frac 13$, which is then used to establish 
continuous dependence on initial data and thus global well-posedness
(Theorem \ref{THM:2}).
In Section \ref{SEC:NF}, 
we present details of the normal form reductions
and prove Propositions \ref{PROP:main} and \ref{PROP:main2}.

\section{Notations, 
function spaces, and their basic properties}
\label{SEC:2}

\subsection{Notations} \label{SUBSEC:notations}
For $a, b > 0$, we use $a\lesssim b$ to mean that
there exists $C>0$ such that $a \leq Cb$.
By $a\sim b$, we mean that $a\lesssim b$ and $b \lesssim a$.
We also use $a+$ (and $a-$) to denote $a + \eps$ (and $a - \eps$, respectively) for arbitrarily small $\eps \ll 1$.

Given a function $u$ on $\T\times \R$, we
use $\ft{u}$ and $\F (u)$
to denote the space-time Fourier transform of $u$ given by
\[ \ft{u}(n, \tau) = \int_{\T\times\R} e^{- inx} e^{-it\tau} u(x, t) dx dt.\]

\noi
When there is no confusion,
we may simply use $\ft{u}$ or $\F(u)$
to denote
the spatial, temporal, or space-time Fourier transform
of $u$, depending on the context.
In dealing with the spatial Fourier transform,
we often denote $\ft{u}(n, t)$
by $\ft{u}_n(t)$.

For $k \in \Z_{\geq 0}: = \Z\cap [0, \infty)$, 
we define the dyadic intervals $I_k$ by setting
 $I_0 = \{\xi: |\xi| < 1\}$
and
$I_k = \{ \xi : 2^{k-1} \leq |\xi|  < 2^k\}$ for  $k \geq 1$.
Next, 
we define the dyadic intervals $I_{k}^M$, $k \geq \log_2 M$,  adapted to a given {\it dyadic} parameter $M \geq 1$,
by setting
\[I_{k}^M = 
\begin{cases}
\rule[-3mm]{0pt}{0pt}I_k, & k > \log_2 M,\\
\displaystyle
\bigcup_{k=0}^{\log_2 M}I_k, & k = \log_2 M.
\end{cases}
\]

\noi
For simplicity, when $k = \log_2M$, 
we set 
\[ I^M_\text{low} =  I^M_{\log_2M}.\]

\noi
In the following, all the definitions depend on this dyadic parameter $M \geq 1$.
For convenience, we set
\[ \Z_M = \Z \cap [\log_2 M, \infty).\]

For $k \in \Z_M$ and $j \geq 0$, let
\[ D_{k, j}^M = \big\{ (n, \tau) \in \Z\times \R: \ n \in I_{k}^M, \ \tau + n^4 \in I_j \big\} \]

\noi
and $D_{k, \leq j}^M = \bigcup_{j' \leq j} D_{k, j'}^M$.
We also define $D_{\leq j}$ by
$D_{\leq j} =  \bigcup_{k \in \Z_M} D_{k, \leq j}^M.$

For $k \in \Z_M$, 
we use $\P_k$ to
denote the projection operator on $L^2(\T)$
defined by
$\ft{\P_k u}(n) = \ind_{I_k^M}(n) \ft{u}(n)$.
Note that $\P_{\log_2 M}$ is the projection onto ``low'' frequencies $\{ |n| \les M\}$.
With a slight abuse of notation, we also use
 $\P_k$ to denote the projection operator on $L^2(\T \times \R)$
given by $\F(\P_k u)(n, \tau) = \ind_{I_k^M}(n) \F(u)(n, \tau)$.
We also set 
\[ \P_{\leq k } = \sum_{\log_2 M \leq \l \leq k}\P_\l
\qquad \text{and}\qquad \P_{> k } = \sum_{ \l > k}\P_\l.\]

Let $\eta_0: \R \to [0, 1]$ be an even smooth cutoff function
supported on $[-\frac{8}{5}, \frac{8}{5}]$
such that $\eta_0 \equiv 1$ on $[-\frac{5}{4}, \frac{5}{4}]$.
We define $\eta$ by $\eta(\xi) = \eta_0( \xi) - \eta_0(2\xi)$,
and set  $\eta_k(\xi) = \eta(2^{-k}\xi)$ for $k \in \Z$.
Namely, $\eta_k$ is supported on
$\big\{ \frac{5}{4} \cdot 2^{k-1}  \leq |\xi|\leq \frac{8}{5} \cdot 2^k \big\}$.
As before, we define
$\eta_{\leq k } = \sum_{\l \leq k} \eta_\l$, etc.

Given a set of indices such as
$j_i$ and $k_i$, $i = 1, \dots, 4$,
we use
 $j^*_i$ and $k^*_i$ to denote the decreasing rearrangements
of these indices.
Also, given a set of frequencies
$n_i$, $i = 1, \dots, 4$,
we use
 $n^*_i$ to  denote the decreasing rearrangements
of $|n_i|$, $i = 1, \dots, 4$.

In the following, we use $S(t)= e^{-i t \dx^4}$
to denote the solution operator to
the linear fourth order Schr\"odinger equation: $i \dt u = \dx^4 u $.
Namely, for $f \in L^2(\T)$, we have
\[ S(t) f = \sum_{n\in \Z} e^{inx - i n^4 t} \ft{f}(n).\]

In performing normal form reductions in Section \ref{SEC:NF}, 
we use the following interaction representation $\u$ (of $u$) on $\T\times \R$:
\begin{align}
\u(t) := S(-t) u(t) = e^{it\dx^4} u(t).
\label{interaction}
\end{align}

\noi
On the Fourier side, we have  $\ft{\u}_n(t) = e^{it n^4} \ft u_n(t)$, $n \in \mathbb{Z}$.
With this notation, 
\eqref{4NLS1} can be written as
\begin{align} \label{4NLS2}
\dt \ft \u_n & = -  i
\sum_{\G(n)}
e^{  - i  \phi(\bar{n})t }
\ft \u_{n_1} \cj{\ft \u_{n_2}}\ft \u_{n_3}
+ i|\ft \u_n|^2 \ft \u_n
 \notag \\
& =: - i\,  \textsf{N}(\u)_n(t) +i\,  \textsf{R}(\u)_n(t),
 \end{align}

\noi
where  
 the plane $\G(n)$ is defined 
\begin{align}
\G(n) 
= \big\{(n_1, n_2, n_3) \in \Z^3:\, 
 n = n_1 - n_2 + n_3 \text{ and }  n_1, n_3 \ne n\big\}
\label{Gam1}
 \end{align}

\noi
and 
the phase function $\phi(\bar{n})$ is defined by 
\begin{align}\label{Phi}
 	\phi(\bar n) & = \phi(n_1, n_2, n_3, n) = n_1^4 - n_2^4 + n_3^4 - n^4\notag\\
&   =  - (n - n_1)(n-n_3) 
\big( n_1^2 +n_2^2 +n_3^2 +n^2 + 2(n_1 +n_3)^2\big).
\end{align}

\noi
Here,  the last equality holds under $n = n_1 - n_2 + n_3$.
See Lemma 3.1 in \cite{OTz}.

We also recall the phase function $\mu(\bar n)$
for the usual Schr\"odinger equation:
\begin{align}\label{Mu}
\mu(\bar{n}):& =  - n_1^2 + n_2^2- n_3^2 + n^2 \notag \\
& = 2(n_2 - n_1) (n_2 - n_3)
= 2(n - n_1) (n - n_3),
\end{align}

\noi
where  the last two equalities hold under $n = n_1 - n_2 + n_3$.

\subsection{Function spaces and their basic properties}
\label{SUBSEC:spaces}

Recall the definition of the standard Sobolev space $H^s(\T)$:
\[ \| f \|_{H^s} = \|\jb{n}^s \ft f(n)\|_{\l^2_n},\]

\noi
where $\jb{\,\cdot\,} = (1+|\cdot|^2)^\frac{1}{2}$.
Given $M \geq1$, we define the  $H^s_M$-norm adapted to the parameter $M\geq 1$
by 
\[ \| f \|_{H^s_M} = \big\|(M^2 + n^2)^\frac{s}{2} \ft f(n)\big\|_{\l^2_n}.\]

\noi
Clearly, 
the $H^s_M$-norm is equivalent to the standard $H^s$-norm.
When $s< 0$, however, it follows from the dominated convergence theorem that
\begin{align}
\lim_{M\rightarrow \infty} \| f \|_{ H_M^s} = 0
\label{Sob2}
\end{align}

\noi
for all $f \in H^s(\T)$. 
This decay property \eqref{Sob2} plays an important role in our analysis.

Next, we define our solution space adapted to this parameter $M \geq 1$.
In \cite{BO1}, Bourgain introduced the dispersive Sobolev space $X^{s, b}(\T\times \R)$ via the norm:
\begin{align}\label{Xsb1}
\| u\|_{X^{s, b}(\T\times \R)} = 
\|\jb{n}^s \jb{\tau + n^4}^b \ft {u}(n, \tau)\|_{\l^2_n L^2_\tau(\Z\times\R)}
= \| \jb{\dx}^s \jb{\dt}^b \u \|_{L^2(\T\times \R)},
\end{align}

\noi
where $\u$ is  the interaction representation defined in \eqref{interaction}.
The Fourier restriction norm method, utilizing the $X^{s, b}$-spaces and their variants,
 has been very effective in 
studying nonlinear evolution equations in low regularity settings.
In the following, we consider  the $X^{s, b}$-spaces adapted to  short time scales
and the parameter $M \geq1$.
When $M = 1$, these spaces were 
 introduced by 
Ionescu-Kenig-Tataru \cite{IKT}
in the context of the KP-I equation.
Also, see Christ-Colliander-Tao
\cite{CCT3} and Koch-Tataru \cite{KT1, KT2} for similar definitions.
While we state some basic properties of these function spaces, 
we refer readers to \cite{GO} for the details of their proofs.

Fix $M \geq 1$.
For  $k \in \Z_M$,  we define the dyadic $X^{s, b}$-type spaces $X_{M, k}$ by 
\begin{align*} 
X_{M, k} = \Big\{
& f_k \in L^2(\Z \times \R): f_k(n,\tau) \text{ is supported on } I^M_k \times \R \notag \\
& \text{ and }
\|f_k \|_{X_{M, k}}:=\sum_{j=0}^\infty
2^\frac{j}{2} \|\eta_j(\tau+n^4) f_k(n ,\tau)\|_{\ell^2_n L^2_{\tau}}<\infty
\Big\}.
\end{align*}

\noi
Then, the following properties hold
for  $X_{M, k}$, $k \in \Z_M$
(with implicit constants independent of $M \geq 1$):

\begin{itemize}
\item[(i)] 
We have
\begin{align*}
\bigg\| \int_{\R} |f_k(n,\tau )| d\tau \bigg\|_{\ell^2_n}
\les
\|f_k\|_{X_{M,k}}
\qquad \text{and}
\qquad
\int_\R \|g_k(n, \tau)\|_{\l^2_n} d\tau \les\|f_k\|_{X_{M, k}}
\end{align*}

\noi
for all $f_k \in X_{M, k}$, 
where $g_k(n, \tau) = f_k(n, \tau - n^4)$.

\smallskip

\item[(ii)]

\noi
For  $k,\ell \in \Z_M$ and $f_k \in X_k$, we have
then
\begin{align}
\sum_{j=\ell+1}^\infty 2^\frac{j}{2}
\bigg\| \eta_j(\tau+n^4) \int_{\R} |f_k(n,\tau')|
\, 2^{-\l}(1+2^{-\l}|\tau-\tau'|)^{-4}d\tau'\bigg\|_{\l^2_n L^2_\tau } &\notag \\
+2^\frac{\ell}{2} \bigg\| \eta_{\leq \l}(\tau+n^4) \int_{\R} |f_k(n,\tau')|
\, 2^{-\l}(1+2^{-\l}|\tau-\tau'|)^{-4}d\tau'\bigg\|_{\l^2_n L^2_\tau }&\les
\|f_k\|_{X_{M, k}},
\label{Xk2}
\end{align}

\noi
where the implicit constant is independent of $k$ and $\l$.
See \cite{Guo1} for the proof.

\smallskip

\item[(iii)]

As a consequence of (ii), we have
\begin{align}
\big\| \F[\gamma(2^\l(t-t_0))\cdot \F^{-1}(f_k)]\big\|_{X_{M, k}}
\les \|f_k\|_{X_{M, k}}
\label{Xk3}
\end{align}

\noi
for $k,\ \l\in \Z_M$, $t_0\in \R$, $f_k \in X_{M, k}$ and
$\g\in \mathcal{S}(\R)$,
where  the implicit constant in \eqref{Xk3}
is also independent of $k, \l$, and $t_0$.

\end{itemize}

Next, we consider the time localization of the $X_{M, k}$-space	
onto the  time scale $\sim 2^{-[\alpha k]}$, where
$\alpha>0$ is to be determined later.
Here, $[x]$ denotes the integer part of $x$.
For $k\in \Z_M$ we define the
spaces $F_{M, k}^\al$ and $N^\al_{M, k}$ by
\begin{align*}
F_{M, k}^\al=\Big\{
& u \in L^2(\T\times\R): \ft{u }(n,\tau) \text{ is supported in }I_k^M\times\R \\
& \text{ and }\|u\|_{F_{M, k}^\al}=\sup_{t_k\in \R}
\big\|\F[  \eta_0(2^{[\alpha k]}(t-t_k))\cdot u ]\big\|_{X_{M, k}}<\infty
\Big\},
\\
N_{M, k}^\al=\Big\{
& u \in L^2(\T\times\R): \ft{u}(n,\tau) \text{ is supported in }{I}_{k}^M\times\R \\
& \text{ and }
\|u\|_{N_{M, k}^\al}=\sup_{t_k\in\R}
\big\|(\tau+n^4+i2^{[\alpha k]})^{-1}\F[ \eta_0(2^{[\alpha k]}(t-t_k)) \cdot u ]\big\|_{X_{M, k}}<\infty
\Big\}.
\end{align*}

\noi
Given $T> 0$,  we define the time restriction spaces
 $F_{M, k}^\alpha(T)$ and $N_{M, k}^\alpha(T)$ by
\begin{align*}
F_{M, k}^\alpha(T)
& =\Big \{ u \in C([-T,T];L^2(\T)):
\|u\|_{F_{M, k}^\al(T)}=\inf_{\wt{u}=u \text{ on } \T\times [-T,T]}\|\wt u \|_{F_{M, k}^\alpha}\Big\}, \\
N_{M, k}^\alpha(T)
& =\Big\{ u \in C([-T,T];L^2(\T)):
\|u\|_{N_{M, k}^\alpha(T)}=\inf_{\wt{u}= u \text{ on } \T\times [-T,T]}\|\wt u\|_{N_{M, k}^\alpha}\Big\}.
\end{align*}

\noi
Here, the infimum is taken over all extensions $\wt u \in C_0(\R; L^2(\T))$.

We finally define our solution space $F^{s, \al}_M(T)$
and its dual space $N^{s, \al}_M(T)$
by putting together the dyadic spaces defined above
via  the Littlewood-Paley decomposition.
For $s\in \R$, $\al > 0$,  and $T> 0$, we define the spaces
$F^{s, \al}_M(T)$ and $N^{s, \al}_M(T)$ by
\begin{align*}
F^{s, \al}_M(T)
& = \Big\{ u:
\|u\|_{F^s_M(T)}^2=\sum_{k\in \Z_M}2^{2ks}\|\P_k u\|_{F_{M, k}^\alpha(T)}^2<\infty \Big\}, \\
N^{s, \al}_M(T)
& =\Big \{ u:
\|u\|_{N^s_M(T)}^2=\sum_{k\in \Z_M} 2^{2ks}\|\P_k u\|_{N_{M, k}^\alpha(T)}^2<\infty \Big\}.
\end{align*}

\noi
Here, $\al = \al (s) > 0$ is a parameter to be chosen later.
See Subsection \ref{SUBSEC:alpha}.
When $M = 1$, we simply drop the subscript $M$
from the function spaces
and use $F^{s, \al}(T)$, etc.

In order to handle the short-time structure embedded in the definitions of 
$F^{s, \al}_M(T)$ and $N^{s, \al}_M(T)$, 
we define the corresponding energy space $E^s_M(T)$
by 
\begin{align}
\|u\|_{E^s_M(T)}^2= 
\sum_{k\in \Z_M} \sup_{t_k\in [-T,T]}2^{2ks}\|\P_k u(t_k)\|_{L^2(\T)}^2
\label{Es}
\end{align}

\noi
for  $u\in C([-T,T];H^{\infty}(\T))$.
While the definition of  $E^s_M(T)$ depends on $M \geq 1$, 
it is independent of the parameter $\alpha>0$.
This space is essentially the usual energy space $C([-T,T];H^s_M(\T))$ but
with a logarithmic difference.
See Subsection \ref{SUBSEC:Es}.

\medskip

We conclude this subsection by recalling some basic lemmas
from \cite{GO}.
While these properties are stated and proved for $M = 1$ in \cite{GO}, 
 a straightforward modification yields the corresponding statements for $M \geq 1$
 (with implicit constants independent of $M \geq 1$).
In the following, we fix  $M \geq 1$ and $\al > 0$.

The first lemma shows that 
a smooth time cutoff supported on an interval of size $\sim 2^{-[\al k]}$
acts boundedly on $N^\al_{M, k}$.

\begin{lemma}\label{LEM:embed3}
Let  $k\in \Z_M$, $t_k\in \R$, and $\g \in \mathcal{S}(\R)$.
Then, we have 
\begin{align*}
\big\| (\tau + n^4 + i 2^{[\al k]})^{-1}\F [
\g(2^{[\al k]}(t-t_k)) \, \cdot \, & \F^{-1}(f_k)]\big\|_{X_{M, k}} \notag \\
& \les \big\|(\tau + n^4 + i 2^{[\al k]})^{-1} f_k\big\|_{X_{M, k}}
\end{align*}

\noi
for $f_k$ supported on $I^M_k \times \R$.
Here, the implicit constant is independent
of $M$, $\al$, $k$, and $t_k$.
\end{lemma}

The second lemma shows that 
$F^{\al}_k$- and
$F^{s, \al}$-norms control the supremum in time
(of the appropriate spatial norms).

\begin{lemma}
\label{LEM:embed1}

\textup{(i)}
Let $u$ be a function on $\T\times \R$ such that
$\supp \ft u  \subset I_k^M \times \R$, $k \in \Z_M$.
Then, we have
\begin{align*}
\| u \|_{L^\infty_t L^2_x} \les \|u\|_{F^\al_{M, k}}.
\end{align*}

\noi
Similarly, we have
\begin{align}
\big\| \F^{-1}[\eta_{\leq j}(\tau + n^4) \ft{u}\, ] \big\|_{L^\infty_t L^2_x}
\les \|u\|_{F^\al_{M, k}}
\label{infty2}
\end{align}

\noi
for any $j \in \Z_{\geq 0}$.
Here, \eqref{infty2} also holds
when we replace $\eta_{\leq j}$ by $\eta_j$ or $\eta_{> j}$.

\smallskip

\noi
\textup{(ii)}
Let $s\in \R$ and  $T> 0$.
Then, we have
\begin{equation*}
\sup_{t\in [-T,T]}\|u(t)\|_{H^s_M} \les \|u\|_{F^{s,\alpha}_M(T)}.
\end{equation*}

\end{lemma}

In the following, we define the corresponding function spaces
with the temporal
regularity~$b$.
For $k \in \Z_M$ and $b \in \R$,
define  $X_{M, k}^b$ by
\begin{align*}
 \|f_k \|_{X_{M, k}^b}:=\sum_{j=0}^\infty
2^{jb} \|\eta_j(\tau+n^4) f_k(n ,\tau)\|_{\ell^2_n L^2_{\tau}}
\end{align*}

\noi
for $f_k $ supported on $I_k^M \times \R$.
Note that  $X_{M,k} = X^\frac{1}{2}_{M,k}$.
Then, we define the spaces $F^{ b, \al}_{M, k}$ and $F^{s, b,  \al}_M(T)$
with $X_{M, k}^b$,
just as we defined $F^\al_{M, k}$ and $F^{s,  \al}(T)$ with $X_{M, k}$.

The following lemma 
allows us to gain  a small power of time localization
at a slight expense of the regularity in modulation.

\begin{lemma} \label{LEM:timedecay}
Let $T > 0$ and $b < \frac{1}{2}$.
Then, we have \begin{align*}
\|\P_k u \|_{F^{b, \al}_{M, k}}
\les T^{\frac{1}{2} - b -} \|\P_k u \|_{F^{\al}_{M, k}}
\end{align*}

\noi
for any function $u$ supported on $\T\times [-T, T].$
\end{lemma}

The following lemma shows that the multiplication by 
a sharp cutoff function in time is ``almost'' bounded in $X_{M, k}$.
\begin{lemma} \label{LEM:sup}
Let $k \in \Z_M$.
Then, for any interval $I = [t_1, t_2] \subset \R$, we have 
\begin{align*}
 \sup_{j\in\Z_{\geq 0}} 2^\frac{j}{2}
\big\|\eta_j (\tau + n^4)\F[\ind_{I}(t) \cdot  \P_k u ]  \big\|_{\l^2_n L^2_\tau}
\les \| \F(\P_k u)\|_{X_{M, k}}, 
\end{align*} 

\noi
where the implicit constant is independent of $M$, $k$ and $I$.
\end{lemma}

Lastly,  we state a linear estimate
associated with the fourth order Schr\"odinger equation.

\begin{lemma} \label{LEM:linear}
Let  $T> 0$.
Suppose that  $u \in C([-T,T];H^\infty(\T))$
is a solution to the following nonhomogeneous linear fourth order  Schr\"odinger equation:
\begin{align*}
i\dt u-\dx^4 u=v \qquad \mbox{on } \ \T\times (-T,T),
\end{align*}

\noi
where $v \in C([-T,T];H^\infty(\T))$.
Then, for any $s\in \R$  and $\alpha\geq 0$,
we have
\begin{align*}
\|u\|_{F^{s,\alpha}_M(T)}\les \
\|u\|_{E^{s}_M(T)}+\|v\|_{N^{s,\alpha}_M(T)}.
\end{align*}
\end{lemma}

\subsection{On the energy space}
\label{SUBSEC:Es}

As pointed out above, 
the energy space $E^s_M(T)$ defined in \eqref{Es}
is essentially 
 the usual energy space $C([-T,T];H^s_M(\T))$ but
there is a  logarithmic difference that we need to handle.
In this subsection, we introduce a sequence  $\{a_{k_0}\}_{k_0\in \Z_M}$
of symbols  that allows us to control the $E^s_M$-norm.
Similar symbols have been used in \cite{KT1, KT2}.

Fix $k_0 \in \Z_M$.
For sufficiently small $\dl_0 = \dl_0(s)> 0$ (to be chosen later\footnote{See Proposition \ref{PROP:energy} below.}), we define 
a symbol $a^0_{k_0}$ on $\R$ by setting
\begin{align}
 a^0_{k_0}(\xi) = |\xi|^{2s} \min \bigg\{\frac{|\xi|}{2^{k_0}}, \frac{2^{k_0}}{|\xi|}\bigg\}^{\dl_0}
 \label{Es2}
\end{align}

\noi
for $|\xi|  = 2^k $ with $k \in \Z_M$
and we extend the definition of $a^0_{k_0}$ onto $\R$ by linear interpolation.
In particular, it is constant on $[-M, M]$.
As it is,  $a^0_{k_0}$ is not smooth
and thus we need to smooth it out.

Let $\eta_0: \R \to [0, 1]$ be a smooth cutoff function
with  $\eta_0 (\xi) \equiv 1$ for $|\xi|\leq \frac{5}{4}$
and $= 0$ for $|\xi| \geq  \frac{8}{5}$
as above.
Then, choose $c_0 > 0$ 
such that $\int c_0 \eta_0(\xi)  d\xi = 1$.
Given $k \in \Z_M$, 
we define a symbol $a_{k_0}$ in a neighborhood of a dyadic point $2^k$ by 
\[ a_{k_0}(\xi) = (a_{k_0}^0 * \theta_k)(\xi) \qquad \text{on } J_k : = \big\{ \xi \in \R: |\xi - 2^k| \leq \tfrac 14\cdot  2^k\big\},\]

\noi
where $\theta_k(\xi) =  \frac{10 c_0}{ 2^k} \eta_0\big(\frac{10}{2^k}\xi\big)$.
For $\xi \notin \bigcup_{k \in \Z_M} J_k$, 
we set 
$a_{k_0} (\xi) = a^0_{k_0} (\xi)$.
Then, the symbol  $a_{k_0}$ satisfies the following properties:
\begin{itemize}
\item[(i)] For  $\g = 1, 2$, we have
\begin{align}
 |\dd^\g_\xi a_{k_0}(\xi) | \les a_{k_0}(\xi) \cdot (M^2 + \xi^2)^{-\frac{\g}{2}}.
 \label{Es2a}
\end{align}

\item[(ii)] For $|\xi| \le \frac M 2$ and $k_0 \in \Z_M$, we have
\begin{align}
 a_{k_0}(\xi) = a_{k_0} (0) \sim M^{2s+\dl_0} 2^{-\dl_0 k_0}. 
\label{Es2b}
 \end{align}

\item[(iii)] For $\xi \in I_k^M$, we have 
\[a_{k_0}(\xi) \sim 2^{2ks} 2^{-\dl_0|k-k_0|}. \]

\end{itemize}

\smallskip

\noi
As a consequence of (ii) and (iii), we have, for $|\xi|\sim | \xi'|$, 
\begin{align}
 a_{k_0}(\xi) \sim a_{k_0}(\xi').
\label{Es3}
\end{align}

Next, we define a sequence $\{ E_{k_0}\}_{k_0 \in \Z_M}$ 
of energy functionals by 
\begin{align}
E_{k_0}(u)(t) = \jb{a_{k_0}(D) u(t), u(t)}_{L^2}
= \sum_{n \in \Z}a_{k_0}(n) |\ft u(n, t)|^2.
\label{Es4}
\end{align}

\noi
Then, from \eqref{Es2} and \eqref{Es3}, we have
\begin{align}
 2^{2k_0s}\|\P_{k_0} u(t)\|_{L^2(\T)}^2 \les  E_{k_0}(u)(t).
\label{Es5}
\end{align}

\noi
In particular, 
from \eqref{Es} and \eqref{Es5}, we have
\begin{align}
\| u\|_{E^s_M(T)}^2 
\les 
\sum_{k_0\in \Z_M} \sup_{t_{k_0}\in [-T,T]} E_{k_0}(u) (t_{k_0})
\label{Es6}
\end{align}

\noi
for any $T > 0$.
In Section \ref{SEC:energy}, 
we establish the desired energy estimate \eqref{Aenergy}
by estimating 
$\sup_{t_{k_0}\in [-T,T]} E_{k_0}(u) (t_{k_0})$
in a summable manner over $k_0 \in \Z_M$.

\subsection{On the choice of $\alpha$} 
\label{SUBSEC:alpha}
In Subsection \ref{SUBSEC:spaces}, we defined the function spaces
$F^{s, \al}_M(T)$ and $N^{s, \al}_M(T)$ depending on a parameter $\al > 0$.
In this subsection, we provide a heuristic discussion 
on how to choose $\al > 0$. 
In fact, we choose the smallest $\al > 0$ so that 
 a solution to \eqref{4NLS1} localized
around the spatial frequencies $\{ |n| \sim 2^k\}$
behaves like a linear solution up to time
 $\sim 2^{- \al k}$.\footnote{Namely,  $2^{- \al k}$ is the first time scale on which the nonlinear 
effect  becomes visible.}
In the following, we set $M = 1$ for simplicity.

Fix $ s< 0$ and $k \in \Z_{\geq 0}$. 
Let $f \in L^2(\T)$ with $\supp \ft f \subset I_k$ such that $\|f \|_{H^s} = 1$.
Then, we have $\|f \|_{L^2}\sim 2^{-ks}$.
Let $u$ be the (smooth) solution to \eqref{4NLS1} with $u|_{t = 0} = f$,
satisfying the following Duhamel formulation:
\[
u(t) = S(t) f  - i  \int_0^t
S(t - t') \Nf(u) (t')  dt',
\]

\noi
where $\Nf(u)$ is the nonlinear part of \eqref{4NLS1} defined in \eqref{Xnonlin}.
We investigate the largest time scale $T$ such that 
 $u(t) \approx S(t) f$ on $[0, T]$.
By the standard $X^{s, b}$-estimates and the $L^4$-Strichartz estimate \eqref{L4}
as in \cite{OTz}, we have 
\begin{align*}
\bigg\|  \int_0^t
S(t - t') \Nf(u) & (t')  dt'\bigg\|_{L^\infty_T L^2_x}
 \les \bigg\| \int_0^t
S(t - t') \Nf(u) (t')  dt'\bigg\|_{X^{0, \frac{1}{2}+}_T}
 \les \big\||u|^2 u \big\|_{X^{0,-\frac12+}_T} \notag\\
& =   \sup_{\|v\|_{X^{0, \frac 12 - }} = 1} \bigg|\int_{\T \times [0, T]}  v |u|^2 u dx dt\bigg|
=   \sup_{\|v\|_{X^{0, \frac 12 - }} = 1} \| v\|_{L^4_{x, T}}\| u \|_{L^4_{x, T}}^3\notag\\
& \les T^{\frac34-} \|u\|_{X^{0,\frac12+}_T}^3
\intertext{By making a heuristic substitution $u(t) \approx S(t) f$,}
& \les T^{\frac34-} \|f\|_{L^2}^3 \sim T^{\frac34-}  2^{-3ks}. 
\end{align*}

\noi
Here, $X^{s, b}_T$ denotes the local-in-time version of 
the $X^{s, b}$-space restricted on the time interval $[0, T]$.
This shows that the solution $u$ basically propagates
linearly on the time scale $T$
if $T^{\frac34-}  2^{-3ks} \ll 2^{-ks}$,
i.e.~ $T \ll 2^{-\al k}$
with \begin{align}
\label{al}
\al = -\frac{8s}3 + \eps
\end{align}

\noi
for some small $\eps > 0$.
Indeed, the condition \eqref{al} on $\al$ 
naturally appears
in establishing the crucial trilinear estimate.
See Section \ref{SEC:trilinear}.

\section{Strichartz and related multilinear estimates}
\label{SEC:Strichartz}

In this section, we state and prove certain multilinear Strichartz estimates.
While the basic structure of the argument follows closely
that in \cite{GO},  we obtain stronger estimates with simpler proofs
thanks to the stronger quartic dispersion.

Recall  the following periodic $L^4$- and $L^6$-Strichartz estimates:
\begin{equation}\label{Stri1}
\|u\|_{L^4_{x, t}(\T\times \R)} \lesssim \|u\|_{X^{0, \frac{5}{16}}}
\quad \text{and} \quad
\|S(t) \phi \|_{L^6_{x, t}(\T\times \R)} \leq C_\eps  |I|^\eps \|\phi\|_{L^2}
\end{equation}

\noi
for any $\eps > 0$,
where $\phi$ is a function on $\T$ such that $\supp \ft{\phi}$
is contained in an interval $I$ of length $|I|$.
These estimates are essentially due to Bourgain \cite{BO1}.
See \cite{OTz} for the proof of the $L^4$-Strichartz estimate.
The $L^6$-Strichartz estimate follows
from the algebraic identity \eqref{Phi}
and the divisor counting argument as in \cite{BO1}.

By the Galilean transformation
and the transference principle, 
we have the following estimate;
if   we assume that
$\supp \ft u \subset
D_{\leq j}\cap( I\times \R) $ for some interval $I$,  
then we have
\begin{align}
 \quad \|u \|_{L_{x,t}^6}\leq C_\eps |I|^\eps 2^{\frac{j}{2} }\|u \|_{L^2_{x, t}}
\label{Stri3}
\end{align}

\noi
for any $\eps > 0$.
As a corollary to \eqref{Stri1} and \eqref{Stri3}, we have the following lemma.
See  \cite{GO} for the proofs.

\begin{lemma}\label{LEM:L4L6}
Let $u_{k_i, j_i}$ be a function on $\T \times \R$
such that $\supp \ft{u}_{k_i, j_i} \subset D_{k_i, \leq j_i}^M$.
Then, we have 
\begin{align}
& \bigg|\int_{\T\times \R }  u_{k_1, j_1} \cj u_{k_2, j_2}u_{k_3, j_3}\cj u_{k_4, j_4} dx dt\bigg|
\lesssim \prod_{i = 1}^4 2^\frac{5 j_i}{16}
\|\F(u_{k_i, j_i})\|_{\l^2_n L^2_\tau},  \label{L4_0}\\
& \bigg|\int_{\T\times \R}  u_{k_1, j_1} \cj u_{k_2, j_2}u_{k_3, j_3}\cj u_{k_4, j_4} dx dt\bigg|
\lesssim  2^{-\frac{ j_1^*}{2}}2^{\eps k_3^*}\prod_{i = 1}^4 2^\frac{ j_i}{2} \|\F(u_{k_i, j_i})\|_{\l^2_n L^2_\tau}
\label{L6_0}
\end{align}

\noi
for any $\eps > 0$.
Here, $j^*_i$ and $k^*_i$ denote the decreasing rearrangements
of $j_i$ and $k_i$, $i = 1, \dots, 4$.

\end{lemma}

As in \cite{GO}, we can refine the analysis
and obtain the following multilinear estimates.

\begin{lemma} \label{LEM:L64}

Let  $u_i$ be supported in $D^M_{k_i,\leq j_i}$, $i = 1, 2, 3$.
Suppose that  $2^{k_1^*} \gg M $.
Then, the following estimate holds:
\begin{align}\label{aL64}
\|\P_{k_4}\N(u_1, u_2, u_3)\|_{L^2_{x, t} }
\les 
2^{\frac{k_4^*}2}
\min_{i = 1,  3}\big\{ (1+2^{j_i-2k_1^*})^\frac{1}{2}2^\frac{-j_i}{2}\big\}
\bigg(\prod_{i=1}^3 2^\frac{j_i}{2}\|\F(u_i)\|_{\l^2_n L^2_\tau}\bigg),
\end{align}

\noi
Here, $\N(u_1, u_2, u_3)$
is the non-resonant part of the nonlinearity defined in \eqref{NN1}.

\end{lemma}

The proof of Lemma \ref{LEM:L64} is analogous to that of Lemma 5.3 in \cite{GO}.
Note that, thanks to the stronger dispersion, 
we do not need frequency separation which was assumed 
 in Lemma 5.3 in \cite{GO}.

\begin{proof}
Let $f_i = \ft{u}_i $ for $i = 1, 3$
and $f_i = \cj{\ft{u}}_i $ for $i = 2$.
By duality, we have 
\begin{align}
\text{LHS of } \eqref{aL64}
=  \sup_{\substack{\|f_4\|_{L^2} = 1\\
\supp f_4 \subset I^M_{k_4}\times \R}}\intt_{\tau_1 - \tau_2 + \tau_3 - \tau_4 = 0}\sum_{\substack{n_1 - n_2 + n_3 -n_4 = 0\\n_1 \ne n_2, n_4}}
\prod_{i = 1}^4 f_i(n_i, \tau_i) d\tau_1d\tau_2d\tau_3.
\label{aL641}
\end{align}

\noi
For simplicity of notations, 
 we drop the supremum over $f_4$
in the following.
Note that, under the assumption $2^{k_1^*} \gg M $, 
we have
\begin{align*}
n_1^* \sim 2^{k_1^*}.
\end{align*}

\noi
See Remark \ref{REM:lowfreq} below.

\medskip

\noi
$\bullet$ {\bf Case (a):} $|n_4| \les  2^{k_4^*} $.
\\
\indent
Under $n_1 - n_2 + n_3 -n_4 = 0$, 
we have 
$\max\{|n_2|, |n_3|\} \sim 2^{k_1^*}$.
With $g_i(n, \tau) = f_i (n, \tau - n^4)$, we have
\begin{align}
\eqref{aL641}
& \le \int
\sum_{n_4} |g_4(n_4, \tau_4)|
\sum_{n_1, n_2} |g_1(n_1, \tau_1)|
|g_2(n_2, \tau_2)| \notag \\
& \hphantom{XXXXXX} \times
 \big| g_3\big(-n_1 + n_2 + n_4,
 h_3(n_1, n_2, n_4, \tau_1, \tau_2, \tau_4) \big) \big|d\tau_1d\tau_2d\tau_4,
\label{aL641a}
\end{align}

\noi
where $h_3(n_1, n_2, n_4, \tau_1, \tau_2, \tau_4)$ is defined by
\begin{equation}
 h_3(n_1, n_2, n_4, \tau_1, \tau_2, \tau_4)
  = -\tau_1+\tau_2 + \tau_4 +  n_1^4 - n_2^4 - n_4^4 + (-n_1 + n_2 + n_4)^4.
\label{aL641b}
\end{equation}

\noi
For fixed $n_1, n_4, \tau_1, \tau_2$, and $\tau_4$,
define the set $E_{32}= E_{32}(n_1, n_4, \tau_1, \tau_2, \tau_4)$ by
\begin{align*}
E_{32} = \{ n_2 \in \Z: h_3(n_1, n_2, n_4, \tau_1, \tau_2, \tau_4) = O(2^{j_3})\}.
\end{align*}

\noi
Since $n_1 \ne n_4$, we have 
\begin{align*}
|\partial_{n_2} h_3 |
& = 4|- n_2^3 + (-n_1 + n_2 + n_4)^3| \\
& = 4|(-n_1+n_4)(n_2^2 + n_2(-n_1+n_2+n_4)+(-n_1+n_2+n_4)^2)| \\
& \ges \max \{ n_2^2 , n_3^2\}
\sim 2^{2k_1^*},
\end{align*}

\noi
where the second to the last step follows from completing a square:
\[n_2^2 + n_2n_3 +n_3^2
= (n_2 + \tfrac 12 n_3)^2 + \tfrac 34n_3^2
=  \tfrac 34n_2^2 +  ( \tfrac 12 n_2 + n_3)^2 .\]

\noi
Hence, 
we conclude that 
\begin{align}
|E_{32}| \lesssim 1+ 2^{j_3 - 2k_1^*}.
\label{aL64b}
\end{align}

Now we are ready to estimate \eqref{aL641}.
By
Cauchy-Schwarz inequality in $n_2, n_1, n_4$, 
we obtain
\begin{align*}
\eqref{aL641a}
& \les 
(1+ 2^{j_3 - 2 k_1^*})^\frac{1}{2}
\int \sum_{n_4} |g_4(n_4, \tau_4)|
\sum_{n_1} |g_1(n_1, \tau_1)|
\bigg(\sum_{n_2} |g_2(n_2, \tau_2)|^2 \notag \\
& \hphantom{XXXXXX} \times
 \big| g_3\big(-n_1 + n_2 + n_4,
 h_3(n_1, n_2, n_4, \tau_1, \tau_2, \tau_4) \big)\big|^2\bigg)^\frac{1}{2} d\tau_1d\tau_2d\tau_4\\
& \les
2^\frac{k_4}{2}\cdot 
(1+ 2^{j_3 - 2 k_1^*})^\frac{1}{2} 
 \|g_4(n_4, \tau_4)\|_{\l^2_{n_4}L^2_{\tau_4}}
\sup_{n_4} \int \|g_1(n_1, \tau_1)\|_{\l^2_{n_1}}
\bigg(\sum_{n_2} |g_2(n_2, \tau_2)|^2 \notag \\
& \hphantom{XXXXXX} \times
\sum_{n_1} \big\|  g_3\big(-n_1 + n_2 + n_4,
 h_3(n_1, n_2, n_4, \tau_1, \tau_2, \tau_4) \big)\big\|^2_{L^2_{\tau_4}}
 \bigg)^\frac{1}{2} d\tau_1d\tau_2\\
\intertext{Noting that $h_3$ is linear in $\tau_4$
and applying Cauchy-Schwarz inequality in $\tau_1$ and $\tau_2$,} 
& \leq
2^\frac{k_4}{2}\cdot 
(1+ 2^{j_3 - 2k_1^*})^\frac{1}{2} 
\bigg(\prod_{i_1 = 1}^2  \|g_{i_1}\|_{L^1_{\tau_{i_1}} \l^2_{n_{i_1}}}\bigg)
\bigg(\prod_{i_2 = 3}^4
 \|g_{i_2}\|_{\l^2_{n_{i_2}}L^2_{\tau_{i_2}}}\bigg)\\
& \les 
2^\frac{k_4}{2}\cdot 
(1+ 2^{j_3 - 2k_1^*})^\frac{1}{2}2^{-\frac{j_3}{2}}
\prod_{i = 1}^3 2^\frac{ j_i}{2} \|f_i\|_{\l^2_n L^2_\tau}, 
\end{align*}

\noi
 yielding \eqref{aL64}.
Note that even if we  replace the role of $n_1$ and $n_3$, the same argument still holds
with a factor $(1+ 2^{j_1 - 2k_1^*})^\frac{1}{2}2^{-\frac{j_1}{2}}$.
The same comment applies to Cases (b) and (c).

\medskip

\noi
$\bullet$ {\bf Case (b):} $|n_1| \sim 2^{k_4^*}$. (A similar argument applies to the case 
$|n_3| \sim2^{ k_4^*}$.)
\\
\indent
In this case, we have  $\max\{|n_2|, |n_3|\} \sim 2^{k_1^*}$
and thus \eqref{aL64b} holds.
Then, proceeding as before, we have
\begin{align*}
\eqref{aL641a}
& \les 
(1+ 2^{j_3 - 2k_1^*})^\frac{1}{2}
\int \sum_{n_1} |g_1(n_1, \tau_1)|
\sum_{n_4} |g_4(n_4, \tau_4)|
\bigg(\sum_{n_2} |g_2(n_2, \tau_2)|^2 \notag \\
& \hphantom{XXXXXX} \times
 \big| g_3\big(-n_1 + n_2 + n_4,
 h_3(n_1, n_2, n_4, \tau_1, \tau_2, \tau_4) \big)\big|^2\bigg)^\frac{1}{2} d\tau_1d\tau_2d\tau_4\\
& \les 2^{\frac {k_1}2} \cdot
(1+ 2^{j_3 - 2k_1^*})^\frac{1}{2} 
\int \|g_1(n_1, \tau_1)\|_{\l^2_{n_1}} d\tau_1
 \|g_4(n_4, \tau_4)\|_{\l^2_{n_4}L^2_{\tau_4}}
\sup_{n_1, \tau_1} \int \bigg(\sum_{n_2} |g_2(n_2, \tau_2)|^2 
\notag \\
& \hphantom{XXXXXX} \times
\sum_{n_4} \big\|  g_3\big(-n_1 + n_2 + n_4,
 h_3(n_1, n_2, n_4, \tau_1, \tau_2, \tau_4) \big)\big\|^2_{L^2_{\tau_4}}
 \bigg)^\frac{1}{2} d\tau_2.
\end{align*}

\noi
The rest follows as in Case (a).

\medskip

\noi
$\bullet$ {\bf Case (c):} $k_2 = k_4^*$.
\\
\indent
In this case, we have $\max\{|n_1|, |n_3|\} \sim 2^{k_1^*}$.
For fixed $n_2, n_4, \tau_1, \tau_2$, and $\tau_4$,
define the set $E_{31}= E_{31}(n_2, n_4, \tau_1, \tau_2, \tau_4)$ by
\begin{align*}
E_{31} = \{ n_1 \in \Z: h_3(n_1, n_2, n_4, \tau_1, \tau_2, \tau_4) = O(2^{j_3})\},
\end{align*}

\noi
where $h_3$ is as in \eqref{aL641b}.
Note that 
\begin{align*}
|\partial_{n_1} h |
& = 4|n_1^3 - 4(-n_1 + n_2 + n_4)^3| \\
& = 4|(2n_1-n_2-n_4)(n_1^2 + n_1(-n_1+n_2+n_4)+(-n_1+n_2+n_4)^2)|\\
& \sim |n_1-n_3| \cdot 2^{2k_1^*}.
\end{align*}

\noi
If $n_1 = n_3$, 
then we have  $n_1 = \frac{n_2+n_4}2$.
Namely, 
$n_1$ is uniquely determined for fixed $n_2$ and $n_4$
and hence we have $|E_{31}| = 1$.
Otherwise, we have $|\partial_{n_1} h | \gtrsim 2^{2k_1^*}$. 
Therefore, we conclude that
\begin{align*}
|E_{31}| \lesssim 1+ 2^{j_3 - 2k_1^*}.
\end{align*}

\noi
Then, proceeding as before, we have
\begin{align*}
\eqref{aL641a}
& \les 
(1+ 2^{j_3 - 2k_1^*})^\frac{1}{2}
\int \sum_{n_4} |g_4(n_4, \tau_4)|
 \sum_{n_2} |g_2(n_2, \tau_2)|
\bigg(\sum_{n_1} |g_1(n_1, \tau_1)|^2 \notag \\
& \hphantom{XXXXXX} \times
 \big| g_3\big(-n_1 + n_2 + n_4,
 h_3(n_1, n_2, n_4, \tau_1, \tau_2, \tau_4) \big)\big|^2\bigg)^\frac{1}{2} d\tau_1d\tau_2d\tau_4\\
& \les 2^{\frac {k_2}2} \cdot
(1+ 2^{j_3 - 2k_1^*})^\frac{1}{2} 
\int \|g_2(n_2,  \tau_2)\|_{\l^2_{n_2}} d\tau_2
 \|g_4(n_4, \tau_4)\|_{\l^2_{n_4}L^2_{\tau_4}}
\sup_{n_2, \tau_2} \int \bigg(\sum_{n_1} |g_1(n_1, \tau_1)|^2 
\notag \\
& \hphantom{XXXXXX} \times
\sum_{n_4} \big\|  g_3\big(-n_1 + n_2 + n_4,
 h_3(n_1, n_2, n_4, \tau_1, \tau_2, \tau_4) \big)\big\|^2_{L^2_{\tau_4}}
 \bigg)^\frac{1}{2} d\tau_1.
\end{align*}

\noi
Then, the rest follows as before.
\end{proof}

As a corollary to Lemma \ref{LEM:L64}, 
 we obtain  the following multilinear estimates
by further assuming $j_i \ge [\al k_1^*]$, $i = 1,2,3$, and $\alpha \in [0,2]$.

\begin{lemma}\label{LEM:L62}
Let $\al \in [0, 2]$.
Let $u_i$ be a function on $\T \times \R$
such that $\supp \ft{u}_{i} \subset D^M_{k_i, j_i}$ and $2^{k_1^*}\gg M$.
Suppose that  $j_1, j_2, j_3 \geq [\al k_1^*]$.
Then, we have
\begin{align}
& \bigg|\int_{\T\times \R}  \N(u_{1},  u_{2}, u_{3})\cdot \cj u_{4} dx dt\bigg|
\lesssim  2^{-\frac{ j_1^*}{2}}2^{- \frac{1}{2} \al k_1^* + \frac{1}{2}k_4 ^*}\prod_{i = 1}^4 2^\frac{ j_i}{2} \|\F(u_i)\|_{\l^2_n L^2_\tau}.
\label{L62a}
\end{align}

\end{lemma}

 When  
 $j_4\ges j_1^*$, by noting that $2^{-\frac{j_4}{2}} \les 2^{-\frac{ j_1^*}{2}}$, 
the desired estimate \eqref{L62a} directly  follows from Lemma \ref{LEM:L64}.
 When  $j_4\ll j_1^*$,
 we first rewrite the left-hand side of \eqref{L62a} as 
\begin{align*}
\bigg|\int_{\T\times \R}  \N(u_{1},  u_{2}, u_{3})\cdot \cj u_{4} dx dt\bigg| =
\bigg|\int_{\T\times \R}  \N(u_{i_1},  u_{i_2}, u_{i_3})\cdot \cj u_{i_4} dx dt\bigg|,
\end{align*}

\noi
where $(i_1, i_2, i_3, i_4) = (2, 3, 4, 1)$,  $(3, 4, 1, 2)$, or $(4, 1, 2, 3)$
such that $j_{i_4} \ges j_1^*$.
Then,  \eqref{L62a} in this case also  follows from Lemma \ref{LEM:L64}.

\begin{remark}
\rm
\label{REM:lowfreq}
The assumption $2^{k_1^*}\gg M$ is necessary in Lemmas \ref{LEM:L64} and \ref{LEM:L62}. 
In fact, when  $2^{k_1^*} =  M$, we only know that 
$n_1^*$ belongs to the interval $I^M_\text{low}$
but it is possible to have   $n_1^* \ll 2^{k_1^*}= M$. 
We also point out  that  
Lemmas \ref{LEM:L64} and \ref{LEM:L62} also hold
under an alternative assumption:
 $n_1^* \sim 2^{k_1^*}$.
This observation plays an important role in the energy estimate
in Section \ref{SEC:energy}, where we apply symmetrization to eliminate the contribution from 
the low frequencies
$\big\{(n_1,n_2,n_3,n_4): n_1^* \le \frac{M}2 \big\}$.
\end{remark}

 \medskip

We conclude this section by stating a multilinear estimate
when there is a gap between the two largest (spatial) frequencies and the rest.

\begin{lemma}\label{LEM:L63}
Let $\al \in [0, 2]$.
Let $u_i$ be a function on $\T \times \R$
such that $\supp \ft{u}_{i} \subset D^M_{k_i,  j_i}$.
Suppose that $k_3, k_4 \leq k_2^* - 10$, $j_1, j_2, j_3 \geq [\al k_1^*]$,
and $2^{j_1^*} \geq |\phi (\bar n)|$,
where $\phi(\bar n )$ is the phase function defined in \eqref{Phi}.
Then, we have
\begin{align}
& \bigg|\int_{\T\times \R}  \N(u_{1}, u_{2}, u_{3})\cdot  \cj u_{4} dx dt\bigg|
\lesssim
\Ld \cdot \prod_{i = 1}^4 2^\frac{ j_i}{2} \|\F(u_i)\|_{\l^2_n L^2_\tau},
\label{L63}
\end{align}

\noi
where $\Ld$ is given by 
\[ \Ld = \begin{cases}
2^{-\frac{ 1}{2}(3+\al) k_1^* -\frac{k_3^*}2 + \frac{k_4^*}2}, & 
\text{if }|k_3 -k_4|\geq  2, \\
 2^{-\frac{ 1}{2}(3+\al) k_1^* + \frac{k_3^*}{4}}, &
\text{otherwise}.
\end{cases}
\]

\end{lemma}

\begin{proof}
First, we consider the case $|k_3 - k_4 | \geq 2$.
Then, we have $j_1^* \geq 3k_1^* + k_3^* - 5$, 
since $|\phi(\bar n)| \sim |(n_2 - n_3)(n_3 - n_4)|  (n_1^*)^2
\sim (n_1^*)^3|n_3 - n_4| \sim 2^{3k_1^*} |n_3 - n_4|$. 
Then,  \eqref{L63} follows from Lemma \ref{LEM:L62}. 
Here we used  $n_1^* \sim 2^{k_1^*}$ and $n_3^* \sim 2^{k_3^*}$, 
which was implied by the assumption $2^{k_1^*} \gg 2^{k_3^*} \gg 2^{k_4^*} \ge M$.

Next, we consider the case $|k_3 - k_4|\leq 1$.
We separately estimate the contributions from the following two  cases:
(a) $|n_3 - n_4| \geq 2^\frac{k_3^*}{2}$
and (b) $|n_3 - n_4| \leq 2^\frac{k_3^*}{2}$.
In Case (a), we have  $j_1^* \geq \frac{3}{2} k_1^* + \frac{k_3^*}{2} - 5$.
Then, Lemma \ref{LEM:L62} yields
\begin{align}
\text{LHS of } \eqref{L63}
 & \les  2^{-\frac{ 1}{2}(3+\al) k_1^*}
2^\frac{k_3^*}{4}
\prod_{i = 1}^4 2^\frac{ j_i}{2} \|\F(u_i)\|_{\l^2_n L^2_\tau}.
\label{L631}
\end{align}

\noi
In Case (b), we write $I_{k_3} = \bigcup_{\l_i} J_{\l_i}$,
$i = 3, 4$
where $|J_{\l_i}| = 2^\frac{k_3^*}{2}$.
Then, if $n_3 \in J_{\l_3}$ for some $\l_3$, 
there are only $O(1)$ many possible values of $\l_4 = \l_4(\l_3)$
such that  $n_4\in J_{\l_4}$.
Then, by writing
\[\sum_{n_3} \sum_{n_4} = \sum_{\l_3}  \sum_{\l_4 = \l_4(\l_3)} \sum_{n_3 \in J_{\l_3} }
\sum_{n_4 \in J_{\l_4}}\]

\noi
and repeating the previous argument for each $\l_3$,
we only lose $|J_{\l_i}|^\frac{1}{2} =  2^\frac{k_3^*}{4}$ by applying Cauchy-Schwarz inequality
in $n_3$ or $n_4$ at the end.
Finally, applying Cauchy-Schwarz inequality in $\l_3$, we obtain
\eqref{L631}.
\end{proof}

\section{Trilinear estimates}
\label{SEC:trilinear}

In this section, we prove the crucial  trilinear estimate
for the Wick ordered cubic 4NLS \eqref{4NLS1}.
This establishes the nonlinear estimate part \eqref{Anonlinear}
of the short-time Fourier restriction norm method.

\begin{proposition} \label{PROP:3lin}
Let $s \in \big(-\frac{9}{20}, 0\big)$ and $T > 0$.
Then, with $\alpha =  - \frac{8s}3 + $,
there exists $\theta > 0$ such that
\begin{align*}
\|\N(u_1,u_2,u_3)\|_{N_M^{s,\alpha}(T)}
+ \|\RR(u_1,u_2,u_3)\|_{N_M^{s,\alpha}(T)}
& \les  T^\theta \prod_{i = 1}^3 \|u_i\|_{F_M^{s,\alpha}(T)},
\end{align*}

\noi
where $\N(u_1, u_2, u_3)$ and
$\RR(u_1, u_2, u_3)$
are as in \eqref{NN1} and \eqref{NN2}.

\end{proposition}

The proof of Proposition \ref{PROP:3lin}
is analogous to the proof of the trilinear estimate
for the Wick ordered cubic NLS \eqref{NLS1} considered in \cite{GO}.
More precisely, we prove Proposition \ref{PROP:3lin}
by first applying the dyadic decomposition 
and then performing case-by-case analysis
on different frequency interactions.
For readers' convenience, 
we first summarize the size estimates
on the phase function $\phi(\bar n)$ defined in \eqref{Phi}
in various frequency regimes
under the non-resonance assumption $\{n_1, n_3\} \ne \{n, n_2\}$:

\medskip

\begin{itemize}
\item[(i)] If $|n|\sim |n_3| \gg |n_1|, |n_2|$, then
\begin{equation}\label{SC1}
|\phi(\bar n)| \sim (n_1^*)^3 |n_2 - n_1|.
\end{equation}

\item[(ii)] If $|n|\sim |n_2| \gg |n_1|, |n_3|$, then
\begin{equation}\label{SC2}
|\phi(\bar n)| \sim (n_1^*)^4.
\end{equation}

\item[(iii)] If $|n|\sim |n_2|\sim |n_3| \gg |n_1|$, then
\begin{equation}\label{SC3}
|\phi(\bar n)| \sim (n_1^*)^4.
\end{equation}

\end{itemize}

\noi
These size estimates immediately follow from the factorization in \eqref{Phi}.
Note that 
the conditions (i)\,-\,(iii) hold under the symmetries $n_1 \leftrightarrow n_3$
and $n \leftrightarrow n_2$,
and $\{n_1, n_3\} \leftrightarrow \{n, n_2\}$, respectively.
Recall that we have $2^{k_i} \geq M$, $i = 1, \dots, 4$.

In the following, 
by  assuming that $u_i$ has the Fourier transform supported on $I_{k_i}\times \R$,
we prove trilinear estimates 
for different frequency interactions.
We first consider the case
when the output frequency is high
(relative to the input frequencies).\footnote{In particular, this also includes the case when 
all the frequencies are low.}

\begin{lemma}
\label{LEM:HLL}
Let $\alpha\geq 0$.
If $k_4\geq k_1^* - 5$, then we have
\begin{align}
\|\P_{k_4}\N(u_{1}, u_{2}, u_{3})\|_{N_{M,k_4}^\alpha}
+ \|\P_{k_4}\RR  (u_{1}, u_{2}, u_{3})\|_{N_{M,k_4}^\alpha}
& \les 2^{- \frac{3}{4}(\al - \eps) k_1^*} \prod_{i = 1}^3\|u_{i}\|_{F_{M,k_i}^\al}
\label{HLL1}
\end{align}

\noi
for any $\eps >0$.
\end{lemma}

In view of \eqref{NN2}, 
there is no contribution from  the resonant part $\RR(u_1, u_2, u_3)$
except for the case:  $2^{k_1^*} \sim 2^{k_4^*}$,
which is treated in Lemma \ref{LEM:HLL}.
 Lemma \ref{LEM:HLL} also handles the  low frequency case: 
 $2^{k_1^*} \les M$.
The proof of Lemma \ref{LEM:HLL} closely follows 
that of Lemma 6.2 in \cite{GO}.
We present the details for readers' convenience.

\begin{proof}
Let $\g:\R \to [0, 1]$ be a smooth cutoff function supported on $[-1, 1]$
with $\g \equiv 1$ on $[-\frac{1}{4}, \frac{1}{4}]$
such that
\[ \sum_{m\in \Z} \g^3(t-  m ) \equiv 1, \qquad t \in \R.\]

\noi
Then, there exist $c, C> 0$ such that
\begin{align*}
 \eta_0(2^{[\al k_4]} (t - t_{k_4}))
= \eta_0(2^{[\al k_4]} (t - t_{k_4})) \sum_{|m|\leq C} \g^3(2^{[\al k_1^*]+c} (t - t_{k_4})- m)
\end{align*}

\noi
and
\begin{align}
 \eta_0(2^{[\al k_i]}t) \cdot \g(2^{[\al k_1^*]+c}t)
= \g(2^{[\al k_1^*]+c}t)
\label{HLL1b}
\end{align}

\noi
 for $i = 1,2,  3$.
Let $f_{k_i}=\F[ \g(2^{[\al k_1^*]+c}(t-t_{k_4}))\cdot u_i ]$, $i = 1, 3$,
and
$ \cj {f_{k_2}}=\F[ \g(2^{[\al k_1^*]+c}(t-t_{k_4}))\cdot  u_2 ]$.
Then,  it follows from  the  definition and Lemma \ref{LEM:embed3}
that 
\begin{align}
\text{LHS of }\eqref{HLL1} 
& \les  \sup_{t_{k_4} \in \R}    \big\|(\tau+n^4+    i2^{[\alpha k_4]})^{-1}
\ind_{I_{k_4}}(n)
(f_{k_1}*\wt{ f}_{k_2}*f_{k_3})
\big\|_{X_{M, k_4}} \notag\\
& \les  \sup_{t_{k_4}\in \R} \sum_{j_4=0}^{\infty}2^\frac{j_4}{2}
\sum_{j_1,j_2,j_3\geq [\al k_4]}
\big\|(2^{j_4} +2^{[\alpha k_4]})^{-1} 
\notag \\
& \hphantom{XXXXXXXX}\times 
\ind_{D^M_{k_4, j_4}}\cdot
(f_{k_1,j_1}*\wt{ f}_{k_2,j_2}*f_{k_3,j_3})\big\|_{\l^2_nL^2_\tau},
\label{HLL12}
\end{align}

\noi
where $\wt{f}(n,\tau)=f(-n,-\tau)$ and $f_{k_i, j_i}$, $i = 1, 2, 3$, is defined by
\[ f_{k_i, j_i}(n, \tau) =
\begin{cases}
f_{k_i}(n,\tau)\eta_{j_i}(\tau+n^4), &
\text{for }j_i>[\al k_4], \\
f_{k_i}(n,\tau)\eta_{\leq [\al k_4]}(\tau+n^4), & \text{for } j_i = [\al k_4].
\end{cases}
\]

Using the fact $\ind_{D^M_{k_4,j_4}}\leq \ind_{D^M_{k_4,\leq
j_4}}$, we have
\begin{align}
\eqref{HLL12}
& \les\sup_{t_{k_4} \in \R}\Big(\sum_{j_4<[\al k_4]}+\sum_{j_4\geq[\al k_4]}\Big)
2^{\frac{j_4}{2}}
\notag \\
& \hphantom{XXX} \times
\sum_{j_1,j_2,j_3\geq [\alpha k_4]}
\big\|(2^{j_4}+2^{[\alpha k_4]})^{-1}
\ind_{D^M_{k_4, j_4}}
f_{k_1,j_1}*\wt{f}_{k_2,j_2}*f_{k_3,j_3}
\big\|_{\l^2_n L^2_\tau }\notag\\
& \les \sup_{t_{k_4}\in \R}
\sum_{j_1,j_2,j_3,j_4\geq [\al k_4]}2^{-\frac{j_4}{2}}\big\|\ind_{D^M_{k_4,\leq j_4}}\cdot
(f_{k_1,j_1}*\wt{f}_{k_2,j_2}*f_{k_3,j_3})\big\|_{\l^2_n L^2_\tau }\notag \\
& \les \sup_{t_{k_4}\in \R}
\sup_{j_4 \geq [\al k_4]}
\sum_{j_1,j_2,j_3\geq [\al k_4]}2^{-(\frac{1}{2} - ) j_4}\big \|\ind_{D^M_{k_4,\leq j_4}}\cdot
(f_{k_1,j_1}*\wt{f}_{k_2,j_2}*f_{k_3,j_3})\big\|_{\l^2_n L^2_\tau }.
\label{HLL13}
\end{align}

\noi
Then, \eqref{HLL1} follows from 
\eqref{L4_0} in Lemma \ref{LEM:L4L6}
and \eqref{Xk2} with \eqref{HLL1b}.
\end{proof}

\begin{remark}\rm
\label{REM:HLL}
In the proof of Lemma \ref{LEM:HLL}, 
we used the $L^4$-Strichartz estimate
\eqref{L4_0} in Lemma \ref{LEM:L4L6}.
We point out that
the multilinear Strichartz estimates in Lemmas \ref{LEM:L64} and \ref{LEM:L62}
do not yield a better bound in this case.
Consider the case: high $\times$ high $\times$ high $\to$ high.
Then, applying Lemma \ref{LEM:L62} to \eqref{HLL13}
yields a bound with a constant $\sim 2^{(-\al + \frac{1}{2}+\eps)k_1^*}$,
which is worse than the constant $ 2^{- \frac{3}{4}(\al - \eps) k_1^*}$ in \eqref{HLL1} when $\al \leq 2$.
The proof of Lemma \ref{LEM:HLL} based on the $L^4$-Strichartz 
estimate \eqref{L4_0} in Lemma \ref{LEM:L4L6} also allows
us to handle the case: $2^{k_1^*}\le M$, 
for which Lemma \ref{LEM:L62} is not applicable.  See Remark \ref{REM:lowfreq}.
\end{remark}

Next, we consider the case
when the output  frequency is low  relative to  the input frequencies.
In such a case, we have  $2^{k_1^*} \gg 2^{k_4} \geq M$.
We treat this case in the next two lemmas.

\begin{lemma}[high $\times$ high $\times$ high $\to$ low]\label{LEM:HHHL}
Let $\al \geq 0$.
If $k_3\geq \max(20, \log_2M) $, $|k_3-k_i|\leq 5$, $i = 1, 2$, 
and $k_4\leq k_1-10$, then we have
\begin{align}
\|\P_{k_4}\N(u_{1}, u_{2}, u_{3})\|_{N_{M,k_4}^\alpha}
& \les 
\min(\Lambda_1, \Lambda_2)
\|u_{1}\|_{F_{M, k_1}^\alpha}\|u_2\|_{F_{M,k_2}^\alpha}\|u_3\|_{F_{M,k_3}^\alpha},
\label{HHHL1}
\end{align}

\noi
where $\Ld_1$ and $\Ld_2$ are given by
\[ \Lambda_1 = 2^{(-2  + \al +\eps)k_1^* - \al k_4 }
\qquad \text{and}
\qquad 
\Lambda_2 = 2^{(-2 + \frac{\al}{2}+\eps )k_1^* + (\frac{1}{2}-\al) k_4}\]

\noi
for any $\eps >0$.

\end{lemma}
\begin{proof}
In this case, we localize each component function $u_i$
onto subintervals of length $\sim 2^{-\al k_1^*}$.
With  $\g:\R\rightarrow [0,1]$ 
as in the proof of Lemma \ref{LEM:HLL}, 
we have 
\[ \eta_0(2^{[\al k_4]} (t - t_{k_4}))
= \eta_0(2^{[\al k_4]} (t - t_{k_4}))
\sum_{|m|\leq C 2^{[\al k_1^*] - [\al k_4]}} \g^3(2^{[\al k_1^*]+c} (t - t_{k_4})- m)\]

\noi
and
$ \eta_0(2^{[\al k_i]}t) \cdot \g(2^{[\al k_1]+c}t)
= \g(2^{[\al k_1]+c}t)$ for $i = 1,2,  3$.
In particular, we divide the time interval of length $\sim 2^{-\al k_4}$
into $O(2^{\al(k_1^* - k_4)})$ many 
 subintervals of length $\sim 2^{-\al k_1^*}$.
Then, proceeding as in the proof of Lemma \ref{LEM:HLL}
with \eqref{Xk2},
it suffices to prove that
\begin{align*}
2^{\alpha (k_1^*-k_4)}\sum_{j_4\geq [\al k_4]}2^{-\frac{j_4}{2}}\|\ind_{\wt D_{k_4, j_4}^M}\cdot
 (f_{k_1,j_1}*\wt{f}_{k_2,j_2}& *f_{k_3,j_3})\|_{\l^2_n L^2_\tau}\\
& \les\min(\Lambda_1, \Lambda_2)
\prod_{i = 1}^3 2^{\frac{j_i}{2}}\|f_{k_i,j_i}\|_{\l^2_n L^2_\tau}
\end{align*}

\noi
for any 
$f_{k_i,j_i}:\Z\times \R \to  \R_+$  
supported on $\wt D^M_{k_i, j_i}$
with $j_i \geq [\alpha k_1^*]$,   $i=1,2,3$, 
where
\[ \wt D^M_{k_i, j_i} = 
\begin{cases}
 D^M_{k_i,\leq j_i},
 & \text{when }j_i = [\alpha k_1^*],\\
 D^M_{k_i, j_i}, & 
\text{when }j_i > [\alpha k_1^*].
\end{cases}
\]

\noi
Here, we can assume that $j_i \geq [\al k_1^*]$, $i = 1, 2, 3$,
thanks to the time localization over an interval of size $\sim 2^{-[\al k_1^*] }$
and \eqref{Xk2}.
Hence, \eqref{HHHL1} 
with $\Lambda_1$
follows from
\eqref{L6_0} in Lemma \ref{LEM:L4L6}
with \eqref{SC3}, 
while
\eqref{HHHL1} 
with $\Lambda_2$
follows from Lemma \ref{LEM:L62}.
\end{proof}

\begin{lemma}[high $\times$ high $\times$ low $\to$ low] \label{LEM:HHLL}
Let $\al \in [0,2 ]$.
If $k_1\geq 20$, $|k_1-k_2|\leq 5$, and $k_3, k_4\leq k_1-10$, then we have
\begin{align}
\|\P_{k_4}\N(u_{1}, u _{2}, u_{3})\|_{N_{M,k_4}^\alpha}\les
\min(\Lambda_3, \Lambda_4)
\|u_{1}\|_{F_{M,k_1}^\alpha}\|u_{2}\|_{F_{M,k_2}^\al}\|u_3\|_{F_{M,k_3}^\alpha},\label{HHLL1}
\end{align}

\noi
where $\Lambda_3$ and $\Ld_4$ are given by
\[ \Lambda_3 = 2^{(-\frac{3}{2}+ \alpha+\eps)k_1^* - \al k_4 -\beta}
\quad \text{with }  \ 
\be =
\begin{cases}
 \frac{k_3^*}{2}, &  \text{if } |k_3 - k_4| \geq 2,\\ 
 0, & \text{otherwise}
\end{cases}
\]

\noi
and 
\[\Lambda_4 = \begin{cases}
2^{-\frac{ 1}{2}(3-\al-\eps) k_1^* -\frac{k_3^*}{2}+ \frac{k_4^*}{2} - \al k_4}, 
& \text{if }|k_3 - k_4| \geq 2,\\
 2^{-\frac{ 1}{2}(3-\al-\eps) k_1^*+ \frac{k_3^*}{4}- \al k_4} , &
\text{otherwise}
\end{cases}\]

\noi
for any $\eps > 0$.
\end{lemma}

\begin{proof}
We proceed as in the proof of Lemma \ref{LEM:HHHL},
Then, \eqref{HHLL1} with $\Lambda_3$ follows from \eqref{L6_0} in Lemma \ref{LEM:L4L6}
with the size estimates \eqref{SC1} and \eqref{SC2}.
Similarly, \eqref{HHLL1} with $\Lambda_4$ follows from Lemma \ref{LEM:L63}.
\end{proof}

\medskip

In the following, we briefly discuss the proof of Proposition \ref{PROP:3lin}.
From Lemmas \ref{LEM:HLL} - \ref{LEM:HHLL}, we have 
\begin{align}
2^{sk_4}\Big\{
\|\P_{k_4}\N(u_{1}, u_{2}, u_{3})\|_{N_{M,k_4}^\alpha}
+ \|\P_{k_4}\RR  & (u_{1}, u_{2}, u_{3})\|_{N_{M,k_4}^\alpha}\Big\}\notag\\
& \les 
2^{sk_4} \Lambda^* \|u_{1}\|_{F_{M,k_1}^\alpha}\|u_{2}\|_{F_{M,k_2}^\al}\|u_3\|_{F_{M,k_3}^\alpha},
\label{3linX}
\end{align}

\noi
where $\Lambda^*$ denotes the constants in 
Lemmas \ref{LEM:HLL} - \ref{LEM:HHLL}, depending on different frequency interactions.
Note that it suffices to guarantees that 
\begin{align}
2^{sk_4}\Lambda^* \les 2^{-\eps k_1^*} 2^{s(k_1+k_2+k_3)}.
\label{3linY}
\end{align}

\noi
Then, Proposition \ref{PROP:3lin} follows from summing \eqref{3linX} over different dyadic blocks.
Moreover, at a slight expense of the regularity in modulation, 
we can gain a factor $T^\theta$ for some $\theta > 0$.
See \cite{GO} for the details.

In the following, 
we perform case-by-case analysis 
on the constants obtained in Lemmas \ref{LEM:HLL} - \ref{LEM:HHLL}
and compute the restrictions on $s <0$ and $\al > 0$
such that   \eqref{3linY} holds.
In the following, $\eps = \eps(s) > 0$ denotes a small constant which may vary line by line.

\begin{itemize}
\item[(i)] The output frequency is high. 
\quad In view of Lemma \ref{LEM:HLL}, we need to have
$ -  \big(\frac{3}{4}\al - \eps\big)  \leq 2s$.
Hence, it suffices to choose
\begin{equation}
\al =  -\frac{8s}3 + \eps
\label{alpha}
\end{equation}

\noi
for some sufficiently small $\eps = \eps(s) > 0$.
Note that this is consistent with the heuristics
presented in Subsection \ref{SUBSEC:alpha}.

\medskip

\item[(ii)] high $\times$ high $\times$ high $\to$ low:
\quad
In view of  \eqref{HHHL1} with $\Lambda_1$ of Lemma \ref{LEM:HHHL}, we need to have 
\begin{align}
-2+\al +\eps \leq 3s.
\label{S1a}
\end{align}

\noi
Then, it follows from  \eqref{alpha}
that \eqref{S1a} holds for  $s \geq -\frac{6}{17}+\eps$.
Next, we consider the case $s \leq -\frac{6}{17}$.
Then,
from   \eqref{HHHL1} with $\Lambda_2$, we need to have 
\begin{align}
(-2+\tfrac{\al}{2} + \eps)k_1^* + (s + \tfrac{1}{2} - \al) k_4^* \leq 3sk_1^*.
\label{S2}
\end{align}

\noi
In view of \eqref{alpha}, we must have $s \geq -\frac{6}{13} + \eps$
from the coefficients of $k_1^*$,
while
we have $s\leq -\frac{3}{22} + \eps$ from the coefficient of $k_4^*$.
Hence, it follows from  \eqref{S1a} and \eqref{S2} 
that \eqref{3linY} holds for any $s \in \big(-\frac{6}{13}, 0\big)$.

\medskip

\item[(iii)] high $\times$ high $\times$ low $\to$ low:
\quad
First, we consider $s> -\frac{9}{28}$.
From  \eqref{HHLL1} with $\Lambda_3$ of Lemma \ref{LEM:HHLL}, we need to have
\begin{align*}
\big(-\tfrac{3}{2} + \al + \eps\big) k_1^* \leq 2sk_1^*
\qquad \text{and} \qquad
(s- \al) k_4 -\be \le s k_3.
\end{align*}

\noi
In view of  \eqref{alpha}, the first condition provides
$s \geq -\frac{9}{28} + \eps$.
The second condition is trivially satisfied when $k_4 \geq k_3 - 5$.
When $k_3 \geq k_4 + 5$, it gives $s \geq -\frac{1}{2}$.
Hence, 
\eqref{3linY} holds for  $s \in \big(-\frac{9}{28}, 0\big)$.

Next, we consider $s\leq -\frac{9}{28}$.
First, we consider the case  $ |k_3 - k_4|  \leq 1$.
From  \eqref{HHLL1} with $\Lambda_4$,  we need to have
\begin{align}
-\tfrac{ 1}{2}(3-\al-\eps) k_1^* \le 2 s k_1^*
\quad \text{and}\quad
\big(\tfrac{1}{4}- \al\big)k_3^*   \leq 0.
\label{S3}
\end{align}

\noi
In view of \eqref{alpha}, the first condition provides
$s \geq -\frac{9}{20} +  \eps$,
while the second condition provides $s\leq -\frac{3}{32} +\eps$.

Next, let us consider the case $k_3 \geq k_4 +2$.
(The case $k_4 \geq k_3 +2$ is easier.)
In this case,
we need to have
\begin{align}
-\tfrac{1}{2}(3 - \al - \eps)k_1^*-\tfrac12 k_3+ (s + \tfrac{1}{2} - \al)k_4 
\leq 2sk_1^* + sk_3.
\label{S6}
\end{align}\noi
This yields the condition $s \in \big(-\frac{9}{20}, -\frac{27}{220}\big)$.
Hence, it follows from  \eqref{S3} and \eqref{S6} 
that \eqref{3linY} holds for any $s \in \big(-\frac{9}{20}, -\frac{9}{28}\big]$.

\end{itemize}

Putting all the cases (i) - (iii)  together, we see that \eqref{3linY} holds
for $s \in \big(-\frac 9{20}, 0\big)$.

\section{Energy estimate on smooth solutions}
\label{SEC:energy}

In this section, we establish an energy estimate for  (smooth) solutions
to the Wick ordered cubic 4NLS \eqref{4NLS1}.
Let $u \in C(\R; H^\infty(\T))$ be a smooth solution to \eqref{4NLS1}.
Then, in view of \eqref{Es6}, our goal is to estimate
\[\sup_{t\in [-T,T]} E_{k_0}(u) (t)\]

\noi
in a summable manner over $k_0 \in \Z_M$,
where $E_{k_0}(u)$ is as in \eqref{Es4}.
By the fundamental theorem of calculus
with the equation \eqref{4NLS1}, 
we have
\begin{align}
E_{k_0}(u) (t) - E_{k_0}(u) (0)
& = 2 \Re\bigg( \int_0^t \sum_{n\in \Z} a_{k_0}(n) \dt \ft u_n(t') \cj{\ft u_n}(t') dt' \bigg)\notag \\
& = -  2 \Re i  \bigg(\int_0^t
\sum_{n\in \Z} a_{k_0}(n)
\sum_{\G(n)}\ft u_{n_1} \cj{\ft u_{n_2}}\ft u_{n_3}\cj{\ft u_{n} }(t') dt' \bigg)\notag \\
& \hphantom{XXXX} + \underbrace{2 \Re i \bigg( \int_0^t \sum_{n\in \Z} a_{k_0}(n) |\ft u_n(t')|^4 dt'\bigg)}_{=0},
 \label{E1}
\end{align}

\noi
where $\G(n)$ is as in \eqref{Gam1}.
By letting $n_4 = n$ and
symmetrizing under the summation indices $n_1, \dots, n_4$,
we obtain
\begin{align}
E_{k_0}(u) (t) - E_{k_0}(u) (0)
& = \frac{i}{2}   \int_0^t
\sum_{\substack{n_1 - n_2 + n_3-n_4 = 0 \\ n_2\ne n_1, n_3} }
\Psi(\bar{n})
\ft u _{n_1} \cj{\ft u _{n_2}}\ft u_{n_3}\cj{\ft u_{n_4}} (t') dt'\notag\\
& =: R_{k_0}(t),  
\label{E2}
\end{align}

\noi
where $\Psi(\bar{n})$ is defined by
\begin{equation} \label{Psi}
\Psi(\bar{n}) 
= a_{k_0}(n_1) - a_{k_0}(n_2) + a_{k_0}(n_3) - a_{k_0}(n_4).
\end{equation}

\noi
The symbol $\Psi(\bar{n})$  provides an extra decay
via the mean value theorem and the double mean value theorem
(Lemmas 4.1 and 4.2 in \cite{CKSTT1})
applied to the symbol  $a_{k_0}(\xi)$.
See \eqref{DMT1}, \eqref{MVT1}, and \eqref{MVT2}.

\begin{remark} \label{REM:energy1} \rm
In this section, we study an energy estimate on a single solution.
In Section \ref{SEC:uniq}, 
we establish  an energy estimate
for the difference of two solutions
in order to prove uniqueness of solutions.
It is significantly harder to 
establish  an energy estimate
for the difference of two solutions
mainly due to 
(i) the resonant contribution for the difference of solutions
(corresponding the second term on the right-hand side of \eqref{E1}) 
does not vanish
and (ii) the symmetrization process above fails
for the difference of solutions.
In order to overcome this difficulty, 
we perform an infinite iteration of normal form reductions.

\end{remark}

The main goal of this section is to 
establish the following multilinear estimate on
$R_{k_0}$.

\begin{proposition} \label{PROP:energy}
Let $s \in ( -\frac{9}{20}, 0)$ and  $\al = -\frac{8s}3 + $
as in \eqref{alpha}.
Then, there exist $\dl_0 >0$ and $\theta > 0$ such that 
\begin{align}
|R_{k_0}(T)|
& \lesssim
2^{-\dl_0 k_0} T^\theta \|u\|_{F^{s, \al}(T)}^4
\label{R4M1}
\end{align}

\noi
for all $k_0 \in \Z_M$ and $0< T \le 1$.
\end{proposition}

In \cite{GO}, we studied a similar energy estimate 
for solutions to the Wick ordered cubic NLS \eqref{NLS1}.
There, we needed to perform a normal form reduction
(i.e.~add a correction term) in order to achieve a better energy estimate
and hence match the regularity from the trilinear estimate.
The Wick ordered cubic 4NLS \eqref{4NLS1}, however, 
possesses much stronger dispersion
and we do not perform a normal form reduction.

\begin{remark}\label{REM:smoothing}\rm
As in \cite{GO},  the energy estimate \eqref{R4M1}
possesses a certain smoothing property,
namely, \eqref{R4M1} still holds true
even if we replace
the $F^{s, \al}(T)$-norm on the right-hand side
by $F^{s-\dl, \al}(T)$-norm
for some small $\dl > 0$.
This smoothing property plays an important role
in proving a compactness property
of smooth approximating solutions
(Lemma \ref{LEM:Kcpt};  see also Lemma \ref{LEM:conti}).
See the proof of Lemma 8.2 in \cite{GO}.

\end{remark}

\begin{proof}[Proof of Proposition \ref{PROP:energy}]
We first write $R_{k_0}$ as a multilinear operator given by
\begin{align*}
R_{k_0}(t)  
&  = R_{k_0}(u_1, u_2, u_3, u_4)(t)  \\
&  = \frac{i}{2}   \int_0^t
\sum_{\substack{n_1 - n_2 + n_3-n_4 = 0 \\ n_2\ne n_1, n_3} }
\Psi(\bar{n})
\ft u_1(n_1) \cj{\ft u_2(n_2)}\ft u_3(n_3)\cj{\ft u_4(n_4)} (t') dt'.
\end{align*}

\noi
Apply  the dyadic decomposition on
the spatial frequencies
$|n_i| \sim 2^{k_i}$, $k_i \geq \log_2M$, $i = 1, \dots, 4$.
By symmetry, assume that 
$ |n_1|\sim n_1^*.$
Then, 
it suffices to prove
\begin{equation*}
|R_{k_0}(T)|\les T^\theta 2^{-\dl_0 k_0} \prod_{i=1}^4 2^{(s-)k_i}\|\P_{k_i}u_i\|_{F_{M,k_i}^\alpha(T)}.
\end{equation*}

\noi
Here, a small extra decay is needed to sum over dyadic blocks.
Let  $\wt u_i$ be an extension of $u_i$ such that 
$\|\wt u_i\|_{F_{M,k_i}^\alpha}\leq 2 \|\P_{k_i}u\|_{F_{M,k_i}^\alpha(T)}$.
Let $\g:\R \to [0, 1]$ be a smooth cutoff function supported on $[-1, 1]$ such that
\[\sum_{m\in \Z}\g^4(t - m) \equiv 1, \qquad t \in \R.\]

\noi
With $K = k_1^* + c$, 
define $f_{i, j_i, m}$, $j_i \in \Z_{\geq 0}$, $i = 1, \dots, 4$, by  
\[ f_{i, j_i, m} = \F^{-1}\big[\eta_{j_i}(\tau + n^4)
\F[\g(2^{[\al  K]}t - m)\wt u_i]\big].\]

\noi
Then, 
 it suffices to prove 
\begin{align}
 \bigg|\int_{\R}    \ind_{[0,T]}(t)
  \sum_{j_1, \dots, j_4 \in \Z_{\geq 0}} & \sum_{|m|\leq 2^{[\alpha K] }(T+1)}
 \sum_{\substack{n_1 - n_2 + n_3-n_4 = 0 \\
n_2\ne n_1, n_3} }
  \Psi(\bar{n})
 \ft{f}_{1, j_1, m}({n_1}) 
 \cj{\ft{f}_{2, j_2, m}({n_2})} \notag \\
 & \times \ft{f}_{3, j_3, m}({n_3})\cj{\ft{f}_{4, j_4, m}({n_4})} (t) dt \bigg|
  \les T^\theta 2^{-\dl_0 k_0} \prod_{i=1}^4 2^{(s-)k_i}\|\wt u_i\|_{F_{M,k_i}^\alpha},
\label{R43}
\end{align}

In the following,  we prove \eqref{R43}
for each dyadic  modulation size $\sim 2^{j_i}$, $i = 1, \dots, 4$.
In view of \eqref{Xk2}, 
we assume that 
$j_i \geq \al K$.
Define the subsets $\mathcal{A}$ and $\mathcal{B}$ of 
$\{m\in \Z: |m| \leq  2^{[\al K]}(T+1)\}$ by 
\begin{align*}
&  \mathcal{A} = \big\{ m \in \Z: 
\ind_{[0,T]}(t)\g(2^{[\alpha K]}t - m) =  \g(2^{[\alpha K]}t - m) \big\}, \\
&  \mathcal{B} = \big\{ m \in \Z: 
\ind_{[0,T]}(t)\g(2^{[\alpha K]}t - m) \ne  \g(2^{[\alpha K]}t - m)
\text{ and } \ind_{[0,T]}(t)\g(2^{[\alpha K]}t - m) \not \equiv 0 \big\}.
\end{align*}

\noi
Namely, $\mathcal{A}$ denotes the set of $m \in \Z$
such that the support of $\g(2^{[\alpha K]}t - m)$ lies in the interior
of the interval $[0, T]$, 
while $\mathcal{B}$ denotes those $m \in \Z$
such that the support of $\g(2^{[\alpha K]}t - m)$ intersects the boundary point $t = 0$ or $t = T$.
In the following, we separately estimate the contributions from $\mathcal{A}$ and  $\mathcal{B}$.
Lastly, we simply denote $f_{i, j_i, m}$  by $f_{i, j_i}$ in the following.

\medskip
\noi
 {\bf Part 1:}
First, we consider the terms
with $m \in \mathcal{A}$.
Note that  we can drop the sharp cut-off $\ind_{[0,T]}(t)$ on the left-hand side
of \eqref{R43}
in this case.
We prove \eqref{R43} with $\theta = 1$ in this case.
The main ingredients are
the (double) mean value theorem 
and the following lower bound on the largest modulation;
with  $\s_j = \tau_j + n_j^4$, we have
\begin{align*}
\s_1^* := \max (|\s_1|, |\s_2|, |\s_3|, |\s_4|) \gtrsim |\phi(\bar{n})|
\sim (n_1^*)^2 |\mu(\bar n )|,
\end{align*}

\noi
where $\phi(\bar n)$ and $\mu(\bar n)$
are as in \eqref{Phi} and \eqref{Mu}.
Recall  that we have $2^{k_i} \geq M$, $i = 1, \dots, 4$.
Given $k \in \Z_M$, it follows from \eqref{Es2} and \eqref{Es2a}  that
\begin{align*}
|\dd^\g a_{k_0}(\xi)| \les 2^{(2s-\g)k} 2^{-\dl_0 |k - k_0|}.
\end{align*}

\noi
for $\xi \in I_k^M$, $\g = 1, 2$.

\medskip
\noi
$\bullet$ {\bf Case (a):}    $ 2^{k_1^*} \les M$.
\\
\indent
In this case, we have 
$2^{k_1^*} \sim 2^{k_4^*}$.
Then, by the double mean value theorem \cite[Lemma 4.2]{CKSTT1}, we have
\begin{align} \label{DMT1a}
|\Psi(\bar{n})|
& \lesssim |a''_{k_0}(n_1^*)| \cdot |(n_4 -n_1)(n_4 - n_3)|
\les 2^{(2s-2) k_1^*   -\dl_0 |k_1^* -k_0|} |\mu(\bar n)|\notag\\
&  \les 2^{  - \dl_0 k_0} 2^{(2s - 2 + \dl_0)k_1^*} |\mu(\bar n)|.
\end{align}

We first consider the case $n_1^* \ges 2^{k_1^*}$.
In this case, we apply  Lemma \ref{LEM:L62} in view of Remark \ref{REM:lowfreq}.
With \eqref{alpha}, we have
\begin{align}
\sum_{|m|\leq 2^{[\alpha K]}(T+1)}
(\s_1^*)^{-\frac{1}{2}}|\Psi(\bar n)| 2^{(-\frac\alpha2+\frac12)k_1^*}
& \les  T 2^{-\dl_0 k_0} 2^{\frac{\alpha k_1^*}2}  |\mu(\bar n)|^{\frac12} 2^{(2s  -\frac52+ \dl_0)k_1^*}  \notag\\
&
 \les T 2^{-\dl_0 k_0} 2^{(\frac{2}3s  -\frac32+\dl_0 + )k_1^*} \leq T  2^{-\dl_0 k_0} 2^{(4s-)k_1^*}
\label{R4M1aa}
\end{align}

\noi
for  sufficiently small  $\dl_0 = \dl_0(s) > 0$, 
provided that 
$s>-\frac9{20}$. 
Then, \eqref{R43} follows from 
Lemma \ref{LEM:L62} with \eqref{R4M1aa} in this case.

Next, we consider  the case $n_1^* \ll 2^{k_1^*}$.
In this case,  Lemma \ref{LEM:L62} is not applicable. 
Note, however,  that this case occurs only when $n_1^* \ll M$.
Moreover,  from the definition of the symbol $a_{k_0}$, 
we see that   $a_{k_0} (n)$ is constant for $|n|\le \frac{M}2$.
Hence, 
we conclude that  $\Psi (\bar n)= 0$ when $n_1^* \ll M$
and  there is no contribution to \eqref{R43} in this case.

\medskip
\noi
$\bullet$ {\bf Case (b):}  $|n_4 - n_1| , |n_4 - n_3| \ll n_1^*$ and $2^{k_1^*} \gg M$.
\\
\indent
In this case, we have 
$|n_1|\sim |n_2|\sim |n_3|\sim |n_4|\sim n_1^* \sim 2^{k_1^*}$.
Moreover,  by the double mean value theorem, we have
\begin{equation} \label{DMT1}
|\Psi(\bar{n})| 
\les 2^{(2s - 2+ \dl_0) k_1^*  - \dl_0 k_0} |\mu(\bar n)|
\end{equation}

\noi
 as in Case (a).
Then, the rest follows as in Case (a).

\medskip
\noi
$\bullet$  {\bf Case (c):} $|n_4 - n_1| \sim n_1^* \gg |n_4 - n_3|  = |n_1 - n_2|$  and $2^{k_1^*} \gg M$.
\\
\indent
In this case,  we have $|n_2|\sim|n_1|\sim n_1^*$. 
Then, by the mean value theorem, we have
\begin{equation} \label{MVT1}
|a_{k_0}(n_1) - a_{k_0} (n_2)| \lesssim |a_{k_0}'(n_1^*)|\cdot   |n_1 - n_2|
\les  2^{-\dl_0 k_0} 2^{(2s -1+\dl_0) k_1^* } |n_1 - n_2|.
\end{equation}

\noi
Moreover, from \eqref{Phi}, we have $|\phi(\bar n)| \sim (n_1^*)^3 |n_4 - n_3|$.

\smallskip

\noi
$\circ$ Subcase (c.i):
$|n_4 - n_3| \ll n_3^*$.
\\
\indent
In this case,  we  also have $|n_3|\sim|n_4| \sim n_3^*$. Then, 
by the mean value theorem, we have
\begin{equation} \label{MVT2}
|a_{k_0}(n_3) - a_{k_0} (n_4)| 
\les |a_{k_0}'(n_3^*)|\cdot  |n_4 - n_3| 
\les  2^{-\dl_0 k_0} 2^{(2s -1+\dl_0) k_3^* } |n_4 - n_3|.
\end{equation}

\noi
We point out that \eqref{MVT2} holds true even when $n_3^*\les M$, since $a'_{k_0} (n) = 0$ 
for $|n|\le \frac{M}2$.
See \eqref{Es2b}.
Hence, it follows from \eqref{MVT1} and \eqref{MVT2} that 
\[|\Psi(\bar{n})|\lesssim    2^{-\dl_0 k_0} 2^{(2s -1+\dl_0) k_3^* }  |n_4 - n_3|\]

\noi and thus 
\begin{align*}
\sum_{|m|\leq 2^{[\alpha K]}(T+1)}
&  (\s_1^*)^{-\frac{1}{2}}  |\Psi(\bar n)|
2^{-\frac\alpha2 k_1^*}  2^{\frac{k_4^*}2}
 \les T  2^{-\dl_0 k_0} 
2^{(- \frac43s -\frac{3}{2} +)k_1^*} 2^{(2s-\frac 12+ \dl_0) k_3^*}|n_4 - n_3|^\frac{1}{2}\notag\\
&  \les T 2^{-\dl_0 k_0}  2^{  (-\frac43s-\frac{3}{2} +\dl_0+) k_1^* } 2^{2sk_3^*} 
 \les T  2^{-\dl_0 k_0}  
 2^{ (- \frac{10}3s-\frac{3}{2}+\dl_0 +)k^*_1}
\bigg(\prod_{i = 1}^4 2^{(s-) k_i}\bigg).
\end{align*}

\noi
Hence, 
 \eqref{R43} follows from 
Lemma \ref{LEM:L62}, 
provided that 
$s>-\frac9{20}$ and $\dl_0= \dl_0(s)>0$ is sufficiently small.

\smallskip
\noi
$\circ$ Subcase (c.ii):
$|n_4 - n_3| \sim n_3^*$. 
\\
\indent
In this case,  we have 
$ |\phi(\bar{n})| \gtrsim (n_1^*)^3n_3^*$
and $|\Psi(\bar n)| \les 2^{-\dl_0 k_0 }2^{(2s + \dl_0) k_4^*} $.
Thus, we have 
\begin{align*}
 \sum_{|m|\leq 2^{[\alpha K]}(T+1)}
(\s_1^*)^{-\frac{1}{2}} &|\Psi(\bar n)|   2^{-\frac\alpha2 k_1^*}  2^{\frac{k_4^*}2} \\
& \les T 2^{-\dl_0 k_0 }
2^{ (- \frac{10}3s-\frac{3}{2}+\dl_0 +)k_1^*}  2^{(-s-\frac12)k_3^*}  2^{(s+\frac12)k_4^*}
\bigg(\prod_{i = 1}^4 2^{(s-) k_i}\bigg).
\end{align*}

\noi
Hence, 
 \eqref{R43} follows from 
Lemma \ref{LEM:L62}, 
provided that 
$s>-\frac9{20}$ and $\dl_0>0$ is sufficiently small.

\medskip
\noi
$\bullet$  {\bf Case (d):} $|n_4 - n_1|,   |n_4 - n_3|\sim n_1^*  $  and $2^{k_1^*} \gg M$.
\\
\indent
In this case, we have $|\phi(\bar{n})| \sim (n_1^*)^4$ and 
$|\Psi(\bar n)| \les 2^{-\dl_0 k_0 }2^{(2s + \dl_0) k_4^*} $.
Thus, we have
\begin{align*}
\sum_{|m|\leq 2^{[\alpha K]}(T+1)}
(\s_1^*)^{-\frac{1}{2}}|\Psi(\bar n)| 2^{-\frac\alpha2 k_1^*}  2^{\frac{k_4^*}2}
&  \les T 2^{-\dl_0 k_0} 
 2^{(- \frac{13}{3}s - 2+\dl_0+) k_1^*} 2^{(s+\frac12)k_4^*}
\bigg(\prod_{i = 1}^4 2^{(s-) k_i}\bigg).
\end{align*}

\noi
Hence, 
 \eqref{R43} follows from 
Lemma \ref{LEM:L62}, 
provided that 
$s>-\frac9{20}$ and $\dl_0>0$ is sufficiently small.

\medskip
\noi
 {\bf Part 2:}
Next, we consider the terms
with $m \in \mathcal{B}$.
In this case, we use 
 Lemmas \ref{LEM:sup} to handle the sharp cutoff $\ind_{[0, T]}$.
Note that there are only $O(1)$ many values of $m \in \mathcal{B}$.
Namely, we can save 
$(n_1^*)^{-\al}$ as compared to the analysis in Part 1.

We only consider Case (a) above
as the other cases follow in a similar manner.
With \eqref{DMT1a} and $|\phi(\bar n)|\sim (n_1^*)^2|\mu(\cj n)|$, we have 
\begin{align*}
(\s_1^*)^{-\frac{1}{2}+\theta+} |\Psi(\bar n)| \les 2^{-\dl_0 k_0} (n_1^*)^{4s-}
\end{align*}

\noi
for $\theta > 0 $ sufficiently small such that $-1 + 2\theta  < s$.
Suppose $\s_1 = \s_1^*$.
Then, by $L^2_{x, t}, L^6_{x, t}, L^6_{x, t}, L^6_{x, t}$-H\"older's inequality
and  Lemma \ref{LEM:sup}, we have 
\begin{align*}
|R_{k_0}(T)|
& \les 2^{-\dl_0 k_0} (n_1^*)^{4s-}
\sup_{j_1} 2^{ (\frac12 -\theta - ) j_1}
\|  \F(\ind_{[0, T]} f_{1, j_1})\|_{\l^2_n L^2_{\tau}} 
 \sum_{j_2, j_3, j_4}\prod_{i = 2}^4 \|f_{i, j_i}\|_{L^6_{x, t}}\\
& \lesssim
T^\theta 2^{-\dl_0 k_0}
 \prod_{i =  1}^4 2^{(s-) k_i} \|\P_{k_i} \wt u_i\|_{F^\al_{M, k_i}},
\end{align*}

\noi
where we used 
 Lemma  \ref{LEM:timedecay} in the last step.
 This completes the proof of Proposition \ref{PROP:energy}.
\end{proof}

\section{Global existence}
\label{SEC:GWP}

In this section, we prove global existence (Theorem \ref{THM:1})  by putting together
the trilinear estimate (Proposition \ref{PROP:3lin})
and the energy estimate (Proposition \ref{PROP:energy}).
We also make use of the decay property \eqref{Sob2}
of the $H^s_M$-norm 
as $M \to \infty$.
Moreover, we establish an exponential growth bound
on the $H^s$-norms of solutions.
In view of  the time reversibility of \eqref{4NLS1},
we only consider positive times in the following.

\subsection{Proof of Theorem \ref{THM:1}}
The following proposition 
establishes the long-time existence for small initial data
which plays a key role in the proof of Theorem \ref{THM:1}.

\begin{proposition}
\label{PROP:local}
Let $s\in (-\frac9{20},0)$. 
Then, given  $u_0 \in H^s(\T)$, 
there exist $T = T(\|u_0\|_{H^s}) > 0$ and  a local-in-time solution $u$ to 
the Wick ordered cubic 4NLS \eqref{4NLS1} on $[0, T]$ with $u|_{t = 0} = u_0$. 
Furthermore, there exists $\eps_0 >0$
such that if  $u_0 \in H^s(\T)$ satisfies
\begin{align}
\|u_0\|_{H^s_{M}} \le \eps_0
\label{small1}
\end{align}

\noi
for some dyadic $M \geq 1$, 
then the corresponding solution $u$ to \eqref{4NLS1}
 with $u|_{t = 0} = u_0$
 can be extended to 
the unit time interval $[0, 1]$ with the following estimate:
\begin{align}
\label{localbound}
\sup_{t \in [0, 1]} \| u(t)\|_{H^s_M} \le 2 \| u_0 \|_{H^s_M } .
\end{align}
\end{proposition}

\medskip

We first assume Proposition \ref{PROP:local}
and present the proof of  Theorem \ref{THM:1}.

\begin{proof}[Proof of Theorem \ref{THM:1}]
Given $T > 0$, we iteratively apply 
Proposition \ref{PROP:local}
and construct a solution $u$ on $[0, T]$.
Let  $u_0 \in H^s(\T)$.
Then,  there exists $M = M(s,T,u_0,\eps_0) \geq 1$
such that 
\[
\|u_0\|_{H_M^s} \le 2^{ - [T] - 1}\eps_0,
\]

\noi
where $\eps_0$ is as in Proposition \ref{PROP:local}. 
Hence, we can apply  Proposition \ref{PROP:local} $[T] +1$ times
and construct the solution $u$ on $[0, T]$, satisfying 
\begin{align*}
\sup_{t \in [0, T]} \| u(t)\|_{H^s_M} \le 2^{[T]+1} \| u_0 \|_{H^s_M}
\leq \eps_0.
\end{align*}

\noi
This proves Theorem \ref{THM:1}.
\end{proof}

\medskip

Before proceeding to the proof of Proposition \ref{PROP:local}, 
we recall the following lemma (Lemma~8.1  in \cite{GO}).

\begin{lemma}\label{LEM:Kconti}
Let $s \in \R$. Given $u\in C(\R; H^\infty(\T))$,
let $X_M(T) = \|u\|_{E_M^s(T)} + \|\mathfrak N(u) \|_{N_M^{s, \al}(T)}$.
Then, 
$X_M(T)$ is non-decreasing and continuous in $T \in \R_+$.
Moreover, we have
\[ \lim_{T\to 0} X_M(T) = \|u(0)\|_{H_M^s}.\]
\end{lemma}

While our function spaces depend on the parameter $M \geq 1$, 
the proof of Lemmas 8.1  in \cite{GO} applies to  Lemma \ref{LEM:Kconti}
without any change
for {\it fixed} $M \geq 1$.

\begin{proof}[Proof of Proposition \ref{PROP:local}] 
We only sketch the proof under the smallness assumption \eqref{small1},
since it follows closely the argument in  \cite[Section 8]{GO}.
See also Remark \ref{REM:LWP}.
Fix $s \in (-\frac{9}{20}, 0)$
and $\al = -\frac83s +$.
Let $u\in C(\R; H^\infty(\T))$ be a smooth solution to \eqref{4NLS1}
with $u|_{t = 0} = u_0$.
Then, 
it follows from Lemma   \ref{LEM:linear}, Proposition \ref{PROP:3lin},
\eqref{Es6}, and Proposition  \ref{PROP:energy} that 
there exists $\theta = \theta(s)> 0$ such that 
\begin{align}
  \|u\|_{F_M^{s, \al}(T)}   
& \les \|u\|_{E_M^s(T)} +  \|\mathfrak N(u)\|_{N_M^{s, \al}(T)}, 
\label{K1}\\
  \|\mathfrak N(u)\|_{N_M^{s, \al}(T)} 
&   \les  T^\theta \|u\|_{F_M^{s, \al}(T)}^3,
 \label{K2}\\
  \|u\|_{E_M^s(T)}^2   & \leq \|u_0\|_{H_M^s}^2 +  
C T^\theta  \|u\|_{F_M^{s, \al}(T)}^4,
\label{K3}
\end{align}

\noi
for any $T>0$ and $M \geq 1$,  
where $\mathfrak{N}(u) = \N(u) + \RR(u)$ denotes the nonlinearity of 
\eqref{4NLS1} defined in \eqref{Xnonlin}. 
Letting $X_M(T)$ be as in Lemma \ref{LEM:Kconti}, 
it follows from  \eqref{K1}, \eqref{K2}, and \eqref{K3} that 
\begin{align*}
X_M(T)^2 \le  2 \|u_0\|_{H_M^s}^2 + 
C T^{\theta}  \big\{X_M(T)^2+X_M(T)^4\big\}\cdot  X_M(T)^2
\end{align*}

\noi
for any $T > 0$.
Now, choose $\eps_0 > 0$ sufficiently small such that 
\[  C ( 4 \eps_0^2 + 16 \eps_0^4) \le \tfrac12.\]

\noi
Then, 
in view of Lemma \ref{LEM:Kconti}, 
it follows from a continuity argument that 
\begin{align}
X_M(T) \leq 2\|u_0\|_{H_M^s}
\label{localbound3}
\end{align}

\noi
for any $T \in (0, 1]$.
Hence, the a priori bound \eqref{localbound} for smooth solutions
follows from \eqref{localbound3} and \eqref{Es}.

Next, we recall the following compactness lemma.

\begin{lemma} \label{LEM:Kcpt}
Let  $s > -\frac{9}{20}$.
Given
 $u_0 \in H_M^s(\T)$,
let $u_n\in C(\R;  H^\infty(\T))$ be a global solution  to \eqref{4NLS1} with  $u_n|_{t = 0} = \P_{\leq n} u_{0}$.
Then, there exists $T_0=T_0(\|u_0\|_{H^s}) >0$ such that 
the set $\{ u_n\}_{n\in \mathbb{N}}$ is precompact in $ C([-T, T]; H_M^s(\T))$
for $T \leq T_0$.
Moreover, $\| \P_{> N} u_n\|_{C_T H_M^s}$ tends to 0 as $N \to \infty$, 
uniformly in $n \in \mathbb{N}$.
\end{lemma}

See Lemma 8.2 in \cite{GO} for the details of the proof.
See also Lemma \ref{LEM:conti} below.
We point out that the smoothing property of 
the energy estimate in Proposition \ref{PROP:energy}
plays an important role in proving Lemma \ref{LEM:Kcpt}.

In view of Lemma \ref{LEM:Kcpt} with $T_0 = 1$,
we can extract a subsequence,
which we still denote by
 $\{u_n\}_{n \in \mathbb{N}}$,  converging to some $u$ in $C([0, 1]; H_M^s(\T))$.
It remains to show that this limit $u$ is a distributional solution to \eqref{4NLS1}.
It follows from Lemma \ref{LEM:Kcpt} that
 $\{u_n\}_{n \in \mathbb{N}}$ also converges in $E_M^s(1)$.
In view of \eqref{K1} and \eqref{K2}, this in turns implies that  
 $\{u_n\}$ converges to $u$ in $F_M^{s,\alpha}(1)$.
Finally, by applying the trilinear estimate (Proposition \ref{PROP:3lin}), 
we see that the nonlinearity
$\{ \Nf(u_n)\}_{n \in \mathbb{N}}$
converges to $ \Nf(u)$ in $N_M^{s,\alpha}(1)$.
Hence, the limit $u$ is a distributional solution to  \eqref{4NLS1}
on the time interval $[0, 1]$.
This proves local existence for the Wick ordered cubic 4NLS \eqref{4NLS1} in $H_M^s(\T)$
for $s > -\frac 9{20}$.
Moreover, from the a priori estimate for smooth solutions
and the convergence of $u_n$ to $u$ in $E^s_M(1)$, 
\eqref{localbound} also holds for the solution $u$.
\end{proof}

\begin{remark}\label{REM:LWP}\rm

Let us briefly discuss the case when we do not impose the smallness assumption
and $M = 1$.  This is the setting considered in \cite{GO}
and hence is relevant for the proof of the non-existence result 
(Corollary \ref{COR:NE}).

Let $R = \|u_0\|_{H^s}$ 
Then, choose $T_0 = T_0(R)\leq1$ sufficiently small such that 
\[  C T_0^{\theta} ( 4R^2  + 16R^4) \le \tfrac12.\]

\noi
Then, 
a continuity argument with  Lemma \ref{LEM:Kconti}
yields \eqref{localbound3}
for $T \in (0, T_0]$.
By repeating the argument above, 
one can prove local existence on $[0, T_0]$
for  $T_0 = T_0(\|u_0\|_{H^s}) > 0$.

\end{remark}

\subsection{On the growth of Sobolev norms}

In this subsection, we study the growth of the $H^s$-norm of a solution 
to \eqref{4NLS1}, $s \in (-\frac{9}{20}, 0)$, constructed in Theorem \ref{THM:1}.

Fix  $s_0 \in (-\frac9{20},0)$.
The following bound follows from iterating Proposition \ref{PROP:local}.

\begin{lemma}
\label{LEM:GW1}
Let $s_0 \le s < 0$
and  $0< \eps \le \eps_0$, where $\eps_0$ is as in Proposition \ref{PROP:local}. 
Let $u$ be a solution to \eqref{4NLS1} with $u|_{t = 0} = u_0\in H^s(\T)$ 
such that
\begin{align*}
\|u_0\|_{H_M^{s_0}} \le \eps
\end{align*}

\noi
for some dyadic $M\geq 1$.
Then, the following bound holds:
\begin{align}
\label{GW2}
\sup_{t \in [0, T]} \| u(t) \|_{H^s_M} \les 2^{T} \|u_0\|_{H^s_M},
\end{align}

\noi
 for all $ 0 < T\le T_0$, 
where 
\[ T_0\sim \log_2 \Big(\frac{\eps_0}{\eps}\Big). \]

\end{lemma}

\begin{proof}
When  $s = s_0$, 
the estimate \eqref{GW2}  follows from the proof of Proposition \ref{PROP:local}, 
namely iterating \eqref{localbound} $[T] + 1$ times.
For general $s \in (s_0, 0)$, 
we exploit the following equivalence
\begin{align}
\| f \|_{H_M^s}^2 \sim \sum_{K\ge M} K^{-2s_0 + 2s} \|f\|^2_{H^{s_0}_K}
\label{GW3}
\end{align}

\noi
for any $f \in H^s(\T)$ and any dyadic $K\geq M \geq 1$.
We first assume \eqref{GW3} and prove \eqref{GW2}.
By \eqref{GW3}, \eqref{GW2} for $s = s_0$, and the monotonicity
of the $H^s_M$-norm in $M$, we have
\begin{align*}
\sup_{t \in [0, T]} \| u(t) \|_{H^s_M}^2 
& \les \sum_{K\ge M} K^{-2s_0 + 2s} \sup_{t \in [0, T]}  \|u(t) \|^2_{H^{s_0}_K}\\
& \les   2^{2T} \sum_{K\ge M} K^{-2s_0 + 2s}   \|u_0 \|^2_{H^{s_0}_K}
\sim 2^{2T} \|u_0\|_{H^s_M}^2.
\end{align*}

\noi
This proves \eqref{GW2}.

It remains to show \eqref{GW3}.
Let us first consider the contribution from $|n|\le M$.
With $K^2 + n^2 \sim K^2$ for $K \geq M$, we have
\begin{align}
 \sum_{K\ge M} K^{-2s_0 + 2s} \|f_{\le M} \|^2_{H^{s_0}_K}
& \sim  \sum_{K\ge M} K^{  2s} \| f_{\le M}\|^2_{L^2} 
 \sim  M^{  2s} \| f_{\le M}\|^2_{L^2}
\sim \| f_{\le M}\|_{H_M^s}^2,
\label{GW3a}
\end{align}

\noi
where $f_{\le M} =  \F^{-1}\big[\ind_{|n|\le M} \ft f\big]$.
Next, we consider the contribution from $|n| > M$. 
By Fubini's theorem, we have
\begin{align}
 \sum_{K\ge M} K^{-2s_0+ 2s}  \|f_{> M} \|^2_{H^{s_0}_K}
& \sim 
 \sum_{K\ge M} K^{  2s}  \sum_{M< |n| \leq K }  |\ft f(n)|^2  \notag \\
& \hphantom{X}
+  \sum_{K\ge M} K^{-2s_0 +  2s}   \sum_{ |n| > K }  |n|^{2s_0} |\ft f(n)|^2  \notag \\
& \sim \sum_{|n| > M} \bigg(\sum_{K \ge |n|} \frac{K^{2s}}{|n|^{2s}}\bigg)
 |n|^{2s} |\ft f(n)|^2 \notag \\
& \hphantom{X}
 +  \sum_{ |n| > M }\bigg( \sum_{M \le K< |n|} \frac{K^{-2s_0 +  2s} }{|n|^{-2s_0 +  2s} }\bigg)
 |n|^{2s} |\ft f(n)|^2 \notag \\
& \sim \| f_{> M}\|_{H_M^s}^2,
\label{GW3b}
\end{align}

\noi
where $f_{> M} =  f - f_{\le M}$.
Then, \eqref{GW3} follows from \eqref{GW3a} and \eqref{GW3b}.
\end{proof}

By applying Lemma \ref{LEM:GW1}, we obtain 
the following global-in-time bound on the $H^s$-norm of solutions to \eqref{4NLS1} for $s_0 < s < 0$.

\begin{proposition}
\label{PROP:GWP3}
Fix  $s_0 \in (-\frac9{20},0)$.
Let $s\in (s_0,0)$, $B > 0$, and $u$ be a solution to \eqref{4NLS1} 
with $u|_{t = 0} = u_0\in H^s(\T)$ such that 
\[
\|u_0\|_{H^s} \le B.
\]

\noi
Then, we have
\begin{align}
\label{GW4a}
 \| u(t) \|_{H^s} \les  \eps_0^{\frac{s}{s-s_0}} \big( 2^{t} B\big)^{1- \frac{s}{s-s_0}}
\end{align}

\noi
for all $t > 0$, where $\eps_0$ is as in Proposition \ref{PROP:local}.

\end{proposition}

\begin{proof}

By choosing $M\gg1$ sufficiently large, it follows from \eqref{GW3} that 
\[ \|u_0 \|_{H^{s_0}_M} \le M^{- s + s_0}B \leq \eps_0.\]

\noi
Then, it follows from  Lemma \ref{LEM:GW1} that 
\begin{align*}
\sup_{t \in [0, T]} \| u(t) \|_{H^s_M} \les 2^T \|u_0\|_{H^s_M}
\end{align*}

\noi
for all  $T > 0$ such that 
\[
T \les \log_2 \bigg(\frac{M^{ s-s_0}\eps_0}{B}\bigg).
\qquad \text{Namely, } M \ges \bigg(\frac{2^T B}{\eps_0}\bigg)^\frac{1}{s-s_0}.
\]

\noi
Therefore, we obtain
\begin{align*}
\sup_{t \in [0, T]} \| u(t) \|_{H^s} 
\le M^{-s} \sup_{t \in [0, T]} \| u(t) \|_{H^s_M}  \les  \eps_0^{\frac{s}{s-s_0}} \big( 2^{T} B\big)^{1- \frac{s}{s-s_0}}
\end{align*}

\noi
for any $T >0$.  This proves \eqref{GW4a}.
\end{proof}

\begin{remark}
\rm
In  Proposition \ref{PROP:GWP3}, 
we only obtain an exponential upper bound for the growth of the $H^s$-norm.
One may upgrade this exponential bound
to a polynomial bound if one incorporates
a scaling in the argument (as in \cite{KT2}). 
We, however, do not pursue this issue
since (i) our argument with one parameter $M \geq 1$
(without a scaling parameter)
suffices to prove global existence
and (ii) a polynomial bound
is by no mean optimal.
\end{remark}

\section{Uniqueness and continuous dependence}
\label{SEC:uniq}

In Section \ref{SEC:GWP}, 
we proved local and global existence of solutions to the Wick ordered cubic 4NLS \eqref{4NLS1}.
The remaining part of this paper is devoted to the proof of Theorem \ref{THM:2}.
The main difficulty lies in proving uniqueness of solutions.
Once we prove uniqueness, continuous dependence follows immediately.
See Subsection \ref{SUBSEC:uniq}.

In Subsection \ref{SUBSEC:energy2}, 
we set up an energy estimate for the difference of two solutions
with the {\it same} initial condition.
In particular, we state a key identity,
expanding  the energy estimate into 
a sum of infinite series of multilinear expressions of arbitrarily large degrees
(Propositions \ref{PROP:main} and \ref{PROP:main2}).
See Remark \ref{REM:uniq3}.
This identity allows us to establish crucial smoothing estimates.
In Subsection \ref{SUBSEC:uniq}, 
we use this proposition to prove Theorem \ref{THM:2}, in particular uniqueness.
The proofs of Propositions \ref{PROP:main} and \ref{PROP:main2}
are  somewhat lengthy, 
involving an infinite iteration of normal form reductions.
We therefore postpone the proof of 
Propositions~\ref{PROP:main} and~\ref{PROP:main2}  to Section \ref{SEC:NF}.

\subsection{Energy estimate on the difference of two solutions}
\label{SUBSEC:energy2}
In this subsection, 
we consider an energy estimate for the difference of two solutions.
As pointed out in Remark \ref{REM:energy1}, 
there are two main sources of difficulty:
(i) the resonant contribution for the difference of solutions
does not vanish
and (ii) the symmetrization process in \eqref{E2} and \eqref{Psi} 
(for handling the non-resonant contribution) fails
for the difference of solutions; see \eqref{Diff3}.

Let us consider an energy estimate for the difference of two solutions
with the same initial condition.
Given $u_0 \in H^s(\T)$, $s > -\frac 9{20}$, 
let   $u$ and $v$ be two solutions to \eqref{4NLS1}
constructed in Section \ref{SEC:GWP}
with the same initial condition $u|_{t = 0} = v|_{t = 0} =  u_0$.
Then, we have 
$u, v \in C([-T, T]; H^s(\T) )\cap F^{s, \al}(T)$
for some $T = T(\|u_0\|_{H^s}) > 0$.
See Remark \ref{REM:LWP}.
Using the equation \eqref{4NLS1}, we have 
\begin{align}
\frac{d}{dt} \| u (t)- v(t) \|_{H^s}^2
& =   \frac{d}{dt} \sum_{n \in \Z} \langle n \rangle^{2s} |\ft u_n - \ft v_n|^2 \notag \\
&  =   2\Re \sum_{n \in \Z} \langle n \rangle^{2s} 
\frac{d}{dt} (\ft u_n - \ft v_n)\cdot  \cj{(\ft u_n - \ft v_n)}\notag\\
& = -  2\Re i \sum_{n \in \Z} \langle n \rangle^{2s} 
\big[\ft {\N(u)}_n -  \ft{\N(v)}_n\big]
 \cj{(\ft u_n - \ft v_n)} \notag \\
& \hphantom{XX}
+    2\Re i \sum_{n \in \Z} \langle n \rangle^{2s} 
\big[\ft {\RR(u)}_n - \ft {\RR(v)}_n\big]
 \cj{(\ft u_n - \ft v_n)}\notag\\
& = -  2\Re i \sum_{n \in \Z} \langle n \rangle^{2s} 
\big[\ft {\N(u)}_n -  \ft{\N(v)}_n\big]
 \cj{(\ft u_n - \ft v_n)} \notag \\
& \hphantom{XX}
- 2\Re i \sum_{n \in \Z} \jb{n}^{2s}  (|\ft u_n|^2 - |\ft v_n|^2)\cj{(\ft u_n - \ft v_n)} \ft v_n \notag\\
& =: \I + \II,
\label{Diff3}
\end{align}

 \noi
 where $\N(u)$ and $\RR(u)$ are as in \eqref{NN1} and \eqref{NN2}.

We first discuss how to handle the main difficulty (i).
The main  idea is  to perform normal form reductions infinitely many times
and
express
\[
|\ft u_n(t)|^2 - |\ft v_n(t)|^2= 
\big(|\ft u_n(t)|^2-|\ft u_n(0)|^2\big)- \big( |\ft v_n(t)|^2 - |\ft u_n(0)|^2\big) \]

\noi
in \eqref{Diff3} as 
the difference of sums of multilinear forms
of arbitrarily large degrees.

\begin{proposition}\label{PROP:main}

Let $s > -\frac{1}{3}$.
Then, there exist multilinear forms 
$\big\{\NN_0^{(j)}\big\}_{j = 2}^\infty$, 
$\big\{\RR^{(2)}\big\}_{j = 2}^\infty$, 
and 
$\big\{\NN_1^{(j)}\big\}_{j = 1}^\infty$, 
depending on a parameter $K > 0$, 
such that 
\begin{align}
 |\ft u_n(t)|^2 - |\ft u_n(0)|^2 
&  = \sum_{j=2}^\infty \NN_0^{(j)} (u)(n, t')\bigg|_0^t \notag\\
& \hphantom{XXX}+ 
 \int_0^t  \bigg[\sum_{j = 2}^\infty \RR^{(j)}( u)(n, t') 
 + \sum_{j = 1}^\infty \NN_1^{(j)} (u)(n, t')\bigg]dt
\label{main0}
\end{align}

\noi
for any solution  $u \in C(\R; H^s(\T))$ to \eqref{4NLS1}
with smooth (local-in-time) approximations.\footnote{Namely, 
given $t_0 \in \R$, 
there exists a sequence of smooth solutions $\{u_N\}_{N\in \NB}$
to \eqref{4NLS1}
and an interval $I \ni t_0$ such that 
$u_N$ tends to $u$ in $C(I; H^s(\T))$ as $N \to \infty$.
}
Here, 
$\NN_0^{(j)}$ is a  $2j$-linear form, 
while  $\RR^{(j)}$ and $\NN_1^{(j)}$ are $(2j+2)$-linear forms
(depending on $t \in \R$), satisfying the following bounds on $H^s(\T)$;
given any   $\theta \in (0, \frac{2}{3}]$,
there exist functions $C_{0, j}, C_{r, j}, C_{1, j}:\R_+ \to \R_+$, depending on $s$ and $\theta$, such that
\begin{align}
\label{main1} 
\sum_{n\in \Z} 
\Big\|\NN_0^{(j)} (f_1,f_2,\cdots,f_{2j})(n)\Big\|_{L^\infty_t(\R)} 
& \les  C_{0, j}  \prod_{i=1}^{2j} \|f_i\|_{H^s},\\
\sum_{n \in \Z} 
\Big\|\RR^{(j)}(f_1,f_2,\cdots,f_{2j})(n)\Big\|_{L^\infty_t(\R)} 
& \les C_{r, j}   \prod_{i=1}^{2j+2} \|f_i\|_{H^s}, 
\label{main2} \\
\sum_{n \in \Z}
\Big\|\NN_1^{(j)} (f_1,f_2,\cdots,f_{2j})(n)\Big\|_{L^\infty_t(\R)} 
& \les C_{1, j} \prod_{i=1}^{2j+2} \|f_i\|_{H^s},
\label{main3} 
\end{align}

\noi
for any $f_i\in H^s(\T)$ 
and $K > 0$, 
where
\begin{align*}
  C_{0, j}(K)&  = \begin{cases}
K^{\max(-\frac 12, -1- 2s)}  , & \text{if } j = 2,\\
K^{-\frac{(j-1)(1-\theta)}2} o(j^{-2}) & \text{if } j \geq 3,
\end{cases}\\
  C_{r, j}(K)&  = \begin{cases}
K^{ \max(-\frac{1}{2}, - 1-3s)} , & \text{if } j = 2,\\
K^{-\frac{(j-3)(1-\theta)}2}  o(j^{-2}), & \text{if } j \geq 3,
\end{cases}\\
  C_{1, j}(K)&  = \begin{cases}
K^{\frac 12- 2s} , & \text{if } j = 1,\\
K^{-\frac{(j-2)(1-\theta)}2}  o(j^{-2})& \text{if } j \geq 2.
\end{cases}
\end{align*}

\end{proposition}

It follows from the proof presented in Section \ref{SEC:NF}
that  the decay in $j$ 
is much faster than $j^{-2}$
but it suffices for our purpose in taking double difference
in \eqref{Y11a}.

Proposition \ref{PROP:main}
exhibits a smoothing property
analogous to Takaoka-Tsutsumi \cite{TT} in the context of the modified KdV on $\T$.
In \cite{TT}, Takaoka-Tsutsumi performed a normal form reduction
(= integration by parts) once. 
See also Nakanishi-Takaoka-Tsutsumi \cite{NTT}
and Molinet-Pilod-Vento \cite{MPV}, 
where the authors applied 
 normal form reductions twice in obtaining
effective energy estimates for the modified KdV on $\T$.
In order to maximize the smoothing effect, 
however, we instead perform normal form reductions infinitely many times
and re-express 
$ |\ft u_n(t)|^2 - |\ft u_n(0)|^2 $
as a sum of infinite series of multilinear forms
of arbitrarily large degrees.

Next, we turn our attention to the non-resonant part\,$\I$ in \eqref{Diff3}.
In this case, we can not apply the symmetrization argument as in Section \ref{SEC:energy}.
A straightforward energy estimate in terms of the $F^{s,\al}(T)$-norm
without symmetrization works only for $s > -\frac 3{10}$.
See Remarks \ref{REM:uniq} and \ref{REM:uniq2}.
In the following, we apply an infinite iteration
of normal form reductions 
to estimate the non-resonant part\,$\I$ in \eqref{Diff3}
and express $\I$ as a sum of infinite series consisting 
of multilinear terms in $u$ and $v$.
The following proposition 
follows as a corollary to Proposition \ref{PROP:main}.
See Subsection \ref{SUBSEC:8.6} for the proof.

\begin{proposition}\label{PROP:main2}

Let $s > -\frac{1}{3}$.
Then, there exists $T = T(\|u_0\|_{H^s}) > 0$ such that 
\begin{align*}
\bigg|\int_0^t \I(t') dt'\bigg|  & \leq \frac{1}{4} \|u-v\|_{C_TH^s}^2
\end{align*}

\noi
for any $t \in [-T, T]$ and any  two solutions\footnote{As in Proposition \ref{PROP:main}, 
it suffices to assume that 
$u, v \in C([-T, T]; H^s(\T) )$
are two solutions with smooth 
 (local-in-time) approximations.} 
$u, v \in C([-T, T]; H^s(\T) ) \cap F^{s, \al}(T)$  to \eqref{4NLS1}
constructed in Section \ref{SEC:GWP}
with  $u|_{t = 0} = v|_{t = 0} =  u_0 \in H^s(\T)$.

\end{proposition}

We postpone the proof of Propositions \ref{PROP:main}
and \ref{PROP:main2}
to Section \ref{SEC:NF}.
In the next subsection, we present the proof of Theorem \ref{THM:2},
assuming Propositions \ref{PROP:main} and \ref{PROP:main2}.

\subsection{Uniqueness and continuous dependence}\label{SUBSEC:uniq}

In this subsection, we use Propositions~\ref{PROP:main} and ~\ref{PROP:main2} to prove Theorem \ref{THM:2}.
Given  $s > -\frac{1}{3}$, 
let   $u, v \in   C([-T,T]; H^s(\T)) \cap F^{s,\alpha}(T)$ be two solutions to \eqref{4NLS1}
constructed in Section \ref{SEC:GWP}
with the same initial condition $u|_{t = 0} = v|_{t = 0} =  u_0 \in H^s(\T)$, 
satisfying 
\begin{align*}
 \|u\|_{C_T H^s}, \|v\|_{C_T H^s} \leq r
\end{align*}

\noi
for some $r > 0$.
Then, it follows from 
Proposition \ref{PROP:main}
and the multilinearity of 
$\NN_0^{(j)}$, 
 $\RR^{(j)}$,
 and $\NN_1^{(j)}$  that 
\begin{align*}
\sup_{n\in \Z}
 \big\|| & \ft u_n|^2 - |\ft v_n|^2\big\|_{L^\infty_T}
\leq \sum_{n\in \Z}
 \big\||\ft u_n|^2 - |\ft v_n|^2\big\|_{L^\infty_T}
\notag\\
&  = \sum_{n\in \Z}\Big\|\big(|\ft u_n|^2-|\ft u_n(0)|^2\big)- \big( |\ft v_n|^2 - |\ft u_n(0)|^2\big) \Big\|_{L^\infty_T} \notag\\
&  \leq \sum_{n\in \Z}\sum_{j=2}^\infty\Big\| \NN_0^{(j)} ( u)(n) -  \NN_0^{(j)}( v)(n)\Big\|_{L^\infty_T}\notag\\
& \hphantom{XXX}
+ T\sum_{n\in \Z} \sum_{j = 2}^\infty \Big\|
 \RR^{(j)} (u)(n) 
-  \RR^{(j)} ( v)(n) \Big\|_{L^\infty_T} \notag\\
& \hphantom{XXX}
+ T\sum_{n\in \Z}\sum_{j = 1}^\infty \Big\|\NN_1^{(j)} ( u)(n)
- \NN_1^{(j)} ( v)(n) \Big\|_{L^\infty_T} \notag\\
& \les
K^{\max(-\frac 12, -1- 2s)} r^{3} \|u-v\|_{C_TH^s}
+ \sum_{j = 3}^\infty 
 K^{-\frac{(j-1)(1-\theta)}2}  r^{2j-1} \|u-v\|_{C_TH^s}\notag\\
  &  \hphantom{XX}
+ 
TK^{\max(-\frac 12, -1- 3s)} r^{5} \|u-v\|_{C_TH^s}
+ T \sum_{j = 3}^\infty 
 K^{-\frac{(j-3)(1-\theta)}2}  r^{2j+1} \|u-v\|_{C_TH^s}\notag\\
 &  \hphantom{XX}
+ TK^{\frac{1}{2}-2s} r^{3} \|u-v\|_{C_TH^s}
+ T \sum_{j = 2}^\infty 
 K^{-\frac{(j-2)(1-\theta)}2}  r^{2j+1} \|u-v\|_{C_TH^s}.
\end{align*}

\noi
Then, by first choosing $K = K(r) > 0$ sufficiently large
and then choosing $T = T(K) = T(r) > 0$ sufficiently small, 
we conclude that 
\begin{align}
\sup_{n\in \Z}
 \big\|| & \ft u_n|^2 - |\ft v_n|^2\big\|_{L^\infty_T}
\leq \sum_{n\in \Z}
 \big\||\ft u_n|^2 - |\ft v_n|^2\big\|_{L^\infty_T}
\leq \frac{1}{16r}\| u - v \|_{C_TH^s}.
\label{unique2}
\end{align}

\noi
Hence, it follows from \eqref{unique2}
and  Cauchy-Schwarz inequality that 
\begin{align}
\bigg|\int_0^T \II(t) dt\bigg|
& \leq T \| \II \|_{L^\infty_T}  
= 2T\bigg\|   \sum_{n\in \Z} \jb{n}^{2s} 
(|\ft u_n|^2 - |\ft v_n|^2)\cj{(\ft u_n - \ft v_n)} \ft v_n \bigg\|_{L^\infty_T} \notag \\
& \leq \frac 14\| u - v \|_{C_TH^s}^2.
\label{energy4}
\end{align}

\noi
Therefore, by integrating \eqref{Diff3} from $0$ to $T$
with $u(0) = v(0)$
and applying Proposition \ref{PROP:main2} and \eqref{energy4}, we obtain 
\begin{align*}
 \| u - v \|_{C_T H^s}^2
& \leq \frac 12   \| u - v \|_{C_T H^s}^2.
\end{align*}

\noi
This proves local-in-time uniqueness of solutions to 
\eqref{4NLS1} in $C([-T, T]; H^s(\T))\cap F^{s, \al}(T)$
with some $T = T(\|u_0\|_{H^s}) > 0$.
In view of the global-in-time bound in Proposition \ref{PROP:GWP3},
we can iterate this argument and establish uniqueness
globally in time.
Here, uniqueness holds in 
\[
\bigcap_{t \in \R} \big\{u \in C(\R; H^s(\T));  u (\cdot - t) \in F^{s, \al}(T(t, u_0))\big\} \]

\noi
for some appropriate $ T(t, u_0)>0$.\footnote{Since we only need Propositions 
\ref{PROP:main} and \ref{PROP:main2}, 
the uniqueness holds among the solutions
in $C(\R; H^s(\T))$
with smooth approximations.
Note that in such a class, uniqueness is by no means automatic
since we do {\it not} have continuous dependence (at this point).
}

\begin{remark}\label{REM:uniq} \rm
(i) We stress that it is crucial that $u$ and $v$ have the same initial condition
in the argument above.

\smallskip

\noi
(ii) 
We can estimate the non-resonant contribution\,$\I$ in \eqref{Diff3}
in terms of the $F^{s, \al}(T)$-norm for $s > - \frac{3}{10}$.
See Remark \ref{REM:uniq2} below.
This provides uniqueness for a more restrictive range $s > -\frac{3}{10}$.

Note that an energy estimate
of form:
\begin{align*}
  \|u - v\|_{E^s(T)}^2   & \les 
    \|u(0) - v(0)\|_{H^s}^2  
    + 
T^\theta C\big( \| u\|_{F^{s, \al}(T)}, \| v\|_{F^{s, \al}(T)}\big)
\|u -  v\|_{F^{s, \al}(T)}^2
\end{align*}

\noi
for two solutions $u$ and $v$ with different initial data $u(0) \ne v(0)$
is {\it false} for $s < 0$
in view of the failure of local uniform continuity
of the solution map for \eqref{4NLS1} in negative Sobolev spaces.

\smallskip

\noi
(iii) 
By combining the proofs of Propositions~\ref{PROP:main}
and \ref{PROP:main2}, 
we can express
$ \| u (t)- v(t) \|_{H^s}^2$
as a sum of infinite series
consisting of multilinear linear terms (in $u$ and $v$)
of arbitrarily large degrees.
Moreover, 
thanks to the multilinearity of the summands
and the double difference structure of $\I$ and $\II$, 
we can rearrange the series so that 
we can extract two factors of (the Fourier coefficient of) $u-v$ in 
each of the multilinear terms.
See Remark \ref{REM:uniq3}.

\end{remark}

\smallskip
 
 Thanks to the uniqueness of solutions, 
  continuous dependence
 of the solution map for \eqref{4NLS1} on initial data in $H^s(\T)$
basically follows from repeating the argument in Section \ref{SEC:GWP}.

 \begin{lemma}\label{LEM:conti}
Given $s> - \frac1{3}$, 
let $\{u_n\}_{n \in \NB}$ and $u $ are the unique solutions 
to \eqref{4NLS1} in $C(\R; H^s(\T))$ with $u_n |_{t = 0} = u_{0, n}$
and $u|_{t = 0} = u_0$.
If we have
\begin{align*}
\lim_{n\to \infty} \|u_{0,n} -u_0\|_{H^s} = 0,
\end{align*}

\noi
then we have
\begin{align}
\lim_{n\to \infty} \|u_{n} -u\|_{C_TH^s} = 0
\label{conti2}
\end{align}

\noi
for any $T > 0$.
\end{lemma}

\begin{proof}
It suffices to prove \eqref{conti2} for sufficiently small $T>0$
since the general case follows from iterating local-in-time arguments
in view of the global-in-time bound in Proposition \ref{PROP:GWP3}.
Let $T = T(\|u_0\|_{H^s}) > 0$ be the local existence time
from Section \ref{SEC:GWP} for initial data of size 
$\|u_0\|_{H^s} + 1$.
Without loss of generality, we assume that 
$\sup_{n \in \Z} \| u_{0, n}\|_{H^s} \leq \|u_0\|_{H^s} + 1$.

Note that it suffices to prove that $\{ u_n\}_{n\in \NB}$
is precompact in $C([-T, T]; H^s(\T))\cap F^{s, \al}(T)$.
This implies that any subsequence of $\{ u_n\}_{n\in \NB}$
has a convergent subsubsequence.
In view of convergence to $u_0$ at time 0
and the uniqueness of solutions, 
such a convergent subsubsequence must converge to $u$
since it converges to $u_0$ at time 0.
Therefore, the entire sequence 
$\{ u_n\}_{n\in \NB}$ converges to $u$ in $C([-T, T]; H^s(\T))$, 
yielding  \eqref{conti2}.

Since $ u_{0, n}$ converges to $u_0$ in $H^s(\T)$, 
we see that $\{ u_{0, n}\}_{n\in \NB} \cup \{u_0\}$
is compact in $H^s(\T)$.  Then, by Riesz' characterization of compactness, 
given $\eps > 0$, there exists $N \in \NB$ such that 
\begin{align}
\| \P_{> N} u_{0, n}\|_{H^s} < \eps
\qquad \text{and}\qquad 
\| \P_{> N} u_{0}\|_{H^s} < \eps
\label{conti3}
\end{align}

\noi
for all $n \in \NB$.
Then, by exploiting the smoothing property of the energy estimate
\eqref{R4M1} in Proposition \ref{PROP:energy}
(see Remark \ref{REM:smoothing}) 
as  in Lemma 8.2 in \cite{GO}, 
we claim that, 
given  $\eps>0$,  there exists $ N_0 \in \NB$ such that
\begin{equation}
\|\P_{> N}u_n\|_{C_TH^s}<\eps
\label{conti4}
\end{equation}

\noi
for all $N \geq N_0$, uniformly in $n \in \NB$.
In view of 
Remark \ref{REM:smoothing}, 
it follows  from (the proof of) Proposition \ref{PROP:energy}
with the a priori bound:\footnote{The a priori bound \eqref{unique3}
follows from Lemma \ref{LEM:embed1}
and \eqref{localbound3}.}
\begin{equation}
\| u \|_{L^\infty([-T, T]; H^s)} \les \|u\|_{F^{s, \al}(T)}
\le  2\|u_0\|_{H^s} 
\label{unique3}
\end{equation}

\noi
that 
there exists small $\dl > 0$ such that 
\begin{align*}
\big|\| \P_{>N}  u_n\|^2_{E^s(T)}  -\|\P_{>N}u_{0, n}\|_{H^s}^2\big|
 & \les
T^\theta \|\P_{>cN}u_n\|_{F^{s-\dl,\al}(T)}^2\|u_n\|_{F^{s-\dl,\al}(T)}^2
\notag \\
& 
\les C(\|u_0\|_{H^s}) N^{-2\dl}
\longrightarrow 0, 
\end{align*}

\noi
as $N\to \infty$, uniformly in $n\in \mathbb N$. 
Hence, from \eqref{conti3}, 
there exists $N_0 \in \NB$  such that 
\begin{align}
\| \P_{>N} u_n\|_{C_T H^s}^2
& \leq \|\P_{>N} u_n \|_{E^s(T)}^2 
\les \| \P_{>N} u_{0, n}\|_{ H^s}^2
+ C(\|u_0\|_{H^s}) N^{-2\dl}
\les \eps
\label{conti5}
\end{align}

\noi
for all $N \geq N_0$, uniformly in $n \in \mathbb{N}$.
This proves \eqref{conti4}.

Fix $\eps >0$. By \eqref{conti4}, there exists $N_0>0$ such that
$\|\P_{> N_0}u_n\|_{C_TH^s}<\frac{\eps}{3}$ for all $n\in \mathbb N$.
Arguing as in the proof of Lemma 8.2 in \cite{GO}
with Ascoli-Arzel\`a compactness theorem,
we conclude that 
$\{\P_{\leq {N_0}}u_n\}_{n \in \mathbb{N}}$  is precompact in $C([-T, T]; H^s(\T))$.
Hence, there exists a finite cover by balls
of radius $\frac{\eps}{3}$
(in $C_TH^s$)
centered at $\{\P_{\leq {N_0}}u_{n_k}\}_{k = 1}^K$.
Then, the balls of radius $\eps$ (in $C_TH^s$)
centered at $\{u_{n_k}\}_{k = 1}^K$
covers $\{u_{n}\}_{n\in \mathbb N}$.
This proves the precompactness of 
 $\{ u_n\}_{n\in \NB}$
 in $C([-T, T]; H^s(\T))$.

Let us  extract a subsequence,
still denoted by
 $\{u_n\}_{n \in \mathbb{N}}$,  converging to some $u$ in $C([-T, T]; H^s(\T))$.
In view of the uniform tail estimate \eqref{conti5},
this subsequence  
also converges in $E^s(T)$.
Then, 
by  making $T$ smaller, if necessary, 
it follows from \eqref{K1} and \eqref{K2} that 
\begin{align*}
\|u_n-u_m\|_{F^{s,\alpha}(T)}\les \|u_n-u_m\|_{E^s(T)}.
\end{align*}

\noi
Hence, $\{u_n\}$ converges to $u$ in $F^{s,\alpha}(T)$.
\end{proof}

\section{Normal form reductions}
\label{SEC:NF}

It remains to prove 
  Propositions \ref{PROP:main}
and \ref{PROP:main2}.
In this section, we perform an infinite iteration of  normal form reductions
and present the proofs of these propositions 
in Subsections~\ref{SUBSEC:8.5} and \ref{SUBSEC:8.6}.

Let $u$ be a smooth global solution to the Wick ordered cubic 4NLS \eqref{4NLS1}
and $\u(t) = S(-t) u(t) $ be its interaction representation defined in \eqref{interaction}.
Then, by the fundamental theorem of calculus with \eqref{4NLS2}, 
we can write the growth of the energy quantity\footnote{The quantity  $|\ft \u_n(t)|^2$ is often referred to as an  action.}
$ |\ft u_n(t)|^2$ as 
\begin{align}
 |\ft u_n(t)|^2 - |\ft u_n(0)|^2 
&  =  |\ft \u_n(t)|^2 - |\ft \u_n(0)|^2 \notag \\
& = -2\Re i \bigg( \int_0^t \sum_{\G(n)} e^{-i \phi(\bar n) t'} \ft \u_{n_1}\cj{\ft \u_{n_2}}\ft \u_{n_3} \cj{\ft \u_n}(t') dt' \bigg)
\notag\\
\intertext{Integrating by parts in time,} 
& = 2\Re  \bigg( \sum_{\G(n)} \frac{e^{-i \phi(\bar n) t}}{\phi(\bar n)} \ft \u_{n_1}\cj{\ft \u_{n_2}}\ft \u_{n_3} \cj{\ft \u_n}
\bigg) \bigg|_0^t
\notag\\
& \hphantom{X} 
-2\Re  \bigg( \int_0^t \sum_{\G(n)} 
\frac{e^{-i \phi(\bar n) t}}{\phi(\bar n)}
\dt( \ft \u_{n_1}\cj{\ft \u_{n_2}}\ft \u_{n_3} \cj{\ft \u_n})(t') dt' \bigg).
\label{XX1}
\end{align}

\noi
In view of the factorization \eqref{Phi}, 
we see that the gain of $\phi(\bar n)$ in the denominators corresponds
to the gain of derivatives.
The price to pay here is that the second term on the right-hand side
of \eqref{XX1} is now 6-linear.
In order to handle the last term in \eqref{XX1}, we need to apply an integration by parts again, 
yielding 8-linear terms.
In fact, we iterate this procedure infinitely many times in the following.
When we apply integration by parts\footnote{In the following, 
we perform integration by parts without integration symbols, 
which we refer to as differentiation by parts, following \cite{BIT}.
Moreover, we perform integration by parts only in the case the phase factor is
``sufficiently large''.} in an iterative manner, 
 the time derivative may fall on any of the factors, 
 generating higher order nonlinear terms.
We need to keep track of all possible ways in which the time derivatives fall
and sum over the contributions from all possible choices.
This can be a combinatorially challenging task.
In order to handle multilinear terms of increasing complexity appearing
in the infinite iteration of normal form reductions, 
we introduce the notion of ordered bi-trees in the following.

\subsection{Ordered bi-trees}

In \cite{GKO}, the first author 
implemented
an infinite iteration of normal form reductions
to study the cubic NLS on $\T$,
where differentiation by parts was applied to the evolution equation satisfied by 
the interaction representation.
In \cite{GKO}, (ternary) trees and ordered trees played 
an important role 
for indexing such terms and frequencies arising in the general steps of  normal form reductions.

In the following, we instead implement an infinite iteration scheme
of normal form reductions applied to the energy quantity\footnote{More precisely, 
to the evolution equation satisfied by the energy quantity.}
$|\ft \u_n(t)|^2$.
 In particular,  we need tree-like structures that {\it grow in two directions}.
For this purpose,  we introduce the notion of  bi-trees and ordered bi-trees
in the following.
Once we replace trees and ordered trees
by bi-trees and ordered bi-trees,
other related notions can be defined
in a similar manner as in \cite{GKO}
with certain differences to be noted.

\begin{definition} \label{DEF:tree2} \rm
(i) Given a partially ordered set $\TT$ with partial ordering $\leq$, 
we say that $b \in \TT$ 
with $b \leq a$ and $b \ne a$
is a child of $a \in \TT$,
if  $b\leq c \leq a$ implies
either $c = a$ or $c = b$.
If the latter condition holds, we also say that $a$ is the parent of $b$.

\smallskip

\noi
(ii) A tree $\TT $ is a finite partially ordered set satisfying
the following properties:
\begin{enumerate}

\item[(a)] Let $a_1, a_2, a_3, a_4 \in \TT$.
If $a_4 \leq a_2 \leq a_1$ and  
$a_4 \leq a_3 \leq a_1$, then we have $a_2\leq a_3$ or $a_3 \leq a_2$,

\item[(b)]
A node $a\in \TT$ is called terminal, if it has no child.
A non-terminal node $a\in \TT$ is a node 
with  exactly three ordered children denoted by $a_1, a_2$, and $a_3$.

\item[(c)] There exists a maximal element $r \in \TT$ (called the root node) such that $a \leq r$ for all $a \in \TT$.

\item[(d)] $\TT$ consists of the disjoint union of $\TT^0$ and $\TT^\infty$,
where $\TT^0$ and $\TT^\infty$
denote  the collections of non-terminal nodes and terminal nodes, respectively.
\end{enumerate}

\smallskip

\noi
(iii) A {\it bi-tree} $\TT = \TT_1 \cup \TT_2$ is 
a union of two trees $\TT_1$ and $\TT_2$,
where the root nodes $r_j$ of $\TT_j$, $j = 1, 2$,  are joined by an edge.
A bi-tree
$\TT$ consists of the disjoint union of $\TT^0$ and $\TT^\infty$,
where $\TT^0$ and $\TT^\infty$
denote  the collections of non-terminal nodes and terminal nodes, respectively.
By convention, we assume that the root node $r_1$ of the first tree $\TT_1$ is non-terminal,
while the root node $r_2$ of the second tree $\TT_2$ may be terminal.

\smallskip

\noi
(iv) Given a bi-tree $\TT = \TT_1 \cup \TT_2$, 
we define a projection $\Pi_j$, $j = 1, 2$, onto  a tree
by setting 
\[\Pi_j(\TT) = \TT_j.\]

\noi
In Figure \ref{FIG:1}, 
$\Pi_1(\TT)$ corresponds to the tree on the left under the root node $r_1$, 
while 
$\Pi_2(\TT)$ corresponds to the tree on the right under the root node $r_2$.

\end{definition}

\smallskip

Note that the number $|\TT|$ of nodes in a bi-tree $\TT$ is $3j+2$ for some $j \in \mathbb{N}$,
where $|\TT^0| = j$ and $|\TT^\infty| = 2j + 2$.
Let us denote  the collection of trees of the $j$th generation 
(namely,  with $j$ parental nodes) by $BT(j)$, i.e.
\begin{equation*}
BT(j) := \{ \TT : \TT \text{ is a bi-tree with } |\TT| = 3j+2 \}.
\end{equation*}

\begin{figure}[h]
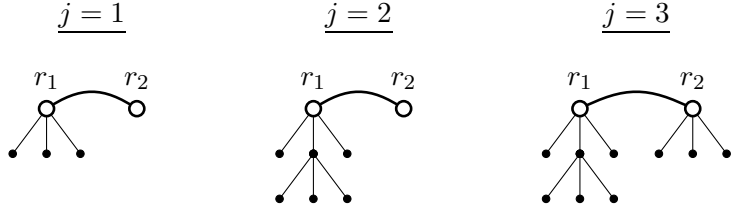

\[ \<T1>  
\qquad\qquad
\<T2> 
\qquad  \qquad
\<T3>\]

\caption{Examples of bi-trees of the $j$th generations, $j = 1, 2, 3$}
\label{FIG:1}
\end{figure}

\vspace{3mm}

Next, we introduce the  notion of ordered bi-trees,
for which we  keep track of how a bi-tree ``grew''
into a given shape.

\begin{definition} \label{DEF:tree3} \rm
 We say that a sequence $\{ \TT_j\}_{j = 1}^J$ is a chronicle of $J$ generations, 
if 
\begin{enumerate}
\item[(a)] $\TT_j \in BT(j)$ for each $j = 1, \dots, J$,

\smallskip

\item[(b)]  $\TT_{j+1}$ is obtained by changing one of the terminal
nodes in $\TT_j$ into a non-terminal node (with three children), $j = 1, \dots, J - 1$.
\end{enumerate}

\noi
Given a chronicle $\{ \TT_j\}_{j = 1}^J$ of $J$ generations,  
we refer to $\TT_J$ as an {\it ordered bi-tree} of the $J$th generation.
We denote the collection of the ordered trees of the $J$th generation
by $\mathfrak{BT}(J)$.
Note that the cardinality of $\mathfrak{BT}(J)$ is given by 
$ |\mathfrak{BT}(1)| = 1$ and 
\begin{equation} 
\label{cj1}
 |\mathfrak{BT}(J)| = 4\cdot 6 \cdot 8 \cdot \cdots \cdot 2J 
 = 2^{J-1}   \cdot J!=: c_J,
 \quad J \geq 2.
 \end{equation}

\end{definition}

\smallskip

We stress that the notion of ordered bi-trees comes with associated chronicles.
For example, 
given two ordered bi-trees $\TT_J$ and $\wt{\TT}_J$
of the $J$th generation, 
it may happen that $\TT_J = \wt{\TT}_J$ as bi-trees (namely as planar graphs) 
according to Definition \ref{DEF:tree2},
while $\TT_J \ne \wt{\TT}_J$ as ordered bi-trees according to Definition \ref{DEF:tree3}.
In the following, when we refer to an ordered bi-tree $\TT_J$ of the $J$th generation, 
it is understood that there is an underlying chronicle $\{ \TT_j\}_{j = 1}^J$.

\smallskip

Given a bi-tree $\TT$, 
we associate each terminal node $a \in \TT^\infty$ with the Fourier coefficient (or its complex conjugate) of the interaction representation 
$\u$ and sum over all possible frequency assignments.
In order to do this, we introduce the index function
assigning 
frequencies to {\it all} the nodes in $\TT$ in a consistent manner.

\begin{definition} \label{DEF:tree4} \rm
(i) Given  a bi-tree $\TT = \TT_1\cup \TT_2$, 
we define an index function ${\bf n}: \TT \to \mathbb{Z}$ such that
\begin{itemize}

\item[(a)] $n_{r_1} = n_{r_2}$, where $r_j$ is the root node of the tree $\TT_j$, $j = 1, 2$,

\item[(b)] $n_a = n_{a_1} - n_{a_2} + n_{a_3}$ for $a \in \TT^0$,
where $a_1, a_2$, and $a_3$ denote the children of $a$,

\item[(c)] $\{n_a, n_{a_2}\} \cap \{n_{a_1}, n_{a_3}\} = \emptyset$ for $a \in \TT^0$, 

\end{itemize}

\noi
where  we identified ${\bf n}: \TT \to \mathbb{Z}$ 
with $\{n_a \}_{a\in \TT} \in \mathbb{Z}^\TT$. 
We use 
$\mathfrak{N}(\TT) \subset \mathbb{Z}^\TT$ to denote the collection of such index functions ${\bf n}$
on $\TT$.

\smallskip

\noi
(ii) Given a tree $\TT$, we also define 
 an index function ${\bf n}: \TT \to \mathbb{Z}$ 
 by omitting the condition (a)
 and denote by 
 $\mathfrak{N}(\TT) \subset \mathbb{Z}^\TT$  the collection of index functions ${\bf n}$
 on $\TT$, when there is no confusion.

\end{definition}

\begin{remark} \label{REM:terminal}
\rm 
(i) In view of the consistency condition, 
we can refer to $n_{r_1} = n_{r_2}$
as the frequency  at the root node without ambiguity.
We shall  denote it by $n_r$.

\smallskip

\noi
(ii) 
Just like  index functions  for (ordered) trees considered in \cite{GKO}, 
an index function ${\bf n} = \{n_a\}_{a\in\TT}$ for a bi-tree $\TT$ is completely determined
once we specify the values $n_a \in \Z$ for the terminal nodes $a \in \TT^\infty$.
An index function $\bn$
for a bi-tree $\TT = \TT_1 \cup \TT_2$
is basically a pair $(\bn_1, \bn_2)$
of index functions $\bn_j$ for the trees $\TT_j$, $j = 1, 2$, (omitting
the non-resonance condition in \cite[Definition 3.5 (iii)]{GKO}),
satisfying the consistency condition (a): $n_{r_1} = n_{r_2}$.

\smallskip

\noi
(iii) Given a bi-tree $\TT \in \mathfrak{BT}(J)$, consider
 the summation of all possible frequency assignments
 $\{ \bn \in \mathfrak{N}(\TT): n_r = n\}$.
While $|\TT^\infty| = 2J + 2$, 
there are $2J$ free variables in this summation.
Namely, the condition $n_r = n$ reduces two summation variables.
It is easy to see this by separately considering
the cases $\Pi_2(\TT) = \{r_2\}$
and $\Pi_2(\TT) \ne \{r_2\}$.
\end{remark}

\medskip

Given an ordered bi-tree 
$\TT_J$ of the $J$th generation with a chronicle $\{ \TT_j\}_{j = 1}^J$ 
and associated index functions ${\bf n} \in \mathfrak{N}(\TT_J)$,
we would like to keep track of the  ``generations'' of frequencies
as in \cite{GKO}.
In the following,  we use superscripts to denote such generations of frequencies.

Fix ${\bf n} \in \mathfrak{N}(\TT_J)$.
Consider $\TT_1$ of the first generation.
Its nodes consist of the two root nodes $r_1$, $r_2$, 
and the children $r_{11}, r_{12}, $ and $r_{13}$ of the first root node $r_1$. 
See Figure \ref{FIG:1}.
We define the first generation of frequencies by
\[\big(n^{(1)}, n^{(1)}_1, n^{(1)}_2, n^{(1)}_3\big) :=(n_{r_1}, n_{r_{11}}, n_{r_{12}}, n_{r_{13}}).\]

\noi
From Definition \ref{DEF:tree4}, we have
\begin{equation*}
n^{(1)}  = n_{r_2}, \quad  n^{(1)} = n^{(1)}_1 - n^{(1)}_2 + n^{(1)}_3, \quad n^{(1)}_2\ne n^{(1)}_1, n^{(1)}_3.
\end{equation*}

Next, we construct  an ordered bi-tree $\TT_2$ of the second generation 
from $\TT_1$ by
changing one of its terminal nodes $a \in \TT^\infty_1 = \{ r_2, r_{11}, r_{12}, r_{13}\}$ 
into a non-terminal node.
Then, we define
the second generation of frequencies by setting
\[\big(n^{(2)}, n^{(2)}_1, n^{(2)}_2, n^{(2)}_3\big) :=(n_a, n_{a_1}, n_{a_2}, n_{a_3}).\]

\noi
Note that  we have $n^{(2)} = n^{(1)}$ or $n_k^{(1)}$ for some $k \in \{1, 2, 3\}$, 
\begin{equation*}
 n^{(2)} = n^{(2)}_1 - n^{(2)}_2 + n^{(2)}_3, \quad n^{(2)}_2\ne n^{(2)}_1, n^{(2)}_3,
\end{equation*}

\noi
where the last identities follow from Definition \ref{DEF:tree4}.
This extension of $\TT_1 \in \mathfrak{BT}(1)$ to $\TT_2\in \mathfrak{BT}(2)$
corresponds to introducing a new set of frequencies
after the first differentiation by parts,
where  the time derivative falls on each of $\ft \u_n$ and $\ft \u_{n_j}$, $j = 1, 2, 3$.\footnote{The complex conjugate signs on  $\ft \u_n$ and $\ft \u_{n_j}$ do not play any significant role.
Hereafter,  we drop the complex conjugate sign.
We also assume that all the Fourier coefficients of $\u$ are non-negative.}

In general, we construct  an ordered bi-tree $\TT_j$ 
of the $j$th generation from $\TT_{j-1}$ by
changing one of its terminal nodes $a  \in \TT^\infty_{j-1}$
into a non-terminal node.
Then, we define
the $j$th generation of frequencies by
\[\big(n^{(j)}, n^{(j)}_1, n^{(j)}_2, n^{(j)}_3\big) :=(n_a, n_{a_1}, n_{a_2}, n_{a_3}).\]

\noi
As before, it follows from Definition \ref{DEF:tree4} that 
\begin{equation*}
 n^{(j)} = n^{(j)}_1 - n^{(j)}_2 + n^{(j)}_3, \quad n^{(j)}_2\ne n^{(j)}_1, n^{(j)}_3.
\end{equation*}

\noi
Given an ordered bi-tree $\TT$, we denote by $B_j = B_j(\TT)$
the set of all possible frequencies in the $j$th generation.
Figure \ref{FIG:2} below shows an example of a bi-tree $\TT \in \mathfrak{BT}(3)$
ornamented by an index function $\bn \in \mathfrak{N}(\TT)$.
\begin{figure}[h]
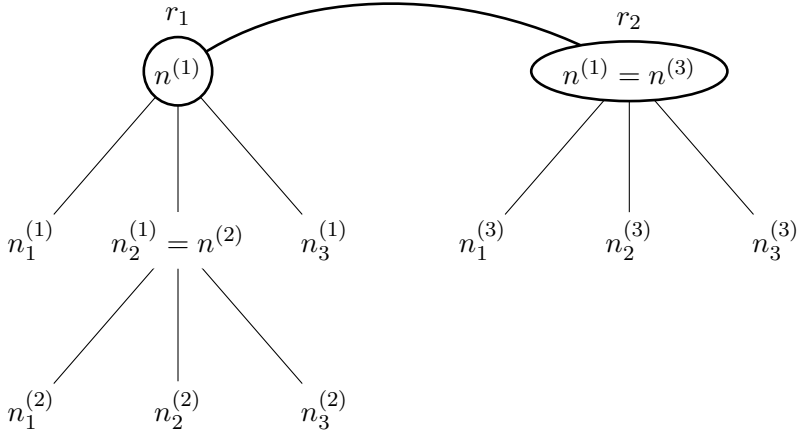

\[ \<T4>  \]

\caption{An example of a bi-tree $\TT \in \mathfrak{BT}(3)$.
Here, we have ornamented the nodes with the values of 
an index function ${\bf n} = \{n_a\}_{a\in\TT}
\in \mathfrak{N}(\TT)$,
specifying the generations of frequencies as discussed above.
}
\label{FIG:2}
\end{figure}

We denote by  $\phi_j$  the corresponding phase function introduced at the $j$th generation:
\begin{align}
\phi_j & =  \phi_j \big( n^{(j)},  n^{(j)}_1, n^{(j)}_2, n^{(j)}_3\big)
:=   \big(n_1^{(j)}\big)^4 - \big(n_2^{(j)}\big)^4 + \big(n_3^{(j)}\big)^4
- \big(n^{(j)}\big)^4.
\label{phi}
\end{align}

\noi
Then, by \eqref{Phi}, we have 
\begin{align*}
|\phi_j| \sim (n^{(j)}_\text{max})^2
\cdot| \big(n^{(j)} - n_1^{(j)}\big) \big(n^{(j)} - n_3^{(j)}\big)|,
\end{align*}

\noi
where $n^{(j)}_\text{max}: = \max\big(|n^{(j)}|, 
|n_1^{(j)}|, |n_2^{(j)}|, |n_3^{(j)}| \big)$.
Lastly, we denote by $\mu_j$
 the  phase function  (at the $j$th generation) corresponding to the usual cubic NLS 
 with the second order dispersion:
\begin{align*}
\mu_j & =  \mu_j \big( n^{(j)},  n^{(j)}_1, n^{(j)}_2, n^{(j)}_3\big)
:=   \big(n_1^{(j)}\big)^2 - \big(n_2^{(j)}\big)^2 + \big(n_3^{(j)}\big)^2
- \big(n^{(j)}\big)^2\notag \\
& =-2 \big(n^{(j)} - n_1^{(j)}\big) \big(n^{(j)} - n_3^{(j)}\big).
\end{align*}

\noi
Note that we have
\begin{align}
|\phi_j | \sim (n_{\max}^{(j)})^2 \cdot |\mu_j| 
\ges |\mu_j|^2.
\label{fphi}
\end{align}

\subsection{First few steps of normal form reductions}\label{SUBSEC:8.2}
We first implement
a formal infinite iteration scheme of normal form reductions
for smooth functions
without justifying switching of limits and summations.
As before, 
let $u$ be a smooth global solution to  \eqref{4NLS1}
and $\u(t) = S(-t) u(t) $ be its interaction representation.
For simplicity of notations, we 
simply set $\u_n = \ft \u_n$ in the following.
We may also drop the minus signs and the complex number~$i$.
In the following, we establish various multilinear estimates.
Our argument has a common feature with \cite{GKO}
in that Cauchy-Schwarz inequality plays an important role.
On the other hand, 
while the divisor counting argument played a crucial role in \cite{GKO}, 
 we do {\it not} use the divisor counting argument
in maximizing a gain of derivative.\footnote{We, however, 
use the divisor counting argument to show that the error term converges to 0,
where we do not need to show any gain of derivatives.
See Subsection~\ref{SUBSEC:error}.}

Given   $s >  - \frac13$, 
fix $K  = K(s) > 0$ (to be chosen later).\footnote{As we see in Subsection \ref{SUBSEC:8.6},
the constant $K$ will also depend other constants.}
Using the notations introduced in the previous subsection, for fixed $n$, 
we have\footnote{Due to the presence of 
$e^{-i \phi_1 t}$, 
the multilinear form  $\NN^{(1)}(\u )_n$
is  non-autonomous in $t$. 
Hence,  strictly speaking, 
we should denote it by   $\NN^{(1)}(t)(\u(t))_n$.
In the following, however, we estimate these 
 multilinear forms, uniformly in  $t \in \R$,
 and thus 
we simply suppress such $t$-dependence
when there is no confusion.
The same comment applies to other multilinear forms.
}
\begin{align}
  \frac{d}{dt} |\u_n(t) |^2
& = -2 \Re i 
 \sum_{\G(n)}
e^{ - i \phi(\bar n) t}  \u_{n_1} \cj{\u_{n_2}} \u_{n_3} \cj{\u_n}\notag\\
& = 
- 2 \Re i 
\sum_{\TT_1 \in \mathfrak{BT}(1)}
\sum_{\substack{{\bf n} \in \mathfrak{N}(\TT_1)\\ n_r = n} }
e^{  - i  \phi_1 t }  \prod_{a \in \TT^\infty_1} \u_{n_{a}}
= : \NN^{(1)}(\u )_n,
\label{1-g}
\end{align}

\noi
where $\G(n)$ is as in \eqref{Gam1}.
We divide the frequency space 
into $|\phi_1| \leq K $ and $|\phi_1|> K$.
Namely, define $A_K$ by 
\begin{align*}
A_K:= \big\{ 
{\bf n} \in \mathfrak{N}(\TT_1) :\, 
& |\phi_1(\bn)|  \leq K, n_r = n  \big\}
\end{align*}

\noi
and write
\begin{equation*}
\NN^{(1)} = \NN_{1}^{(1)} + \NN^{(1)}_2, 
\end{equation*}

\noi
where $\NN_{1}^{(1)}$ is the restriction of $\NN^{(1)}$
onto $A_K$
and 
$ \NN^{(1)}_2 = \NN^{(1)} - \NN_{1}^{(1)}$.  
Thanks to the restriction $|\phi_1|\leq K$, 
we can  estimate the nearly resonant part $\NN_{1}^{(1)}$ as follows.

\begin{lemma}\label{LEM:N^1_1}
Let $\NN^{(1)}_1$ be as above.
Then, for any $s\ \le 0$, 
we have 
\begin{align} 
\label{N^1-1}
\sum_{n\in \Z}| \NN^{(1)}_1(\u)_n| & 
\les K^{\frac 12 -2s} \|\u\|_{H^s}^4. 
\end{align}

\end{lemma}

\begin{remark}\rm
In Lemma \ref{LEM:N^1_1}, 
we established an $\l^1_n$-bound on 
$\big\{\NN^{(1)}_1(\u)_n\big\}_{n \in \Z}$.
In this and the next subsections, we estimate various multilinear
terms in the $\l^1_n$-norm.
We point out that,
in  proving
uniqueness of solutions to \eqref{4NLS1} with the same initial condition,
 it suffices to 
estimate these multilinear terms only in the much weaker $\l^\infty_n$-norm.
Unfortunately, we do not know how to convert 
this gain in summability  to a gain in differentiability
to go below $-\frac 13$.
See Lemma \ref{LEM:N^2_1}.

\end{remark}

\begin{proof}

For notational simplicity, we drop the superscript $(1)$
in the frequencies $n^{(1)} = n_r$ and $n^{(1)}_j$.
In view of \eqref{Phi}, 
the condition $0 < |\phi_1| \leq K$ implies
that 
\begin{align}\label{K1a}
\jb{n}^{-s} \jb{n_1}^{-s}  \jb{n_2}^{-s}  \jb{n_3}^{-s} \les n_{\max}^{-4s}  \leq K^{-2s}
\end{align}

\noi
on $\G(n)$, provided that  $s\leq0$.
Then, 
by crudely estimating the contribution
with Cauchy-Schwarz inequality,  \eqref{K1a}, and $| \mathfrak{BT}(1)| = 1$, 
we have
\begin{align*}
\sum_{n\in \Z} | \NN^{(1)}_1 & (\u)_n|
 \les
\sum_{n\in \Z} 
\sum_{\TT_1 \in \mathfrak{BT}(1)} 
\sum_{|\phi|\leq K }
\sum_{\substack{{\bf n} \in \mathfrak{N}(\TT_1)\\n_r = n\\ \phi_1 = \phi}} 
\prod_{a \in \TT^\infty_1} |\u_{n_{a}}|
 \\
& \les \|\u\|_{H^s} \Bigg\{ \sup_{n\in \Z} 
\bigg( \sum_{\G(n)} 
\frac{K^{1 -4s}}{n_{\max}^{2}|(n-n_1)(n-n_3)|} \bigg)
\cdot 
\bigg(\sum_{n\in \Z} \sum_{\G(n)} 
\prod_{i= 1}^3 \jb{n_i}^{2s} |\u_{n_i}|^2 
\bigg)
\Bigg\}^\frac{1}{2}\\
& \lesssim K^{\frac 12  -2s } \|\u\|_{H^s}^4. 
\end{align*}

\noi
This proves \eqref{N^1-1}.	
Note that the power of $K$ is by no means sharp.
\end{proof}

Next, we consider the  non-resonant term $ \NN^{(1)}_2(\u).$
It turns out that there is no effective estimate for $ \NN^{(1)}_2(\u)$ 
and thus we perform a normal form reduction.
In the following, we restrict our discussion 
to 
\begin{align}
|\phi_1|> K,
\label{mu1}
\end{align}

\noi
namely, the set of frequencies are restricted onto $A_K^c$.
When it is clear from the context, however, 
we suppress such restriction
for notational simplicity.
Differentiating by parts, 
i.e.~integrating by parts without an integral sign,
we obtain
\begin{align}
 \NN^{(1)}_2(\u)_n
& = 
2 \Re \dt \bigg[
\sum_{\TT_1 \in \mathfrak{BT}(1)}
\sum_{\substack{ \n \in \mathfrak{N}(\TT_1)\\n_r = n }}
\frac{  e^{- i  \phi_1 t } }{\phi_1}
\prod_{a \in \TT_1^\infty} \u_{n_{a}}
\bigg]\notag \\
& \hphantom{X}
- 2\Re 
\sum_{\TT_1 \in \mathfrak{BT}(1)}
\sum_{\substack{ \n \in \mathfrak{N}(\TT_1)\\n_r = n }}
\frac{ e^{- i  \phi_1 t } }{\phi_1}
\dt\bigg(\prod_{a \in \TT^\infty_1} \u_{n_{a}}\bigg)
\notag\\
& = 2 \Re \dt \bigg[
\sum_{\TT_1 \in \mathfrak{BT}(1)}
\sum_{\substack{ \n \in \mathfrak{N}(\TT_1)\\n_r = n }}
\frac{ e^{- i  \phi_1 t } }{\phi_1}
\prod_{a \in \TT^\infty_1} \u_{n_{a}}
\bigg]\notag \\
& \hphantom{X}
- 2  \Re 
\sum_{\TT_1 \in \mathfrak{BT}(1)}
\sum_{b\in \TT_1^\infty}
\sum_{\substack{ \n \in \mathfrak{N}(\TT_1)\\n_r = n }}
\frac{ e^{- i  \phi_1 t } }{\phi_1}
\textsf{R}(\u)_{n_b}
\prod_{a \in \TT^\infty_1\setminus\{b\}} \u_{n_{a}}\notag \\
& \hphantom{X}
- 2\Re 
\sum_{\TT_2 \in \mathfrak{BT}(2)}
\sum_{\substack{ \n \in \mathfrak{N}(\TT_2)\\n_r = n }}
\frac{ e^{- i  ( \phi_1+ \phi_2) t } }{\phi_1}
\prod_{a \in \TT^\infty_2} \u_{n_{a}}
 \notag\\
& =: \dt \NN_0^{(2)}(\u)_n +  \RR^{(2)}(\u)_n + \NN^{(2)}(\u)_n.
\label{2-g}
\end{align}

\noi
In the second equality, we
applied the product rule and 
used the equation \eqref{4NLS2}
to replace  $\dt \u_{n_b}$ by 
the resonant part $\textsf{R}(\u)_{n_b}$ and the non-resonant part $\textsf{N}(\u)_{n_b}$.
Note that substituting the non-resonant part $\textsf{N}(\u)_{n_b}$
amounts to 
extending the tree $\TT_1\in \mathfrak{BT}(1)$
(and ${\bf n }\in \mathfrak{N}(\TT_1)$)
to $\TT_2 \in \mathfrak{BT}(2)$
(and to  ${\bf n }\in \mathfrak{N}(\TT_2)$, respectively)
by replacing  the terminal node $b \in \TT^\infty_1$
into a non-terminal node with three children $b_1, b_2,$ and $b_3$.

\begin{remark} \label{REM:2Re} \rm

Strictly speaking, the  phase factor appearing in $\NN^{(2)}(\u)$
may be $\phi_1 - \phi_2$
when the time derivative falls on the terms with the complex conjugate.
In the following, however, we simply write it as $\phi_1 + \phi_2$ since
it does not make any difference for our analysis.
Also, we often replace $\pm 1$ and $\pm i$ by $1$
for simplicity when they do not play an important  role.
Lastly, for notational simplicity, 
we drop  twice the real part symbol ``$2 \Re$''
on  multilinear forms,
but it is understood that all the multilinear forms
appear with 
twice  the real part symbol.

\end{remark}

We first estimate  the boundary term $\NN_0^{(2)}$.

\begin{lemma}
\label{LEM:N^2_0}
Let $\NN^{(2)}_0$ be as in \eqref{2-g}.
Then, for  $s \geq  -\frac12$, we have
\begin{align} \label{N^2-0}
\sum_{n\in \Z} | \NN^{(2)}_0(\u)_n| 
& \les   K^{\max(-\frac 12, -1- 2s)} \|\u\|_{H^s}^4.
\end{align}

\end{lemma}

\begin{proof}

As in the proof of Lemma \ref{LEM:N^1_1},   we drop the superscript $(1)$.
From \eqref{Phi}, we have
\begin{align}
\sup_{n\in \Z}\sum_{\substack{\G(n)\\|\phi_1|> K}}
 \frac{n_{\max}^{-8s}}{|\phi_1|^2}
\les\sup_{n\in \Z}\sum_{\substack{\G(n)\\|\phi_1|> K}}
\frac{1}{|(n-n_1)(n-n_3)|^2 {n_{\max}^{4+8s}}}
\les K^{ \max(-1, -2 - 4s)}
\label{N^2-0a}
\end{align}

\noi
for  $ s\geq - \frac 12$. 
Then, 
by Cauchy-Schwarz inequality with \eqref{N^2-0a} and  $| \mathfrak{BT}(1)| = 1$, we have
\begin{align*}
\sum_{n\in \Z} | \NN^{(2)}_0(\u)_n|
& \les 
\sum_{\TT_1 \in \mathfrak{BT}(1)}
\sum_{n\in \Z}
\sum_{\substack{{\bf n} \in \mathfrak{N}(\TT_1)\\n_r = n \\|\phi_1| >K} }
\frac{n_{\max}^{-4s}}{|\phi_1|}\prod_{a \in \TT^\infty_1} \jb{n_a}^s \u_{n_{a}}
 \\
& \leq  \| \u\|_{H^s} \Bigg\{ 
\bigg( \sup_{n\in \Z}\sum_{\substack{\G(n)\\|\phi_1|> K}}
 \frac{n_{\max}^{-8s}}{|\phi_1|^2}\bigg)
\cdot \bigg(\sum_{n\in \Z} \sum_{\G(n)}  \prod_{i =1}^3 \jb{n_i}^{2s}|\u_{n_i}|^2 \bigg)
\Bigg\}^\frac{1}{2}\\
&  \lesssim K^{  \max(-\frac 12, -1- 2s)} \|\u\|_{H^s}^4.
\end{align*}

\noi
This proves \eqref{N^2-0}.
\end{proof}

The following estimate on $\RR^{(2)}$ is an immediate corollary
to Lemma \ref{LEM:N^2_0}.

\begin{lemma}
\label{LEM:N^2_r}
Let $\RR^{(2)}$ be as in \eqref{2-g}.
Then, for  $s \ge - \frac13$, we have
\begin{align*} 
\sum_{n\in \Z} | \RR^{(2)}(\u)_n| 
& \les   K^{ \max(-\frac{1}{2}, - 1-3s)} \|\u\|_{H^s}^6.
\end{align*}

\end{lemma}

\begin{proof}
This lemma follows from  the proof of Lemma \ref{LEM:N^2_0}
and $\l^2 \subset \l^6$
once we observe that 
\begin{align*}
\sum_{\substack{\G(n)\\|\phi_1|> K}}
 \frac{n_{\max}^{-12s}}{|\phi_1|^2} 
 \les \sum_{\substack{\G(n)\\|\phi_1|> K}}
 \frac{1}{|\mu_1|^{2} \jb{n_{\max}}^{4+12s}}
 \les K^{ \max(-1, -2 - 6s)}, 
\end{align*}

\noi
provided that  $s \ge -\frac13$.
\end{proof}

As in the first step of the normal form reductions, 
we can not estimate 
$\NN^{(2)}$ as it is.
By dividing the frequency space into 
\begin{equation} \label{C1}
C_1 = \big\{ |\phi_1 + \phi_2| \lesssim 6^3 |\phi_1|^{1-\theta}\big\} 
\end{equation}

\noi
for some $\theta \in (0, 1)$ (to be chosen later)
and its complement $C_1^c$,\footnote{Clearly, the number $6^3$ in \eqref{C1} 
does not make any difference at this point.
However, we insert it to match with \eqref{CJ}.
See also  \eqref{C2}.}
split $\NN^{(2)}$ as 
\begin{equation} \label{N^2_1}
\NN^{(2)} = \NN^{(2)}_1 + \NN_2^{(2)},
\end{equation}

\noi
where $\NN^{(2)}_1$ is the restriction of $\NN^{(2)}$
onto $C_1$
and
$\NN_2^{(2)} := \NN^{(2)} - \NN^{(2)}_1$.
Thanks to the frequency restriction, 
we can estimate the first term $\NN^{(2)}_1$ as follows.

\begin{lemma}\label{LEM:N^2_1}
Let $\NN^{(2)}_1$ be as in \eqref{N^2_1}.
Then,
for  $s > - \frac{1}3$, 
 we have
\begin{align*}
\sum_{n\in \Z} | \NN^{(2)}_1(\u)_n| 
\les 
 \|\u\|_{H^s}^6.
\end{align*}
\end{lemma}

Before presenting the proof of Lemma \ref{LEM:N^2_1}, 
let us briefly describe 
how to handle 
the second term of $\NN^{(2)}_2$.
On the support of $\NN^{(2)}_2$, we have 
\begin{equation} \label{mu2}
|\phi_1 + \phi_2|  \gg 6^3 |\phi_1|^{1-\theta} > 6^3 K^{1-\theta}.
\end{equation}

\noi
Namely, the phase function $\phi_1 + \phi_2$ is ``large'' in this case
and hence we can exploit this fast oscillation
by applying the second step of the normal form reduction:
\begin{align} \label{N^3}
\NN^{(2)}_2 (\u)_n
& = 
 \dt \bigg[\sum_{\TT_2 \in \mathfrak{BT}(2)}
\sum_{\substack{\n \in \mathfrak{N}(\TT_2)\\n_r = n}}
\frac{ e^{- i( \phi_1 + \phi_2 )t } }{\phi_1(\phi_1+\phi_2)}
\prod_{a \in \TT^\infty_2} \u_{n_{a}} \bigg]\notag \\
& \hphantom{X} -  
\sum_{\TT_2 \in \mathfrak{BT}(2)}
\sum_{\substack{\n \in \mathfrak{N}(\TT_2)\\n_r = n}}
\frac{ e^{- i( \phi_1 + \phi_2)t } }{\phi_1(\phi_1+\phi_2)}
\, 
\dt \bigg(\prod_{a \in \TT^\infty_2 } \u_{n_{a}}\bigg) \notag \\
& = \dt \bigg[\sum_{\TT_2 \in \mathfrak{BT}(2)}
\sum_{\substack{\n \in \mathfrak{N}(\TT_2)\\n_r = n}}
\frac{ e^{- i( \phi_1 + \phi_2 )t } }{\phi_1(\phi_1+\phi_2)}
\prod_{a \in \TT^\infty_2} \u_{n_{a}} \bigg]\notag \\
& \hphantom{X} -  
\sum_{\TT_2 \in \mathfrak{BT}(2)}
\sum_{b \in\TT^\infty_2} 
\sum_{\substack{\n \in \mathfrak{N}(\TT_2)\\n_r = n}}
\frac{ e^{- i( \phi_1 + \phi_2)t } }{\phi_1(\phi_1+\phi_2)}
\, \textsf{R}(\u)_{n_b}
\prod_{a \in \TT^\infty_2 \setminus \{b\}} \u_{n_{a}} \notag \\
& \hphantom{X} 
- \sum_{\TT_3 \in \mathfrak{BT}(3)}
\sum_{\substack{\n \in \mathfrak{N}(\TT_3)\\n_r = n}}
\frac{ e^{- i( \phi_1 + \phi_2 +\phi_3)t } }{\phi_1(\phi_1+\phi_2)}
\,\prod_{a \in \TT^\infty_3} \u_{n_{a}} \notag\\
& =: \dt \NN^{(3)}_0(\u)_n + \RR^{(3)}(\u)_n + \NN^{(3)}(\u)_n.
\end{align}

\noi
The first two terms  $\NN^{(3)}_0$ and $ \RR^{(3)}$ on the right-hand side 
can be estimated in a straightforward manner
with \eqref{mu1} and \eqref{mu2}.
See Lemmas \ref{LEM:N^J+1_0} and \ref{LEM:N^J+1_r} below.
As for the last term $\NN^{(3)}$,  
we split it as 
\begin{equation*}
\NN^{(3)} = \NN^{(3)}_1 + \NN^{(3)}_2,
\end{equation*}

\noi
where $\NN^{(3)}_1$ and  $\NN^{(3)}_2$
are  the restrictions
onto 
\begin{equation} \label{C2}
C_2 = \big\{ |\phi_1 + \phi_2 + \phi_3| \lesssim 8^3|\phi_1 + \phi_2|^{1-\theta}\big\} 
\cup \big\{ |\phi_1 + \phi_2 + \phi_3| \lesssim 8^3 |\phi_1|^{1-\theta}\big\}
\end{equation}

\noi
and its complement $C_2^c$, respectively.
By exploiting the frequency restriction,
we can  estimate the first term  $\NN^{(3)}_1$ (see Lemma \ref{LEM:N^J+1_1} below).
As for the second term $\NN^{(3)}_2$, 
we apply the third step of the normal form reductions.
In this way, we iterate  normal form reductions
in an indefinite manner.

We conclude this subsection by presenting the 
proof of Lemma \ref{LEM:N^2_1}.

\begin{proof}[Proof of Lemma \ref{LEM:N^2_1}]
Note that 
we have $|\phi_2|\sim |\phi_1|$ thanks to \eqref{C1}.
Then, with \eqref{fphi}
and $|\mu_j| \les (n^{(j)}_{\max})^2$, $j = 1, 2$,  we have 
\begin{align} 
\sup_{n \in \Z}
  \sum_{\substack{{\bf n} \in \mathfrak{N}(\TT_2)\\ n_r = n\\|\phi_1| \sim |\phi_2| > K}} 
\frac{(n^{(1)}_{\max})^{-6s} (n^{(2)}_{\max})^{-6s}}{|\phi_1|^2} 
&  \lesssim 
\sup_{n \in \Z}  \sum_{\substack{{\bf n} \in \mathfrak{N}(\TT_2)\\  n_r = n\\|\phi_1|>K}} 
\frac{1} {|\mu_1 \mu_2| (n^{(1)}_{\max} n^{(2)}_{\max} )^{2+6s }} \notag  \\
& \les  \sup_{n \in \Z} \sum_{\substack{{\bf n} \in \mathfrak{N}(\TT_2)\\  n_r = n}}
\frac{1} {|\mu_1 \mu_2|^{1+ }}
 \les 1,
\label{TT3}
\end{align}

\noi
provided that  $s > - \frac{1}3$.
In the last step, we first summed  over $n_1^{(2)}$ and $n_3^{(2)}$ 
for  fixed $n^{(2)}$ and then summed over 
 $n_1^{(1)}$ and $n_3^{(1)}$ 
 for fixed $n$.

\smallskip

\noi
$\bullet$ {\bf Case 1:}
We first consider the case  $\Pi_2(\TT_2) = \{r_2\}$.
Namely, the second root node $r_2$ is a terminal node.
By Cauchy-Schwarz inequality with \eqref{TT3}, we have
\begin{align*}
\sum_{n\in \Z} | \NN^{(2)}_1(\u)_n | 
& \lesssim 
\sum_{n\in \Z}
\sum_{\substack{\TT_2 \in \mathfrak{BT}(2)\\ \Pi_2(\TT_2) = \{r_2\}}}
 \sum_{\substack{{\bf n} \in \mathfrak{N}(\TT_2)\\n_r = n\\|\phi_1|>K}} 
\frac{1}{|\phi_1|} \prod_{a \in \TT^\infty_2} \u_{n_a} 
\notag  \\
& \les
\| \u\|_{H^s} \Bigg\{ \sum_{n\in \Z} \Bigg(
\sum_{\substack{\TT_2 \in \mathfrak{BT}(2)\\ \Pi_2(\TT_2) = \{r_2\}}}
 \sum_{\substack{{\bf n} \in \mathfrak{N}(\TT_2)\\ n_r = n\\|\phi_1|>K}} 
\frac{\jb{n}^{-s}}{|\phi_1|} \prod_{a \in \TT^\infty_2 \setminus \{r_2\} } \u_{n_a} \Bigg)^2 \Bigg\}^{\frac12} 
\notag  \\
& \lesssim \|\u\|_{H^s} 
\sup_{\substack{\TT_2 \in \mathfrak{BT}(2)\\ \Pi_2(\TT_2) = \{r_2\}}}
\bigg( \sup_{n\in \Z} \sum_{\substack{{\bf n} \in \mathfrak{N}(\TT_2)\\ n_r = n\\|\phi_1|>K}} 
\frac{(n^{(1)}_{\max})^{-6s} (n^{(2)}_{\max})^{-6s}}{|\phi_1|^2} \bigg)^\frac{1}{2} \notag \\
& \qquad  \qquad \times 
 \bigg(  
 \sum_{n\in \Z}
 \sum_{\substack{{\bf n} \in \mathfrak{N}(\TT_2)\\n_r = n}} 
\prod_{a \in \TT^\infty_2\setminus \{r_2\}} \jb{n_a}^{2s}|\u_{n_a}|^2
\bigg)^\frac{1}{2}\notag  \\
& \les \|\u\|_{H^s}^6.
\end{align*}

\noi
In the last step, 
we used the observations in Remark \ref{REM:terminal}.

\smallskip

\noi
$\bullet$ {\bf Case 2:}
Next, we  consider the case  $\Pi_2(\TT_2) \ne \{r_2\}$.
In this case, we need to modify the argument above since
the frequency $n_r = n$ does not correspond to a terminal node.
Noting that $\TT_2^\infty = \Pi_1(\TT_2)^\infty \cup \Pi_2(\TT_2)^\infty$
and hence 
\begin{align*}
\sum_{\substack{{\bf n} \in \mathfrak{N}(\TT_2)\\{  n}_r = n}} 
\prod_{a \in \TT^\infty_2} |\u_{n_a}|^2
= \prod_{j = 1}^2 \bigg(\sum_{\substack{{\bf n} \in \mathfrak{N}(\Pi_j(\TT_2))\\{  n}_{r_j} = n}} 
\prod_{a_j \in \Pi_j(\TT_2)^\infty} |\u_{n_{a_j}}|^2
\bigg), 
\end{align*}

\noi
we have
\begin{align*}
\sum_{n\in \Z} |  & \NN^{(2)}_1(\u)_n|
 \lesssim 
\sum_{n\in \Z}
\sum_{\substack{\TT_2 \in \mathfrak{BT}(2)\\ \Pi_2(\TT_2) \ne \{r_2\}}}
\sum_{\substack{{\bf n} \in \mathfrak{N}(\TT_2)\\ n_r = n\\|\phi_1|>K}} 
\frac{1}{|\phi_1|} \prod_{a \in \TT^\infty_2} \u_{n_a} 
\notag  \\
& \lesssim 
\sup_{\substack{\TT_2 \in \mathfrak{BT}(2)\\ \Pi_2(\TT_2) \ne \{r_2\}}}
\sum_{n \in \Z}
\bigg(  \sum_{\substack{{\bf n} \in \mathfrak{N}(\TT_2)\\ n_r = n\\|\phi_1|>K}} 
\frac{(n^{(1)}_{\max})^{-6s} (n^{(2)}_{\max})^{-6s}}{|\phi_1|^2}\, \bigg)^\frac{1}{2}
 \bigg( \sum_{\substack{{\bf n} \in \mathfrak{N}(\TT_2)\\ n_r = n}}
  \prod_{a \in \TT^\infty_2} \jb{n_a}^{2s}|\u_{n_a}|^2
\bigg)^\frac{1}{2}\notag  \\
& \lesssim 
\sup_{\substack{\TT_2 \in \mathfrak{BT}(2)\\ \Pi_2(\TT_2) \ne \{r_2\}}}
\sum_{n \in \Z}
 \bigg( \sum_{\substack{{\bf n} \in \mathfrak{N}(\TT_2)\\n_r = n}}
  \prod_{a \in \TT^\infty_2} \jb{n_a}^{2s}|\u_{n_a}|^2
\bigg)^\frac{1}{2}\notag  \\
& \les 
\sup_{\substack{\TT_2 \in \mathfrak{BT}(2)\\ \Pi_2(\TT_2) \ne \{r_2\}}}
 \sum_{n\in \Z} \, \prod_{j = 1}^2  \bigg( 
 \sum_{\substack{{\bf n} \in \mathfrak{N}(\Pi_j(\TT_2))\\ n_{r_j} = n}} 
\prod_{a_j \in \Pi_j(\TT_2)^\infty} \jb{n_{a_j}}^{2s}|\u_{n_{a_j}}|^2 \bigg) ^\frac{1}{2}\notag \\
& \les 
\sup_{\substack{\TT_2 \in \mathfrak{BT}(2)\\ \Pi_2(\TT_2) \ne \{r_2\}}}
 \prod_{j = 1}^2   \bigg( 
\sum_{n\in \Z}  \sum_{\substack{{\bf n} \in \mathfrak{N}(\Pi_j(\TT_2))\\n_{r_j} = n}} 
\prod_{a_j \in \Pi_j(\TT_2)^\infty} \jb{n_{a_j}}^{2s}|\u_{n_{a_j}}|^2 \bigg) ^\frac{1}{2}\notag \\
& \les 
  \|\u\|_{H^s}^6.
\end{align*}

\noi
This completes the proof of Lemma \ref{LEM:N^2_1}.
\end{proof}

\begin{remark}\label{REM:Yuzhao}
\rm
The above computation in Cases 1 and 2 in particular shows that,
given $\TT_j  \in \mathfrak{BT}(j)$, $j \in \NB$, 
we have
\begin{align*}
 \sum_{n\in \Z} \Big( \sum_{\substack{{\bf n} \in \mathfrak{N}(\TT_j)\\n_r = n } }
\prod_{a \in \TT^\infty_j} |\u_{n_{a}}|^2 \Big)^{\frac12} \le \|\u_n\|_{\l^2_n}^{2j+2}.
\end{align*}
\end{remark}

\subsection{General step: $J$\,th generation}
\label{SUBSEC:8.3}

After the $J$\,th step, we have
\begin{align} \label{N^J+1}
\NN^{(J)}_2 (\u)_n
& = \dt \bigg[ 
\sum_{\TT_J \in \mathfrak{BT}(J)}
\sum_{\substack{{\bf n} \in \mathfrak{N}(\TT_J)\\n_r = n}}
\frac{ e^{- i \wt{\phi}_Jt } }{\prod_{j = 1}^J \wt{\phi}_j}
\, \prod_{a \in \TT^\infty_J} \u_{n_{a}}
\bigg]\notag \\
& \hphantom{X} 
- \sum_{\TT_{J} \in \mathfrak{BT}(J)}
\sum_{b \in\TT^\infty_J} 
\sum_{\substack{{\bf n} \in \mathfrak{N}(\TT_J)\\n_r = n}}
\frac{ e^{- i \wt{\phi}_Jt } }{\prod_{j = 1}^J \wt{\phi}_j}
\, \textsf{R}(\u)_{n_b}
\prod_{a \in \TT^\infty_J \setminus \{b\}} \u_{n_{a}}  \notag \\
& \hphantom{X} 
- \sum_{\TT_{J+1} \in \mathfrak{BT}(J+1)}
\sum_{\substack{{\bf n} \in \mathfrak{N}(\TT_{J+1})\\n_r = n}}
\frac{ e^{- i \wt{\phi}_{J+1}t } }{\prod_{j = 1}^J \wt{\phi}_j}
\,\prod_{a \in \TT^\infty_{J+1}} \u_{n_{a}} \notag\\
& =: \dt \NN^{(J+1)}_0 (\u)_n+ \RR^{(J+1)}(\u)_n + \NN^{(J+1)}(\u)_n,
\end{align}

\noi
where $\wt{\phi}_J$ is defined by 
\begin{align}
 \wt{\phi}_J := \sum_{j = 1}^J \phi_j.
\label{XX2}
\end{align}

\noi
Recall that $|\phi_1|>K$ and 
\begin{equation} \label{muj}
|\wt{\phi}_j|  \gg 
(2j+2)^{3}\max ( |\wt{\phi}_{j-1}|^{1-\theta},
|\phi_1|^{1-\theta}) >
(2j+2)^{3}K^{1-\theta}, 
\end{equation}

\noi
for $j = 2, \dots, J.$
One of the main tasks
in estimating the multilinear forms in \eqref{N^J+1}
is to control  the rapidly growing
 cardinality
$c_J = |\mathfrak{BT}(J)| $  defined in \eqref{cj1}.
As in \cite{GKO}, 
we control $c_J$ by the growing
constant $(2j+2)^3$ appearing in \eqref{muj}.

First, we estimate $\NN^{(J+1)}_0$ and  $\RR^{(J+1)}$.

\begin{lemma}\label{LEM:N^J+1_0}
Let $\NN^{(J+1)}_0$ be as in \eqref{N^J+1}.
Then, for  $s>-\frac 13$, 
we have\footnote{The implicit constant is  independent of $J$.
The same comment applies to Lemmas \ref{LEM:N^J+1_r} and \ref{LEM:N^J+1_1} below.}
\begin{align*} 
\sum_{n\in \Z}| \NN^{(J+1)}_0(\u)_n|  \lesssim 
K^{-\frac{J(1-\theta)}2 } 
 \|\u\|_{H^s}^{2J+2}. 
\end{align*}
\end{lemma}

\begin{proof}

From \eqref{XX2}, we have 
\begin{align*} 
|\phi_j| \lesssim \max(|\wt{\phi}_{j-1}|, |\wt{\phi}_j|).
\end{align*}

\noi
Then, in view of \eqref{muj}, we have 
\begin{align} 
\label{mujj2}
(2j)^3 K^{1-\theta} |\phi_j| \ll |\wt{\phi}_{j-1}| |\wt{\phi}_j|.
\end{align}

\noi
Hence,  with \eqref{muj} once again,  we have
\begin{align} 
\label{mujj3}
 \prod_{j=1}^{J}\Big( (2j+2)^3 K^{1-\theta} |\phi_j|\Big) 
& \ll 
|\phi_1| |\wt \phi_J|
 \prod_{j=2}^{J}\Big( (2j)^3 K^{1-\theta} |\phi_j|\Big) 
  \ll 
 \prod_{j=1}^{J} |\wt\phi_j|^2.
\end{align}

We only discuss the case  $\Pi_2(\TT_J) = \{r_2\}$
since the modification is straightforward if $\Pi_2(\TT_J) \ne \{r_2\}$.
As in \eqref{TT3}, we have 
\begin{align}
\label{mujj31}
\frac{(n^{(j)}_{\max})^{-6s}}{|\phi_j|} 
\sim \frac{(n^{(j)}_{\max})^{-6s}}{|\mu_j| (n^{(j)}_{\max})^{2}} \les \frac{1}{|\mu_j|^{1+} }  
\end{align}

\noi
for $s > -\frac13$.
Then, 
by \eqref{mujj3} and \eqref{mujj31}, we have
\begin{align}
\label{mujj4}
 \sup_{n\in \Z}
\sum_{\substack{{\bf n} \in \mathfrak{N}(\TT_J)\\ n_r = n\\|\phi_1| >K 
\\ |\wt{\phi}_j|\gg(2j+2)^3K^{1-\theta}\\ j = 2, \dots, J} }
\prod_{j = 1}^J \frac{(n^{(j)}_{\max})^{-6s}}{|\wt{\phi}_j|^2} 
& \ll \frac{K^{-J(1-\theta) } }{\prod_{j = 1}^J(2j+2)^{3}}
\cdot  \sup_{n\in \Z}
\sum_{\substack{{\bf n} \in \mathfrak{N}(\TT_J)\\ n_r = n\\ \phi_j \neq 0\\ j = 1, \dots, J} } 
\prod_{j = 1}^J \frac{(n^{(j)}_{\max})^{-6s}}{|\phi_j|} \notag \\
& \les \frac{K^{- J(1-\theta) } }{\prod_{j = 1}^J(2j+2)^{3}}
\cdot  \sup_{n\in \Z}
\sum_{\substack{{\bf n} \in \mathfrak{N}(\TT_J)\\ n_r = n\\ \mu_j \neq 0\\ j = 1, \dots, J} } 
\prod_{j = 1}^J \frac{1}{|\mu_j|^{1+}} \notag \\
& \le \frac{C^J K^{- J(1-\theta) } }{\prod_{j = 1}^J(2j+2)^{3}}.
\end{align}

\noi
By Cauchy-Schwarz inequality with 
\eqref{mu1}, \eqref{muj}, and \eqref{mujj4}, 
 we have
\begin{align}
\sum_{n\in \Z} | \NN^{(J+1)}_0(\u)_n|
& \lesssim 
\|\u\|_{H^s}
\sum_{\substack{\TT_J \in \mathfrak{BT}(J)\\\Pi_2(\TT_J) = \{r_2\}}}
\Bigg\{ 
\bigg( 
\sup_{n\in \Z}
\sum_{\substack{{\bf n} \in \mathfrak{N}(\TT_J)\\n_r = n\\|\phi_1| >K 
\\ |\wt{\phi}_j|\gg(2j+2)^3K^{1-\theta}\\ j = 2, \dots, J} }
\prod_{j = 1}^J \frac{(n^{(j)}_{\max})^{-6s}}{|\wt{\phi}_j|^2}
\bigg) \notag \\
& \hphantom{XXXXXXXXXX} \times 
\bigg( 
\sum_{n\in\Z}
\sum_{\substack{{\bf n} \in \mathfrak{N}(\TT_J)\\n_r = n}} 
\prod_{a \in \TT^\infty_J \setminus \{r_2\}} \jb{n_{a}}^{2s}|\u_{n_a}|^2
\bigg)\Bigg\}^\frac{1}{2}  \notag \\
& \lesssim \frac{c_J\cdot C^{\frac{J}{2}}}{\prod_{j = 1}^J(2j+2)^{\frac32}}
K^{-\frac{J(1-\theta)}2 } 
\|\u\|_{H^s}^{2J+2} \lesssim 
K^{-\frac{J(1-\theta)}2 } 
\|\u\|_{H^s}^{2J+2}. 
\label{N^J+1_0-3}
\end{align}

\noi
This completes the proof of Lemma \ref{LEM:N^J+1_0}.
\end{proof}

\begin{lemma}\label{LEM:N^J+1_r}
Let $\RR^{(J+1)}$ be as in \eqref{N^J+1}.
Then, for  $s\geq \max\big(- \frac{1}{3}+,  -\frac{3-2\theta}{5}\big)$, 
we have
\begin{align} 
\sum_{n\in \Z}| \RR^{(J+1)}(\u)_n|
 \lesssim 
K^{-\frac{{(J-2)}(1-\theta)}2 }  \|\u\|_{H^s}^{2J+4}. 
\label{N^J+1_r}
\end{align}

\noi
In particular, if $\theta \in (0, \frac{2}{3}]$, 
then \eqref{N^J+1_r} holds for $s > -\frac 13$.

\end{lemma}

\begin{proof}
Just like Lemma \ref{LEM:N^2_r} on $ \RR^{(2)}$, 
this lemma  follows from a modification of the proof of Lemma  \ref{LEM:N^J+1_0}.

We first consider the case $|\wt \phi_J| \ges |\phi_J|$.
Noting that  $ (n^{(J)}_{\max})^{-4s} \les |\wt \phi_{J}|$ for $s\geq -\frac12$, 
it follows from the second inequality in  \eqref{mujj3} and \eqref{mujj31}, we have
\begin{align}
\label{mujj41}
 \sup_{n\in \Z}\sum_{\substack{{\bf n} \in \mathfrak{N}(\TT_J)\\n_r = n\\|\phi_1| >K 
\\ |\wt{\phi}_j|\gg(2j+2)^3K^{1-\theta}\\ j = 2, \dots, J} }
& (n_{\max}^{(J)})^{-4s} 
 \prod_{j = 1}^J \frac{(n^{(j)}_{\max})^{-6s}}{|\wt{\phi}_j|^2} \notag\\
& \ll \frac{K^{-(J-1)(1-\theta) } }{\prod_{j = 2}^J(2j+2)^{3}}
\cdot  \sup_{n\in \Z}
\sum_{\substack{{\bf n} \in \mathfrak{N}(\TT_J)\\n_r = n\\ \phi_j \neq 0
\\ j = 1, \dots, J} }  \frac{(n^{(J)}_{\max})^{-4s}}{|\wt \phi_{J}|}
\prod_{j = 1}^{J} \frac{(n^{(j)}_{\max})^{-6s}}{|\phi_j|} \notag \\
& \les \frac{C^J K^{- (J-1)(1-\theta) } }{\prod_{j = 2}^J(2j+2)^{3}}, 
\end{align}

\noi
provided that  $s > -\frac 13$.
Then,  proceeding as in 
\eqref{N^J+1_0-3}
with \eqref{mujj41}, 
we obtain 
\eqref{N^J+1_r} in this case.

Next, consider the case 
 $|\wt \phi_J| \ll |\phi_J|$. 
In this case, we have  
 $|\phi_J| \sim |\wt \phi_{J-1}|$. 
 Proceeding as in  \eqref{mujj3} with \eqref{mujj2}, we have
\begin{align*} 
|\phi_1| |\wt \phi_{J-1}| |\wt \phi_J|^2 \prod_{j=2}^{J-1} \Big((2j)^3 K^{1-\theta} |\phi_j| \Big)
  \ll 
 \prod_{j=1}^{J} |\wt\phi_j|^2.
\end{align*}

\noi
From \eqref{muj}, we have 
$|\wt \phi_J| \gg (2J+2)^3 |\wt \phi_{J-1}|^{1-\theta} \sim (2J+2)^3 |\phi_J|^{1-\theta}$.
This gives 
\begin{align*} 
K^{(J-2)(1-\theta)} | \phi_{J}|^{2-2\theta} \prod_{j=1}^{J} \Big((2j+2)^3  |\phi_j| \Big)
  \ll 
 \prod_{j=1}^{J} |\wt\phi_j|^2.
\end{align*}

\noi
Hence,  we obtain
\begin{align}
\label{mujj44}
 \sup_{n\in \Z}
\sum_{\substack{{\bf n} \in \mathfrak{N}(\TT_J)\\n_r = n\\|\phi_1| >K 
\\ |\wt{\phi}_j|\gg(2j+2)^3K^{1-\theta}\\ j = 2, \dots, J} }
& (n_{\max}^{(J)})^{-4s}
\prod_{j = 1}^J \frac{(n^{(j)}_{\max})^{-6s}}{|\wt{\phi}_j|^2} \notag\\
& \ll \frac{K^{-(J-2)(1-\theta) } }{\prod_{j = 1}^{J} (2j+2)^{3}}
\cdot  \sup_{n\in \Z}
\sum_{\substack{{\bf n} \in \mathfrak{N}(\TT_J)\\ n_r = n\\\phi_j \neq 0 \\
j = 1, \dots, J} }  \frac{(n^{(J)}_{\max})^{-10s}}{| \phi_{J}|^{3-2\theta}}
\prod_{j = 1}^{J-1} \frac{(n^{(j)}_{\max})^{-6s}}{|\phi_j|} \notag \\
& \les \frac{C^J K^{- (J-2)(1-\theta) } }{\prod_{j = 1}^{J} (2j+2)^{3}},
\end{align}

\noi
provided that  $s > -\frac 13$ and $s \geq -\frac{3-2\theta}{5}$.
Then,  proceeding as in 
\eqref{N^J+1_0-3}
with \eqref{mujj44}, 
we obtain 
\eqref{N^J+1_r} in this case.
\end{proof}

Finally, we consider  $\NN^{(J+1)}$.
As before, we write
\begin{equation} \label{N^J+1_1}
\NN^{(J+1)} = \NN^{(J+1)}_1 + \NN^{(J+1)}_2,
\end{equation}

\noi
where $\NN^{(J+1)}_1$ is the restriction of $\NN^{(J+1)}$
onto 
\begin{equation} \label{CJ}
C_J = \big\{ |\wt{\phi}_{J+1}| \lesssim (2J+2)^{3} |\wt{\phi}_J|^{1-\theta}\big\} 
\cup \big\{ |\wt{\phi}_{J+1}| \lesssim (2J+2)^{3} |\phi_1|^{1-\theta}\big\} 
\end{equation}

\noi
and
$\NN^{(J+1)}_2 := \NN^{(J+1)} - \NN^{(J+1)}_1$.
In the following lemma, we estimate the first term $\NN^{(J+1)}_1$.
Then,  we apply a normal form reduction 
once again to the second term $\NN^{(J+1)}_2$ 
as in \eqref{N^J+1}
and repeat this process indefinitely.
In the next subsection, we 
also show that this error term 
$\NN^{(J+1)}_2$ tends to 0 in the $\l^1_n$-nom as $J \to \infty$.

\begin{lemma}\label{LEM:N^J+1_1}
Let $\NN^{(J+1)}_1$ be as in \eqref{N^J+1_1}.
Then, for $ s> -\frac{1}{3}$, 
we have
\begin{align} 
\sum_{n\in \Z} | \NN^{(J+1)}_1(\u)_n| 
\lesssim 
K^{-\frac{(J-1)(1-\theta)}2 } \|\u\|_{H^s}^{2J+4}.
\label{N^J+1_1a}
\end{align}
\end{lemma}

\begin{proof}
We proceed with \eqref{mujj2} as in the proofs of Lemmas 
\ref{LEM:N^J+1_0} and \ref{LEM:N^J+1_r}.
It follows from the restriction $|\wt{\phi}_{J+1}| = |\wt{\phi}_J+\phi_{J+1}| \lesssim (2J+2)^3|\wt{\phi}_J|^{1-\theta}$
that 
 $|\phi_{J+1}| \les J^3|\wt{\phi}_J|$.
Then from \eqref{mujj2}, 
we have
\begin{align*} 
|\phi_1|  |\phi_{J+1}| \prod_{j=2}^{J} (2j+2)^3 K^{1-\theta} |\phi_j| 
\ll J^3 
 \prod_{j=1}^{J} |\wt\phi_j|^2.
\end{align*}

\noi
Proceeding 
as in \eqref{mujj4}, we have
\begin{align}
\label{mujj6}
  \sup_{n\in \Z}
\sum_{\substack{{\bf n} \in \mathfrak{N}(\TT_{J+1})\\ n_r = n\\|\phi_1| >K 
\\ |\wt{\phi}_j|\gg(2j+2)^3 K^{1-\theta}\\ j = 2, \dots, J} }
& 
(n^{(J+1)}_{\max})^{-6s} 
\prod_{j = 1}^J \frac{(n^{(j)}_{\max})^{-6s}}{|\wt{\phi}_j|^2} \notag\\
& \ll \frac{ K^{-(J-1)(1-\theta) } }{\prod_{j = 2}^{J-1} (2j+2)^{3}}
\cdot  \sup_{n\in \Z}
\sum_{\substack{{\bf n} \in \mathfrak{N}(\TT_{J+1})\\ n_r = n\\|\phi_j| \neq 0 
\\ j = 1, \dots, J+1} } 
\prod_{j = 1}^{J+1} \frac{(n^{(j)}_{\max})^{-6s}}{|\phi_j|} \notag \\
& \les  \frac{C^{J+1} K^{-(J-1)(1-\theta) }}{\prod_{j = 2}^{J-1} (2j+2)^{3}},
\end{align}

\noi
provided that $s > -\frac 13$.
Then,  \eqref{N^J+1_1a}
follows from the Cauchy-Schwarz argument with \eqref{mujj6} 
once we note that 
\[
 \frac{c_{J+1}C^{\frac{J+1}{2}}}{\prod_{j = 2}^{J-1}(2j+2)^{\frac 32}}
 \les 1,
\]
uniformly in $J$.
\end{proof}

\subsection{On the error term}
\label{SUBSEC:error}

In this subsection, we prove that the error term $\NN_2^{(J+1)}(\u)$
tends to 0 
as $J \to \infty$ under some regularity assumption on $\u$.

From \eqref{N^J+1}, we have
\begin{align} \label{X1}
 \NN_2^{(J+1)}(\u)_n
&  =- \sum_{\TT_{J+1} \in \mathfrak{BT}(J+1)}
\sum_{\substack{{\bf n} \in \mathfrak{N}(\TT_{J+1})\\n_r = n}}
\frac{ e^{- i \wt{\phi}_{J+1}t } }{\prod_{j = 1}^J \wt{\phi}_j}
\,\prod_{a \in \TT^\infty_{J+1}} \u_{n_{a}} \notag\\
& = - \sum_{\TT_{J} \in \mathfrak{BT}(J)}
\sum_{b \in\TT^\infty_J} 
\sum_{\substack{{\bf n} \in \mathfrak{N}(\TT_J)\\n_r = n}}
\frac{ e^{- i \wt{\phi}_Jt } }{\prod_{j = 1}^J \wt{\phi}_j}
\, \textsf{N}(\u)_{n_b}
\prod_{a \in \TT^\infty_J \setminus \{b\}} \u_{n_{a}}, 
\end{align}

\noi
where it is understood that the summations in \eqref{X1}
are restricted to frequencies satisfying \eqref{mu1} and \eqref{muj}.\footnote{In fact, 
$ \NN_2^{(J+1)}$ is also restricted to $C_J^c$
but we do not need to use this fact. 
Namely, our argument also shows that 
$ \NN^{(J+1)}(\u) \to 0$ as $J \to \infty$.}

\begin{lemma}\label{LEM:N^J+1_2}
Let $\NN^{(J+1)}_2$ be as in \eqref{X1}.
Then, given $\u \in H^\frac{1}{6}(\T)$, 
we have
\begin{align} 
\sum_{n \in \Z} | \NN^{(J+1)}_2(\u)_n|  \longrightarrow 0, 
\label{X1a}
\end{align}

\noi
as $J \to \infty$.
\end{lemma}

\begin{proof}

The following simple estimate yields the  minimum regularity restriction
$s \geq \frac 16$, 
required for $\NN_2^{(J+1)}\to 0$ as $J \to \infty$.
By Hausdorff-Young's, H\"older's, and Sobolev's inequalities, 
we have 
\begin{align}
\| \textsf{N}(\u)_{n_b}\|_{\l^\infty_{n_b}}
= \bigg\|
\sum_{\substack{{\bf n} \in \mathfrak{N}(\TT_{1})\\n_r = n_b}}
\prod_{a \in \TT^\infty_1 } \u_{n_{a}}  
\bigg\|_{\l^\infty_{n_b}}
& \leq \big\|\F^{-1}(|\ft \u_n|)\big\|_{L^3_x}^3 \les \|\u\|_{H^{\frac16}}^3. 
\label{X2}
\end{align}

\noi
$\bullet$ {\bf Case 1:}
We first consider  the case  $\Pi_2(\TT_J) = \{r_2\}$.

Suppose that $b \ne r_2$ in \eqref{X1}.
In this case, by summing over 
 the  $2J$ variables $\{ n_a\}_{a \in \TT^\infty_J\setminus\{b, r_2\}}$
 first and then over $n \in \Z$, we have
\begin{align}
 \sum_{n \in \Z} 
\sum_{\substack{{\bf n} \in \mathfrak{N}(\TT_J)\\ n_r = n}} 
 \frac{1}{(n_{\max}^{(1)})^{1+}} 
 \prod_{a \in \TT^\infty_J\setminus\{b, r_2\}} |\u_{n_a}|^2
& \les  
 \sum_{n\in \Z} 
\frac{1}{\jb{n}^{1+}}
\sum_{\substack{{\bf n} \in \mathfrak{N}(\TT_J)\\ n_r = n}} 
\prod_{a \in \TT^\infty_J\setminus\{b, r_2\}} |\u_{n_a}|^2 \notag\\
& \les \|\u\|_{L^2}^{4J}.
\label{X3}
\end{align}

\noi
Then, by Cauchy-Schwarz inequality with \eqref{mujj3}, \eqref{X2}, \eqref{X3},
and $|\TT_J \setminus\{r_2\}| = 2J + 1$, 
we have
\begin{align}
\sum_{n \in \Z}|  & \NN^{(J+1)}_2(\u)_n|
  \lesssim  
\sum_{\substack{\TT_J \in \mathfrak{BT}(J)\\\Pi_2(\TT_J) = \{r_2\}}}
\sum_{b \in\TT^\infty_J} 
\sum_{n \in \Z}
|\u_n|
\sum_{\substack{{\bf n} \in \mathfrak{N}(\TT_J)\\ n_r = n\\|\phi_1| >K 
\\ |\wt{\phi}_j|\gg(2j+2)^3K^{1-\theta}\\ j = 2, \dots, J} }
\|\textsf{N}(\u)_{n_b}\|_{\l^\infty_{n_b}} 
\notag\\
& \hphantom{XXXXXXXX} \times 
 \frac{1}{\prod_{j = 1}^J |\wt{\phi}_j|}
\prod_{a \in \TT^\infty_J\setminus\{b, r_2\}} |\u_{n_a}|
\notag \\
& \lesssim  
J \|\u_n\|_{\l^2_{n}}\|\u\|_{H^\frac{1}{6}}^3
\notag\\
& \hphantom{XX} \times 
\sum_{\substack{\TT_J \in \mathfrak{BT}(J)\\\Pi_2(\TT_J) = \{r_2\}}}
\Bigg\{\sum_{n \in \Z} \bigg( 
\sum_{\substack{{\bf n} \in \mathfrak{N}(\TT_J)\\ n_r = n\\|\phi_1| >K 
\\ |\wt{\phi}_j|\gg(2j+2)^3K^{1-\theta}\\ j = 2, \dots, J} }
 \frac{1}{\prod_{j = 1}^J |\wt{\phi}_j|}
\prod_{a \in \TT^\infty_J\setminus\{b, r_2\}} |\u_{n_a}|
\bigg)^2\Bigg\}^\frac{1}{2}  \notag \\
& \ll
\frac{ J  c_J}{\prod_{j = 1}^J(2j+2)^{\frac{3}{2}}}
K^{-\frac{J(1-\theta)}{2}} 
\|\u\|_{L^2} \|\u\|_{H^\frac{1}{6}}^3
 \notag\\
& \hphantom{XX} \times 
\sup_{\substack{\TT_J \in \mathfrak{BT}(J)\\\Pi_2(\TT_J) = \{r_2\}}}
\Bigg\{\bigg( \sup_{n\in \Z} 
\sum_{\substack{{\bf n} \in \mathfrak{N}(\TT_J)\\ n_r = n\\
\phi_j \ne 0
\\ j = 1, \dots, J} }
(n_{\max}^{(1)})^{1+}
\prod_{j = 1}^J \frac{1}{|{\phi}_j|}\bigg)
\notag \\
& \hphantom{XX} \times 
\bigg( \sum_{n\in \Z}  \sum_{\substack{{\bf n} \in \mathfrak{N}(\TT_J)\\ n_r = n}} 
\frac{1}{(n_{\max}^{(1)})^{1+}} 
\prod_{a \in \TT^\infty_J\setminus\{b, r_2\}} |\u_{n_a}|^2
\bigg)\Bigg\}^\frac{1}{2}  \notag \\
&\lesssim 
\frac{ J   C^J  c_J}{\prod_{j = 1}^J(2j+2)^{\frac{3}{2}}}
K^{-\frac{J(1-\theta)}{2}}
 \|\u\|_{L^2}^{2J+1} \|\u\|_{H^\frac{1}{6}}^3
 \longrightarrow 0 
\label{X3a}
\end{align}

\noi
for any $K > 0$, 
as $J \to \infty$.
See Subsection \ref{SUBSEC:8.5}
for a more precise condition on $K$
required for the convergence of the series for $\NN^{(j)}_0$, $\NN^{(j)}_1$, and 
and $\RR^{(j)}$.

Next, suppose that $b = r_2$ in \eqref{X1}.
This time,  we need to work our way from the bottom.
Let us first state and prove a useful lemma
which follows from the non-resonance condition in Definition \ref{DEF:tree4} (c)
and the divisor counting argument.

\begin{lemma}\label{LEM:XT1}
Let $J \in \NB$.
Given an ordered bi-tree $\TT_J\in \mathfrak{BT}(J)$
with a chronicle $\{\TT_j\}_{j = 1}^J$ such that 
 $\Pi_2(\TT_J) = \{r_2\}$, 
fix $a\in \TT^\infty_J \setminus \TT^\infty_{J-1}$.\footnote{By convention, 
we set $\TT_0$  to be a bi-tree of the zeroth generation consisting of the two root nodes $r_1$ and $r_2$
joined by an edge.
Hence, we have $\TT_0^{\infty} = \{r_1, r_2\} $.}  
Then, for fixed $m \in \Z$
and $\nu_j \in \Z$, $j = 1, \dots, J$, we have
\begin{align}
\#\big\{ \bn \in \mathfrak{N}(\TT_J): n_a = m, \ \mu_j (\bn) = \nu_j, j = 1, \dots, J \big\}
\leq C^J \prod_{j = 1}^J|\nu_j|^{0+}.
\label{XT1}
\end{align}
		
\end{lemma}

In view of   $\Pi_2(\TT_J) = \{r_2\}$, 
 we can identify the ordered bi-tree $\TT_J$ 
with an ordinary (ternary) ordered tree (in the sense of \cite[Definition 3.3]{GKO}).
Lemma \ref{LEM:XT1} is really a property of an ordered tree.
Before proceeding to the proof of Lemma \ref{LEM:XT1}, let us 
recall the following arithmetic fact \cite{HW}.
Given $n \in \NB$, the number $d(n)$ of the divisors of $n$
satisfies
\begin{align}
d(n) \leq C_\dl n^\dl
\label{divisor}
\end{align}

\noi
for any $\dl > 0$.	

\begin{proof}[Proof of Lemma \ref{LEM:XT1}]
We first consider the case  $J = 1$.  
Let ${r_{1j}}, j = 1, 2, 3$ be the children of the first root note $r_1$.
Then, it follows from $\mu_1 = -2(n_{r_{12}} - n_{r_{11}})(n_{r_{12}} - n_{r_{13}})$
and \eqref{divisor}
that given $n_{r_{1k}} = m$ for some $k \in \{1, 2, 3\}$,
there are at most $o(|\mu_1|^{0+})$ many choices 
for $n_{r_{1j}}$, $j \ne k$ and hence for $n_r = n_{r_1}$.

When $J \geq 2$, \eqref{XT1} follows from an induction.
In obtaining the ordered bi-tree $\TT_J$,
we replaced one of the terminal nodes, say $b\in \TT^\infty_{J-1}$
into a non-terminal node.
In particular, $a\in \TT^\infty_J \setminus \TT^\infty_{J-1}$
must be  a child of $b$. 
Then, applying the argument for the $J = 1$ case, 
we see that for fixed $n_a = m \in \Z$
and $\mu_J \in \Z$, 
there are at most $o(|\mu_J|^{0+})$ many choices 
for $n_b$ (and the frequencies of the other two children of $b$).
Now that we have fixed $n_b$
(up to  $o(|\mu_J|^{0+})$ many choices),  \eqref{XT1} follows from 
the inductive hypothesis on $\TT_{J-1}$.
\end{proof}

\begin{remark}\rm
Note that Lemma \ref{LEM:XT1} also holds even if we replace any of $\mu_j$ by $\phi_j$.
\end{remark}

We continue with the proof of Lemma \ref{LEM:N^J+1_2}.
Before proceeding to the case  $b = r_2$, 
let us go over the main idea in the previous case ($ b \ne r_2$).
When $b \ne r_2$, 
we placed $\textsf{N}(\u)_{n_{b}}$ in the $\l^\infty_{n_{b}}$-norm
and by expressing the summation over $ {\bf n}  \in \mathfrak{N}(\TT_J)$ as 
\begin{align}
\sum_{{\bf n} \in \mathfrak{N}(\TT_J)}
=\sum_{n \in \Z} \sum_{\substack{{\bf n} \in \mathfrak{N}(\TT_J)\\n_r = n}} , 
\label{X4-}
\end{align}

\noi
we applied  Cauchy-Schwarz inequality (in particular, in $n_r = n$)
in the second inequality in 
\eqref{X3a}, thus creating the factor
\[\|\u_n\|_{\l^2_{n}}
\|\textsf{N}(\u)_{n_b}\|_{\l^\infty_{n_b}} \les 
 \|\u_n\|_{\l^2_{n}}\|\u\|_{H^\frac{1}{6}}^3\]

\noi
thanks to  \eqref{X2}.
This left $2J$ factors 
$\u_{n_a}$, $a \in \TT^\infty_J\setminus\{b, r_2\}$,
to which we applied \eqref{X3}.
In this argument, it was crucial that we have $b \ne r_2$
so that we have the factor $\u_n = \u_{n_{r_2}}$ for the application of Cauchy-Schwarz inequality in $n_{r_2} = n$.

When $b = r_2$, we no longer have the factor 
$\u_n = \u_{n_{r_2}}$.
Instead, the term corresponding to the frequency $n_{r_2}$ is given by 
$\textsf{N}(\u)_{n_{r_2}}$, which we place
in the $\l^\infty_{n_{r_2}}$-norm
as in the previous case.
Now, fix $\al \in \TT^\infty_J \setminus \TT^\infty_{J-1}$.  
Note that $\al \ne r_2$. 
Write
\begin{align}
\sum_{{\bf n} \in \mathfrak{N}(\TT_J)}
=\sum_{m \in \Z} \sum_{\substack{{\bf n} \in \mathfrak{N}(\TT_J)\\n_\al = m}} .
\label{X4}
\end{align}

\noi
Namely, we single out the frequency $n_\al= m$
at the terminal node $\al \in \TT^\infty_J \setminus \TT^\infty_{J-1}$.  
Compare this with \eqref{X4-} from the previous case, 
where we singled out 
 the frequency $n_{r} = n$
at the terminal node $r_2 \in \TT_J^\infty$.
In the following, 
we use $\u_m = \u_{n_\al}$
as a replacement of $\u_n = \u_{n_{r}}$ in the previous case
and apply Cauchy-Schwarz inequality in $n_\al = m$.
Also, note that, as a variant of  \eqref{X3}, we have
\begin{align}
 \sum_{m \in \Z} 
\sum_{\substack{{\bf n} \in \mathfrak{N}(\TT_J)\\ n_\al = m}} 
 \frac{1}{(n_{\max}^{(J)})^{1+}} 
 \prod_{a \in \TT^\infty_J\setminus\{r_2, \al\}} |\u_{n_a}|^2
& \les  
 \sum_{m\in \Z} 
\frac{1}{\jb{m}^{1+}}
\sum_{\substack{{\bf n} \in \mathfrak{N}(\TT_J)\\ n_\al = m}} 
\prod_{a \in \TT^\infty_J\setminus\{r_2, \al\}} |\u_{n_a}|^2 \notag\\
& \les \|\u\|_{L^2}^{4J}.
\label{X3b}
\end{align}

\noi
Indeed, \eqref{X3b} follows from 
first summing over 
 the  $2J$ variables $\{ n_a\}_{a \in \TT^\infty_J\setminus\{r_2, \al\}}$
 and then over $m \in \Z$ with 
$n_{\max}^{(J)} \geq |n_\al|$.
Then, from Cauchy-Schwarz inequality,
\eqref{mujj3}, \eqref{X2}, \eqref{X4-},  and \eqref{X4}, 
 we have
\begin{align}
\sum_{n \in \Z}|  \NN^{(J+1)}_2 & (\u)_n|
  \lesssim  
\|\textsf{N}(\u)_{n}\|_{\l^\infty_{n}} 
\sum_{\substack{\TT_J \in \mathfrak{BT}(J)\\\Pi_2(\TT_J) = \{r_2\}}}
\sum_{\substack{{\bf n} \in \mathfrak{N}(\TT_J)\\ |\phi_1| >K 
\\ |\wt{\phi}_j|\gg(2j+2)^3K^{1-\theta}\\ j = 2, \dots, J} }
 \frac{1}{\prod_{j = 1}^J |\wt{\phi}_j|}
\prod_{a \in \TT^\infty_J\setminus\{ r_2\}} |\u_{n_a}|
\notag \\
& \ll
\frac{   c_J}{\prod_{j = 1}^J(2j+2)^{\frac{3}{2}}}
K^{-\frac{J(1-\theta)}{2}}
\|\u\|_{H^\frac{1}{6}}^3
\notag \\
& \hphantom{XXXX} \times 
\sup_{\substack{\TT_J \in \mathfrak{BT}(J)\\\Pi_2(\TT_J) = \{r_2\}}}
\Bigg\{
\sum_{m\in \Z} 
\sum_{\substack{{\bf n} \in \mathfrak{N}(\TT_J)\\n_\al = m\\|\phi_1| >K 
\\ |\wt{\phi}_j|\gg(2j+2)^3K^{1-\theta}\\ j = 2, \dots, J} }
|\u_{n_\al}|^2\cdot (n_{\max}^{(J)})^{1+}
\prod_{j = 1}^J \frac{1}{|{\phi}_j|}
\Bigg\}^\frac{1}{2} 
\notag \\
& \hphantom{XXXX} \times 
\Bigg\{ \sum_{{\bf n} \in \mathfrak{N}(\TT_J)}
\frac{1}{(n_{\max}^{(J)})^{1+}} 
\prod_{a \in \TT^\infty_J\setminus\{r_2, \al\}} |\u_{n_a}|^2
\Bigg\}^\frac{1}{2}  \notag \\
\intertext{By \eqref{X3b}, \eqref{fphi},  and Lemma \ref{LEM:XT1},}
& \lesssim  
\frac{   C^J c_J}{\prod_{j = 1}^J(2j+2)^{\frac{3}{2}}}
K^{-\frac{J(1-\theta)}{2}}
 \|\u\|_{L^2}^{2J} \|\u\|_{H^\frac{1}{6}}^3
\notag \\
& \hphantom{XXXX} \times 
\sup_{\substack{\TT_J \in \mathfrak{BT}(J)\\\Pi_2(\TT_J) = \{r_2\}}}
\Bigg\{
\sum_{m\in \Z} |\u_{m}|^2
\sum_{\substack{\nu_j\in \Z\setminus\{ 0\}\\j = 1, \dots, J}}
\frac{1}{|\nu_J|^{\frac{3}{2}-}}
\prod_{j = 1}^{J-1} \frac{1}{|\nu_j|^{2-}}
\Bigg\}^\frac{1}{2} \notag \\
&\lesssim 
\frac{   C^J c_J}{\prod_{j = 1}^J(2j+2)^{\frac{3}{2}}}
K^{-\frac{J(1-\theta)}{2}}
 \|\u\|_{L^2}^{2J+1} \|\u\|_{H^\frac{1}{6}}^3
 \longrightarrow 0,
\label{X5}
\end{align}

\noi
as $J \to \infty$.

\smallskip

\noi
$\bullet$ {\bf Case 2:}
Next, we consider the case  $\Pi_2(\TT_J) \ne \{r_2\}$.
Note that we have $b \ne r_2$ by assumption.
In this case, we can proceed as in \eqref{X5} 
by replacing $r_2$ by $b$
and choosing $\al \in \TT^\infty_J \setminus (\TT^\infty_{J-1} \cup \{b\})$.
\end{proof}

\begin{remark}\rm 
(i) If we assume a higher regularity $\u \in H^\s(\T)$, $\s >\frac 12$, 
we can conclude~\eqref{X1a} simply by the algebra property of $H^\s(\T)$,
which suffices for our purpose.
See Subsection 4.4 in \cite{OSTz}.
We, however, decided to include the argument above
since this provides the sharp regularity criterion ($\s \geq \frac 16$)
for the vanishing of the error term.
Moreover, 
Lemma~\ref{LEM:XT1}
seems to be of independent interest.

\smallskip

\noi
{(ii)}
In view of the equation \eqref{4NLS2} with the cubic nonlinearity,  
we see $\s = \frac 16$ is the minimum regularity
required
for the application of the Leibniz rule
in \eqref{2-g},  \eqref{N^3}, and \eqref{N^J+1}. See~\cite{GKO}.
By a computation similar to that in this subsection,
we can also justify  the switching of the time derivatives and the summations
when $\s \geq \frac 16$
(by the dominated convergence theorem).
We  point out that it is also possible
to justify 
 the switching of the time derivatives and the summations
 as temporal distributions under a weaker assumption.
See Lemma 5.1 in \cite{GKO}.
\end{remark}

\subsection{Proof of Proposition \ref{PROP:main}}
\label{SUBSEC:8.5}

In this section, we put together all the estimates obtained in Subsection \ref{SUBSEC:8.2} - \ref{SUBSEC:error}
and prove Proposition \ref{PROP:main}.

Let $u$ be a smooth global solution
to the Wick ordered cubic 4NLS \eqref{4NLS1}
and $\u$ be its interaction representation as above.
Then, by applying the normal form reductions $J$ times, we obtain\footnote{Once again, we are replacing $\pm 1$
and  $\pm i$
by 1 for simplicity since they play no role in our analysis.}
\begin{align*} 
\dt | \u_n |^2 
=  \dt \sum_{j = 2}^{J+1} \NN^{(j)}_0(\u)_n
 + \sum_{j = 2}^{J+1} \RR^{(j)}(\u)_n
 + \sum_{j = 1}^{J+1} \NN^{(j)}_1(\u)_n 
 + \NN^{(J+1)}_{2} (\u)_n.
\end{align*}

\noi
In view of Lemma \ref{LEM:N^J+1_2}, 
by taking the limit as $J \to \infty$, we obtain
\begin{align*} 
\dt | \u_n |^2
=  \dt \sum_{j = 2}^\infty \NN^{(j)}_0(\u)_n
  + \sum_{j = 2}^\infty \RR^{(j)}(\u)_n
 + \sum_{j = 1}^\infty \NN^{(j)}_1(\u)_n.
\end{align*}

\noi
Then, integration over $[0, t]$
yields the identity \eqref{main0}
for smooth solutions.\footnote{With a slight abuse of notations, 
we are identifying
$\NN^{(j)}_0(\u)_n$ with $\NN^{(j)}_0(u)_n$, etc.
The same comment applies in the following.} 
Furthermore, the multilinear estimates \eqref{main1}, \eqref{main2}, and \eqref{main3}
follow from 
Lemmas \ref{LEM:N^1_1}, 
 \ref{LEM:N^2_0},  \ref{LEM:N^2_r}, \ref{LEM:N^2_1}
\ref{LEM:N^J+1_0}, \ref{LEM:N^J+1_r}, and \ref{LEM:N^J+1_1}
and choosing $\theta \in (0, \frac 23]$.\footnote{We fix 
an absolute constant $\theta \in (0, \frac 23]$ once and for all
and thus suppress
dependence of various constants on $\theta$ in the following.
}

In the following, 
we verify \eqref{main0} - \eqref{main3}
for rough solutions $u$ to \eqref{4NLS1}
by an approximation argument.
Fix $s \in( -\frac 13, 0)$.
Let $u$ be a (possibly non-unique) solution to \eqref{4NLS1}
in  $C([-T,T]; H^s(\T)) \cap F^{s,\alpha}(T) $, 
i.e.~with $M = 1$,
constructed in Section \ref{SEC:GWP}.
Note that we have  $T = T(\|u(0)\|_{H^s}) >0$.
See Remark \ref{REM:LWP}.

Let $u_N$ be a smooth solution to \eqref{4NLS1}
with $u_N|_{t = 0} = \P_{\le N}u(0)$.
Then, from the construction in Section \ref{SEC:GWP}, 
there exists a subsequence $\{u_{N_k}\}_{k\in\NB}$ 
such that
\begin{align}
\lim_{k \to \infty} \| u- u_{N_k}\|_{C_TH^s} = 0.
\label{conv1}
\end{align}

\noi
Moreover, 
$u_{N_k}$ and $u$ satisfy a uniform bound:
\begin{align}
\sup_{k \in \NB} \| u_{N_k}\|_{C_TH^s}, \| u\|_{C_TH^s}
\leq r \sim \|u(0)\|_{H^s}.
\label{conv1a}
\end{align}

\noi
Hence, 
it follows from \eqref{conv1} and  \eqref{conv1a} that, 
for each fixed $n \in \Z$, we have
\begin{align}
\lim_{k\to \infty}
\Big\{|\ft u_{N_k} (n, t)|^2 - |\ft  u_{N_k}  (n, 0)|^2  \Big\}
= |\ft u (n, t)|^2 - |\ft  u (n, 0)|^2,  
\label{conv12}
\end{align}

\noi uniformly in $t \in [-T, T]$.

Since $u_{N_k}$ is smooth, we have 
\begin{align}
\label{main0N}
|\ft u_{N_k} (n, t)|^2 - |\ft  u_{N_k}  (n, 0)|^2  
& = \sum_{j =2}^\infty \N_0^{(j)} (u_{N_k})(n, t') \bigg|_0^t \notag\\
& \hphantom{X}
+ \int_0^t \bigg[\sum_{j =1}^\infty \N_1^{(j)}(u_{N_k})(n, t') 
+ \sum_{j =2}^\infty \RR^{(j)} (u_{N_k})(n, t') \bigg] dt'.
\end{align}

\noi
Note that the identity
 \eqref{main0} for a rough solution $u$
 follows from 
  \eqref{conv12} and 
 \eqref{main0N} once we prove
\begin{align}
\lim_{k \to \infty} \sum_{j =2}^\infty \N_0^{(j)} (u_{N_k})(n, t) & =  \sum_{j =2}^\infty \N_0^{(j)} (u)(n, t) ,
\label{conv21}\\
\lim_{k \to \infty} \int_0^t \sum_{j =2}^\infty \RR^{(j)} (u_{N_k})(n, t') dt' & = \int_0^t  \sum_{j =2}^\infty \RR^{(j)} (u)(n, t') dt',
\label{conv22}\\
\lim_{k \to \infty} \int_0^t \sum_{j =1}^\infty \N_1^{(j)} (u_{N_k})(n, t') dt' 
& = \int_0^t  \sum_{j =1}^\infty \N_1^{(j)} (u)(n, t') dt',
\label{conv23}
\end{align}

\noi
 uniformly in $t \in [-T, T]$.
In the following, we only verify 
\eqref{conv21}, since \eqref{conv22} and \eqref{conv23} follow in an analogous manner.

From Lemmas \ref{LEM:N^2_0} and \ref{LEM:N^J+1_0} with 
 the multilinearity of 
$\N_0^{(j)}$
and \eqref{conv1a}, 
we can choose 
 $K = K(r) \gg1 $ such that 
\begin{align*}
 \bigg| \sum_{j = 2}^\infty \N_0^{(j)} & (u_{N_k})(n, t)  -  
\sum_{j =2}^\infty \N_0^{(j)} (u)(n, t) \bigg|\notag\\
& \les 
K^{\max(-\frac 12, -1- 2s)} r^{3} \|u-u_{N_k}\|_{C_TH^s}
+ \sum_{j = 3}^\infty 
 K^{-\frac{(j-1)(1-\theta)}2}  r^{2j+1} \|u-u_{N_k}\|_{C_TH^s}
 \notag\\
&  \les C(r) \|u-u_{N_k}\|_{C_TH^s}
\longrightarrow 0, 
\end{align*}

\noi
as $k \to \infty$,
 uniformly in $t \in [-T, T]$.
This proves
\eqref{conv21}.

Lastly, note that, in view of the global existence (Theorem \ref{THM:1})
and the Sobolev norm bound (Proposition \ref{PROP:GWP3}),
we can iterate the above discussion
to conclude the identity \eqref{main0} for all $t \in \R$. 
This completes the proof of Proposition \ref{PROP:main}.

\begin{remark}\label{REM:app}\rm
In the above argument, we assumed
that  $u \in C([-T,T]; H^s(\T)) \cap F^{s,\alpha}(T) $.
It is, however, sufficient to assume that 
 $u \in C([-T,T]; H^s(\T)) $ is a solution to \eqref{4NLS2}
for some $T = T(\|u_0\|_{H^s}) > 0$
with some smooth approximating solutions $\{u_n \}_{n \in \NB}$
such that  
\begin{align*}
\lim_{n \to \infty} \| u- u_{n}\|_{C_TH^s} = 0
\qquad \text{and}\qquad 
\sup_{n \in \NB} \| u_{n}\|_{C_TH^s}, \| u\|_{C_TH^s}
\les \|u(0)\|_{H^s}, 
\end{align*}
 
 \noi
 replacing \eqref{conv1} and \eqref{conv1a}.
The same comment applies to the proof of 
Proposition \ref{PROP:main2} presented in the next subsection.

\end{remark}

\subsection{Energy estimate for the non-resonant part}
\label{SUBSEC:8.6}

In this subsection, we briefly discuss
the proof of 
Proposition \ref{PROP:main2}
on the non-resonant part of the energy estimate \eqref{Diff3}
for  the difference of two solutions with the same initial condition.
 In fact, we reduce the matter
 to (a slight modification of)
 the discussion in the previous subsections.

Given $u_0 \in H^s(\T)$, $s > -\frac 1{3}$, 
let   $u$ and $v$ be two solutions to \eqref{4NLS1}
on $[-T, T]$
with the same initial condition $u|_{t = 0} = v|_{t = 0} =  u_0$, satisfying
\begin{align*}
 \| u\|_{C_TH^s},  \| v\|_{C_TH^s}
\leq r \sim \|u_0\|_{H^s}.
\end{align*}

\noi
Let $\u$ and $\v$ denotes the interaction representations of $u$ and $v$, respectively.
Then, 
from \eqref{Diff3}, we have
\begin{align}
\I  & = -  2\Re i \sum_{n\in \Z} \langle n \rangle^{2s} 
\big[\ft{\textsf{N}(\u)}_n -  \ft{\textsf{N}(\v)}_n\big]
 \cj{(\ft \u_n - \ft \v_n)} \notag \\
&  =  - 2\Re i \sum_{n\in \Z} \langle n \rangle^{2s} 
\ft{\textsf{N}(\u)}_n\cj{\ft \u_n} 
+ 2\Re i \sum_{n\in \Z} \langle n \rangle^{2s}  \ft{\textsf{N}(\u)}_n \cj{\ft \v_n} \notag\\
& \hphantom{X}+ 2\Re i \sum_{n \in \Z} \langle n \rangle^{2s} \ft{\textsf{N}(\v)}_n \cj{\ft \u_n} 
- 2\Re i \sum_{n\in \Z} \langle n \rangle^{2s}  \ft{\textsf{N}(\v)}_n \cj{\ft \v_n} \notag\\
& =:   \I_{uu} - \I_{uv}- \I_{vu}+ \I_{vv}.
\label{Y1}
\end{align}

\noi
From \eqref{1-g} with \eqref{4NLS2}, we have 
\begin{align}
\I_{uu} =  \sum_{n\in \Z} \jb{n}^{2s} \N^{(1)} (\u)_n
\qquad \text{and}\qquad 
\I_{vv} =  \sum_{n\in \Z} \jb{n}^{2s} \N^{(1)} (\v)_n.
\label{Y2}
\end{align}

\noi
By repeating the arguments in the previous subsections, we have\footnote{In Subsections \ref{SUBSEC:8.2}
and \ref{SUBSEC:8.3}, we performed the normal  form reductions
for each fixed $n\in \Z$ and the weight $\jb{n}^{2s}$ in \eqref{Y2}
does not affect the argument.}
\begin{align}
\int_0^t \I_{uu}(t') dt'
&  = \sum_{j=2}^\infty \sum_{n\in \Z}\jb{n}^{2s}\NN_0^{(j)} (\u)(n, t')\bigg|_0^t \notag\\
& \hphantom{X}
+ 
 \int_0^t  \bigg[  \sum_{j = 2}^\infty\sum_{n\in \Z} \jb{n}^{2s}\RR^{(j)}( \u)(n, t') 
+  \sum_{j = 1}^\infty \sum_{n\in \Z}\jb{n}^{2s} \NN_1^{(j)} (\u)(n, t')
 \bigg]dt'
\label{Y3}
\end{align}

\noi
and
\begin{align}
\int_0^t \I_{vv}(t') dt'
&  = \sum_{j=2}^\infty \sum_{n\in \Z}\jb{n}^{2s}\NN_0^{(j)} (\v)(n, t')\bigg|_0^t \notag\\
& \hphantom{X}+ 
 \int_0^t  \bigg[  \sum_{j = 2}^\infty \sum_{n\in \Z}\jb{n}^{2s}\RR^{(j)}( \v)(n, t') 
 + \sum_{j = 1}^\infty \sum_{n\in \Z} \jb{n}^{2s}\NN_1^{(j)} (\v)(n, t')\bigg]dt'.
\label{Y4}
\end{align}

In order to handle the cross terms $\I_{uv}$ and $\I_{vu}$, we need to introduce new notations.
Define
$\wt \N^{(1)} (\u, \v)_n$ by 
\begin{align*}
\wt \NN^{(1)}(\u, \v )_n
= - 2 \Re i 
\sum_{\TT_1 \in \mathfrak{BT}(1)}
\sum_{\substack{{\bf n} \in \mathfrak{N}(\TT_1)\\ n_r = n} }
e^{  - i  \phi_1 t } \bigg( \prod_{a \in \Pi_1(\TT_1)^\infty} \u_{n_{a}}\bigg)
\bigg(  \prod_{b \in \Pi_2(\TT_1)^\infty} \v_{n_{b}}\bigg).
\end{align*}

\noi
Namely, 
$\wt \N^{(1)} (\u, \v)_n$ is constructed from $ \N^{(1)}_n$  in \eqref{1-g}
by taking different functions $\u$ and $\v$
over the terminal nodes of the first tree $\Pi_1(\TT_1)$
and the second tree $\Pi_2(\TT_1)$,\footnote{Note that  
the second tree $\Pi_2(\TT_1)$ consists only of the (second) root node $r_2$.
We, however, use this notation in order to be consistent with the general case.
See \eqref{Y6}, \eqref{Y7}, and \eqref{Y8}. }
respectively.

\noi
Then, we have 
\begin{align*}
\I_{uv} =  \sum_{n\in \Z} \jb{n}^{2s} \wt \N^{(1)} (\u, \v)_n
\qquad \text{and}\qquad 
\I_{vu} =  \sum_{n\in \Z} \jb{n}^{2s} \wt \N^{(1)} (\v, \u)_n.
\end{align*}

\noi
We also make similar modifications to the multilinear terms introduced 
in Subsections \ref{SUBSEC:8.2} and \ref{SUBSEC:8.3}
and define $\wt  \NN^{(j)}_0 (\u, \v)_n$, $\wt  \RR^{(j)} (\u, \v)_n$, 
and $\wt  \NN^{(j)} (\u, \v)_n$ by\footnote{As mentioned
in Remark \ref{REM:2Re}, 
we are dropping unimportant $\pm$, $\pm i$, and $2\Re$.  }
\begin{align} 
\wt  \NN^{(j)}_0 (\u, \v)_n
 &:  = \sum_{\TT_{j-1} \in \mathfrak{BT}(j-1)}
\sum_{\substack{{\bf n} \in \mathfrak{N}(\TT_{j-1})\\n_r = n}}
\frac{ e^{- i \wt{\phi}_{j-1} t } }{\prod_{k = 1}^{j-1} \wt{\phi}_{k}}
  \notag\\
& \hphantom{XXXXXX}
\times  \bigg(\prod_{a \in \Pi_1(\TT_{j-1})^\infty} \u_{n_{a}}\bigg)
 \bigg(\prod_{b \in \Pi_2(\TT_{j-1})^\infty} \v_{n_{b}}\bigg), 
\label{Y6}
\\
\wt \RR^{(j)}(\u, \v)_n
& : =  \sum_{\TT_{j-1} \in \mathfrak{BT}(j-1)}
\sum_{\al \in\TT^\infty_{j-1}} 
\sum_{\substack{{\bf n} \in \mathfrak{N}(\TT_{j-1})\\n_r = n}}
\frac{ e^{- i \wt{\phi}_{j-1}t } }{\prod_{k = 1}^{j-1} \wt{\phi}_k}
\, \textsf{R}(\u)_{n_\al}
\label{Y7}
  \notag\\
& \hphantom{XXXXXX}
\times  \bigg(\prod_{a \in \Pi_1(\TT_{j-1})^\infty\setminus \{\al\}} \u_{n_{a}}\bigg)
 \bigg(\prod_{b \in \Pi_2(\TT_{j-1})^\infty\setminus \{\al\}} \v_{n_{b}}\bigg), \\
 \wt \NN^{(j)}(\u, \v)_n
&  : = 
 \sum_{\TT_{j} \in \mathfrak{BT}(j)}
\sum_{\substack{{\bf n} \in \mathfrak{N}(\TT_{j})\\n_r = n}}
\frac{ e^{- i \wt{\phi}_{j}t } }{\prod_{k = 1}^{j-1} \wt{\phi}_k}
\, \bigg(\prod_{a \in \Pi_1(\TT_{j})^\infty} \u_{n_{a}}\bigg)
 \bigg(\prod_{b \in \Pi_2(\TT_{j})^\infty} \v_{n_{b}}\bigg).
\label{Y8}
\end{align}

\noi
Compare these definitions with \eqref{N^J+1}.
Moreover, we define
 $\wt  \NN^{(j)}_1 (\u, \v)_n$ and  $\wt  \NN^{(j)}_2 (\u, \v)_n$
 as the restrictions
 of  $\wt  \NN^{(j)} (\u, \v)_n$
onto $C_{j-1}$ and $C_{j-1}^c$ (see \eqref{CJ}).
Then, from the discussion in the previous subsections, we have
\begin{align}
\int_0^t \I_{uv}(t') dt'
&  = \sum_{j=2}^\infty \sum_{n\in \Z}\jb{n}^{2s}\wt \NN_0^{(j)} (\u, \v)(n, t')\bigg|_0^t \notag\\
& \hphantom{XX}+ 
 \int_0^t  \bigg[ \sum_{j = 2}^\infty\sum_{n\in \Z}\jb{n}^{2s} \wt \RR^{(j)}( \u, \v)(n, t') 
 \notag\\
 & \hphantom{XXXXXXX}+ 
 \sum_{j = 1}^\infty\sum_{n\in \Z} \jb{n}^{2s}\wt \NN_1^{(j)} (\u, \v)(n, t')\bigg]dt'
\label{Y9}
\end{align}

\noi
and
\begin{align}
\int_0^t \I_{vu}(t') dt'
&  = \sum_{j=2}^\infty \sum_{n\in \Z}\jb{n}^{2s} \wt \NN_0^{(j)} (\v, \u)(n, t')\bigg|_0^t \notag\\
& \hphantom{XXX}+ 
 \int_0^t  \bigg[  \sum_{j = 2}^\infty\sum_{n\in \Z}\jb{n}^{2s} \wt  \RR^{(j)}( \v, \u)(n, t') 
 \notag\\
 & \hphantom{XXXXXXX}
 + \sum_{j = 1}^\infty\sum_{n\in \Z} \jb{n}^{2s} \wt \NN_1^{(j)} (\v, \u)(n, t')
 \bigg]dt'.
\label{Y10}
\end{align}

In the following, we simply drop the factor $\jb{n}^{2s}$.\footnote{By making use of
the factor $\jb{n}^{2s}$, we may extend Proposition \ref{PROP:main2}
to $s > -\frac 12$. This, however, involves modifications
of the multilinear estimates in Subsections \ref{SUBSEC:8.2}
and \ref{SUBSEC:8.3}.
In view of the regularity restriction $s > -\frac 13$ 
for the resonant part (Proposition \ref{PROP:main}), 
we simply use the multilinear estimates from 
Subsections~\ref{SUBSEC:8.2}
and~\ref{SUBSEC:8.3}
and prove Proposition \ref{PROP:main2}
for $s > -\frac 13$.
}
Then, by applying  the multilinear estimates in 
Lemmas \ref{LEM:N^1_1}, 
 \ref{LEM:N^2_0},  \ref{LEM:N^2_r}, \ref{LEM:N^2_1}
\ref{LEM:N^J+1_0}, \ref{LEM:N^J+1_r}, and \ref{LEM:N^J+1_1}
(with  $\theta \in (0, \frac 23]$), 
it follows from \eqref{Y1}, \eqref{Y3}, \eqref{Y4}, \eqref{Y9}, and \eqref{Y10} that\footnote{By writing 
the double differences of the multilinear terms of the $j$th generation in a telescoping sum, 
we obtain $O(j^2)$ many terms.  This loss of $O(j^2)$ does not cause any issue in view
of the fast decay in $j$ in the multilinear estimates
in Subsection \ref{SUBSEC:8.3}.} 
\begin{align}
\bigg| \int_0^t \I(t') dt'\bigg| 
 & =  \bigg| \int_0^t \Big\{\big( \I_{uu}(t') - \I_{uv}(t')\big) - \big(\I_{vu}(t')- \I_{vv}(t')\big)\Big\}
dt'\bigg| \label{Y11a}\\
& \les
K^{\max(-\frac 12, -1- 2s)} r^{2} \|u-v\|_{C_TH^s}^2
+ \sum_{j = 3}^\infty 
 K^{-\frac{(j-1)(1-\theta)}2}  r^{2j-2} \|u-v\|_{C_TH^s}^2\notag\\
&  \hphantom{XX}
+ 
TK^{\max(-\frac 12, -1- 3s)} r^{4} \|u-v\|_{C_TH^s}^2
+ T \sum_{j = 3}^\infty 
 K^{-\frac{(j-3)(1-\theta)}2}  r^{2j} \|u-v\|_{C_TH^s}^2\notag\\
&  \hphantom{XX}
 + TK^{\frac{1}{2}-2s} r^{2} \|u-v\|_{C_TH^s}^2
+ T \sum_{j = 2}^\infty 
 K^{-\frac{(j-2)(1-\theta)}2}  r^{2j} \|u-v\|_{C_TH^s}^2, \notag 
\end{align}

\noi
uniformly for $t \in [-T, T]$.
Note that we obtained two factors of $u - v$
thanks to the double difference structure of \eqref{Y1}.
Then, by first choosing $K = K(r) > 0$ sufficiently large
and then choosing $T = T(K) = T(r) > 0$ sufficiently small, 
we conclude that 
\begin{align*}
\bigg| \int_0^t \I(t') dt'\bigg|  & \leq \frac{1}{4} \|u-v\|_{C_TH^s}^2
\end{align*}

\noi
for  $t \in [-T, T]$.
This completes the proof of Proposition \ref{PROP:main2}.

\begin{remark}\label{REM:uniq3}\rm
Integrating \eqref{Diff3} from $0$ to $t$, 
we obtain the identity:
\begin{align}
 \| u (t)- v(t) \|_{H^s}^2
  & =   \int_0^t \Big\{( \I_{uu}(t') - \I_{uv}(t')) - (\I_{vu}(t')- \I_{vv}(t'))\Big\}
dt'\notag\\
& \hphantom{X}
+ \int_0^t    \sum_{n\in \Z} \jb{n}^{2s} 
\big(\mathfrak{S}_\infty(u)(n, t') -
\mathfrak{S}_\infty(v)(n, t')
\big)\cj{(\ft u_n - \ft v_n)} \ft v_n(t') dt', 
\label{Y11}
\end{align}

\noi
where $\mathfrak{S}_\infty(u)(n, t)  = |\ft u_n(t)|^2 - |\ft u_n(0) |^2$
as in \eqref{intromain1}.
In view of \eqref{main0}, \eqref{Y3}, \eqref{Y4},
\eqref{Y9}, and \eqref{Y10}, 
we see that both the first and second terms
on the right-hand side of \eqref{Y11}
can be expressed as
a sum of infinite series consisting of multilinear terms
of increasing degrees.
Furthermore,
thanks to the multilinearity of the summands, 
we can extract two factors of (the Fourier coefficient of) $u-v$ in both the first and second terms
on the right-hand side of \eqref{Y11}.

In this paper, we established {\it spatial} multilinear estimates
(for fixed $t$)
and showed that these multilinear terms are summable, provided $s > -\frac 13$.
This allows us to obtain the enhanced uniqueness in Theorem \ref{THM:2}.
It may be of interest to establish {\it space-time} estimates on  these multilinear terms
(arising from  the energy of the difference of two solutions in the $E^s(T)$-norm rather than the $C_TH^s$-norm), 
namely 
in terms of the $F^{s, \al}(T)$-norm as in Section~\ref{SEC:energy},
possibly allowing us to go below $s = -\frac 13$.
We point out that the argument in 
 \cite{CKSTT04} may be of use in 
 estimating multilinear terms of (arbitrarily) large degrees.

\end{remark}

\begin{remark}\label{REM:uniq2}\rm

If we proceed with an energy estimate
in the spirit of 
Proposition \ref{PROP:energy} in 
Section \ref{SEC:energy}
(but without symmetrization)
in terms of the $F^{s, \al}(T)$-norm, 
we can establish the following energy bound;
let $s \in ( -\frac{3}{10}, 0)$ and  $\al = -\frac{8s}3 + \eps$
as in \eqref{alpha}.
Then, there exists $\theta > 0$ such that 
\begin{align}
\bigg|\int_0^T
\I(t) 
dt \bigg|
\les T^\theta \Big( \| u\|_{F^{s, \al}(T)}^2+ \| v\|_{F^{s, \al}(T)}^2\Big)
\| u - v\|_{F^{s, \al}(T)}^2
\label{energy2a}
\end{align}

\noi
for any $T \in (0, 1]$.	
By combining with
the linear and nonlinear estimates
(Lemma \ref{LEM:linear} and Proposition \ref{PROP:3lin}),
the energy estimate
\eqref{energy2a} yields uniqueness for $s > -\frac{3}{10}$,
not sufficient for Theorem \ref{THM:2}.

\end{remark}

\begin{ackno}\rm
T.O. and Y.W.~were supported by the European Research Council (grant no.~637995 ``ProbDynDispEq'').
The authors would like to thank the anonymous referee
for the helpful comments which improved the presentation of the paper.

\end{ackno}

\end{document}